\documentclass[11pt,a4paper]{amsart} 
\usepackage[foot]{amsaddr}

\usepackage{microtype}
\usepackage{indentfirst}
\usepackage{amsmath, amsthm, amssymb, mathtools}
\usepackage{tikz}
\usepackage{tikzscale}
\usetikzlibrary{arrows.meta,calc,positioning,decorations.pathmorphing,patterns}
\usepackage{xcolor}
\usepackage{xspace}
\usepackage{enumitem}
\usepackage{graphicx}

\usepackage{hyperref} 
\hypersetup{
  colorlinks=true,
  citecolor=black!50!red,
  linkcolor=black!50!red,
  linktoc=all
}
\usepackage[all]{hypcap} 
\usepackage[capitalise]{cleveref}
\usepackage[font=footnotesize,labelfont=bf]{caption}
\usepackage{subcaption}
\reversemarginpar

\newcounter{constant}
\newcommand{\newconstant}[1]{\refstepcounter{constant}\label{#1}}
\newcommand{\uc}[1]{c_{\textnormal{\tiny \ref{#1}}}}
\newcounter{bigconstant}
\newcommand{\newbigconstant}[1]{\refstepcounter{bigconstant}\label{#1}}
\newcommand{\ubc}[1]{C_{\textnormal{\tiny \ref{#1}}}}
\newtheorem{theorem}{Theorem}[section]
\newtheorem{proposition}[theorem]{Proposition}
\newtheorem{lemma}[theorem]{Lemma}

\newtheorem{corollary}[theorem]{Corollary}

\numberwithin{equation}{section} 

\theoremstyle{definition}
\newtheorem{definition}[theorem]{Definition}

\newtheorem{remark}[theorem]{Remark}

\newtheorem{problem}[theorem]{Problem}

\AddToHook{env/lemma/begin}{\crefalias{theorem}{lemma}}
\AddToHook{env/corollary/begin}{\crefalias{theorem}{corollary}}
\AddToHook{env/proposition/begin}{\crefalias{theorem}{proposition}}
\AddToHook{env/definition/begin}{\crefalias{theorem}{definition}}
\AddToHook{env/remark/begin}{\crefalias{theorem}{remark}}
\AddToHook{env/problem/begin}{\crefalias{theorem}{problem}}

\newcommand{\Eheal}{\mathcal{E}_{\mathrm{heal}}}
\newcommand{\Eout}{\mathcal{E}_{\mathrm{out}}}
\newcommand{\Ecluster}{\mathcal{E}_{\mathrm{cluster}}}
\newcommand{\Ein}{\mathcal{E}_{\mathrm{disp}}}

\renewcommand{\P}{\mathbb{P}}

\newcommand{\E}{\mathbb{E}}
\newcommand{\R}{\mathbb{R}}

\newcommand{\N}{\mathbb{N}}
\newcommand{\Z}{\mathbb{Z}}
\renewcommand{\L}{\mathcal{L}}

\newcommand{\charf}[1]{\mathbf{1}_{#1}}

\newcommand{\A}{\mathcal{A}}
\newcommand{\B}{\mathcal{B}}

\newcommand{\T}{\mathcal{T}}
\renewcommand{\S}{\mathcal{S}}

\newcommand{\dd}{\mathop{}\!d}
\DeclareMathOperator{\dist}{d}
\DeclareMathOperator{\diam}{diam}

\DeclareMathOperator{\poisson}{Poisson}
\DeclareMathOperator{\ber}{Bernoulli}
\DeclareMathOperator{\bin}{Binomial}

\DeclareMathOperator{\var}{Var}
\DeclareMathOperator{\cov}{Cov}

\DeclareMathOperator{\In}{In}
\DeclareMathOperator{\Out}{Out}

\DeclarePairedDelimiter{\abs}{\lvert}{\rvert}
\DeclarePairedDelimiter{\norm}{\lVert}{\rVert}
\DeclarePairedDelimiter\ceil{\lceil}{\rceil}
\DeclarePairedDelimiter\floor{\lfloor}{\rfloor}
\DeclarePairedDelimiter{\parens}{\lparen}{\rparen}
\DeclarePairedDelimiter{\braces}{\lbrace}{\rbrace}

\DeclarePairedDelimiterX{\inner}[1]{\langle}{\rangle}{\inneraux#1}

\newcommand{\normI}[1]{\norm{#1}_{\infty}}

\begin{document}

\title{Epidemic Phase Transitions in the Zero-Range Process}

\date{July 24, 2026}
\author{Rangel Baldasso$^1$}
\address[1]{Department of Mathematics, PUC-Rio, Rua Marquês de São Vicente 225, Gávea, 22451-900 Rio de Janeiro, RJ - Brazil.}
\email{rangel@puc-rio.br}
\author{Marcelo Hilário$^2$}
\email{mhilario@mat.ufmg.br}
\author{Ian Ornelas$^2$}
\email{ian.orn@gmail.com}
\address[2]{Department of Mathematics, Universidade Federal de Minas Gerais, Av.\ Antonio Carlos 6627, 31270-901 Belo Horizonte, MG - Brazil.}

\begin{abstract}
  We consider a model for the spread of an infection within an interacting particle system on $\mathbb{Z}^d$, generalizing a framework introduced by Kesten and Sidoravicius to a zero-range process in equilibrium with density $\rho > 0$.
  In our model, at any time, individuals are either healthy or infected.
  The infection spreads instantaneously whenever infected and healthy particles occupy the same site, while infected particles heal independently at rate $\delta > 0$.
  We investigate the extinction-survival phase diagram of this process starting from a configuration where only the particles at the origin are infected.
  For every fixed positive healing rate, we prove that the infection becomes extinct almost surely if the density $\rho$ is sufficiently small and survives with positive probability if $\rho$ is sufficiently large, establishing the existence of a non-trivial critical density.
  At sufficiently high densities, survival occurs even under instantaneous healing.
  We also show that, for every positive density, the infection survives when the healing rate is sufficiently small. 
\end{abstract}



\maketitle

\section{Introduction}
\label{sec:intro}

In an impressive series of works~\cite{KestenSidoravicius05, KestenSidoravicius06, KestenSidoravicius08}, Kesten and Sidoravicius studied epidemic processes propagating within a moving population modeled as a Poisson system of independent random walks (IRW) on $\mathbb{Z}^d$.
In the absence of recovery, they first established linear bounds on the speed of propagation in~\cite{KestenSidoravicius05} and ultimately proved a full shape theorem in~\cite{KestenSidoravicius08}.
The inclusion of a healing mechanism in~\cite{KestenSidoravicius06} led to an epidemic model akin to an SIS contact process in a moving population, for which they proved a phase transition between extinction (immunity) and survival (contagion) phases as the healing rate varies.

In the present work, we advance the study in this framework by considering a moving population subjected to zero-range interactions.
Individuals, or particles, initially distributed across the sites of the $d$-dimensional integer lattice $\mathbb{Z}^d$, move according to a zero-range process (ZRP) in equilibrium at a given density $\rho > 0$.
At any time $t \geq 0$, each particle is either healthy (susceptible) or infected.
The infection mechanism is instantaneous: whenever a healthy and an infected particle occupy the same lattice site, the healthy particle immediately becomes infected.
Furthermore, infected particles recover independently at a constant rate $\delta > 0$.
However, recovery can only occur if an infected particle is isolated at a site; if it shares a site with other infected particles, any recovery attempt is overridden due to instantaneous reinfection.

We start with all particles healthy, except for those at the origin (if any), which are infected.
We say that the infection dies out (or becomes extinct) almost surely (a.s.) if, with probability one, there exists a time after which every particle is healthy.
If, at all times, there exists at least one infected particle, we say the infection survives (globally).
We prove the existence of a phase transition for global survival as the density varies at any fixed healing rate.
We also prove survival at every positive density when the healing rate is sufficiently small.
For large enough values of the density, survival holds for every healing rate, including instantaneous healing.

Before stating our main results precisely, let us detail the nature of the environment under consideration.
The zero-range process in $\mathbb{Z}^d$, first introduced by Spitzer in~\cite{Spitzer70}, is a continuous-time Markov process with state space $\mathbb{N}_0^{\mathbb{Z}^d}$ evolving  as follows. 
If at time $t$ a site $x$ is occupied by $n$ particles, then at rate $g(n)$ one particle moves out of $x$ toward a site chosen uniformly at random among its $2d$ nearest neighbors.
Here, we assume that the rate function $g : \mathbb{N}_0 \to \mathbb{R}_{+}$ is kept fixed, and that it satisfies $g(0) = 0$ and
\begin{equation}
    \Gamma_{-} \leq g(n)-g(n-1) \leq \Gamma_{+}, \qquad n \in \mathbb{N}
\end{equation}
for constants $0 < \Gamma_{-} \leq 1 \leq \Gamma_{+}$.
Under these assumptions, the dynamics are well-posed, and the process has a one-parameter family of spatial product invariant measures $\mu_{\rho}$, where $\rho > 0$ denotes the density of particles, as proved in~\cite{Andjel82} (see also~\cite{KipnisLandim99}).
In \cref{sec:zr_invariant_measures}, we provide further properties of the zero-range process that will be useful throughout the paper.

We now state our main results.
Our first theorem establishes that the infection dies out if the particle density is sufficiently low.
\begin{theorem}
\label{thm:immunity_low_rho}
  For any fixed recovery rate $\delta > 0$, there exists $\bar\rho = \bar\rho(\delta) > 0$ such that, for every $\rho \in (0, \bar\rho]$, the infection initialized at the origin and evolving on top of the zero-range process in equilibrium with density $\rho$ and recovery rate $\delta$ dies out almost surely.
\end{theorem}

Our second result concerns the high-density regime.
We show that above a certain threshold density $\rho_{\star} < \infty$, the infection survives with positive probability regardless of the healing rate $\delta$.
Indeed, survival holds even if $\delta=\infty$ (meaning that particles heal instantaneously when sitting alone at a site).
\begin{theorem}
\label{thm:contagion_high_rho}
  There exists $\rho_{\star} < \infty$ such that for every $\rho \ge \rho_{\star}$, the infection initialized at the origin and evolving on top of the zero-range process in equilibrium with density $\rho$ and recovery rate $\delta \in [0,\infty]$ survives with positive probability.
\end{theorem}
\cref{thm:contagion_high_rho} generalizes~\cite[Theorem 1.4]{BaldassoStauffer21} where the authors prove the analogous result for the infection on top of the IRW environment relying on the Lipschitz surface techniques developed in~\cite{GracarStauffer19_2}.
In the present work we instead employ an oriented percolation approach as developed in \cref{sec:survival}.

By the monotonicity of the survival probability as a function of the density $\rho$, \cref{thm:contagion_high_rho,thm:immunity_low_rho}  establish a phase transition for fixed $\delta$ at a density threshold $\rho_c(\delta)$, as stated next.
\begin{theorem}
\label{thm:phase_transition_rho}
  For each fixed $\delta \in (0,\infty]$, there exists a critical density threshold $\rho_c(\delta) \in (0,\infty)$ such that: for every $\rho \in (0,\rho_c(\delta))$, the infection initialized at the origin and evolving on top of the zero-range process in equilibrium with density $\rho$ dies out a.s., whereas it survives with positive probability for every $\rho > \rho_c(\delta)$.
\end{theorem}

Note that the phase transition takes place even for $\delta=\infty$: the critical threshold $\rho_c(\infty)$ belongs to the interval $(0,\rho_{\star}]$.

We also investigate survival as the healing rate $\delta$ varies.
By monotonicity, the infection survives for every $\delta \in [0,\infty]$ if $\rho > \rho_c(\infty)$. Furthermore, at every positive density, survival holds provided that the healing rate is sufficiently low:
\begin{theorem}
\label{thm:contagion_low_delta}
  For every $\rho >0$, there exists $\delta_{\star}(\rho) \in (0,\infty)$ such that, for every $\delta\le\delta_{\star}(\rho)$, the infection initialized at the origin and evolving on top of the zero-range process in equilibrium with density $\rho$ and recovery rate $\delta$ survives with positive probability.
\end{theorem}

Let us define
\begin{equation}
    \begin{aligned}
    \rho_c(\infty-)
    &=\inf\big\{\rho>0:\text{the infection survives for every }\delta<\infty\big\}\\
    &=\sup_{\delta<\infty}\rho_c(\delta).
    \end{aligned}
\end{equation}
\cref{thm:immunity_low_rho,thm:contagion_high_rho} imply that $\rho_c(\infty-)$ is bounded away from $0$ and $\infty$, while \cref{thm:contagion_low_delta} together with monotonicity implies that for every density $\rho \in (0, \rho_c(\infty-))$ the model undergoes a non-trivial phase transition at a well-defined critical healing threshold $\delta_c(\rho) \in (0,\infty)$ as we synthesize in the next theorem.

\begin{theorem}
\label{thm:phase_transition_delta}
 For each fixed $\rho \in (0,\rho_c(\infty-))$, there exists a critical healing threshold $\delta_c(\rho) \in (0,\infty)$ such that: for every $\delta \in (0,\delta_c(\rho))$, the infection initialized at the origin and evolving on top of the zero-range process in equilibrium with density $\rho$ survives with positive probability, whereas it dies out a.s.\ for every $\delta > \delta_c(\rho)$.
\end{theorem}

By monotonicity $\rho_c(\infty-)\leq\rho_c(\infty)$, and we conjecture that these two quantities coincide, as in \cref{fig:phase_diagram_left}.
Equivalently, we believe that a non-degenerate  threshold $\delta_c(\rho)$ should exist for every $\rho$ in the interval $(0,\rho_c(\infty))$ beyond which $\delta_c(\rho)$ becomes infinite.

The other possibility, namely that the strict inequality $\rho_c(\infty-)<\rho_c(\infty)$ holds, would imply that for every $\rho \in (\rho_c(\infty-), \rho_c(\infty))$ the infection would die out for $\delta=\infty$ while it would survive as soon as $\delta<\infty$, as illustrated in \cref{fig:phase_diagram_right}.

In addition, we believe that $\rho_c(\delta) < \rho_c(\infty-)$ for every $\delta\in(0,\infty)$, or equivalently, that $\lim_{\rho\to\rho_c(\infty-)^-} \delta_c(\rho) = \infty$ as in \cref{fig:phase_diagram_left}.
The other possibility in which $\rho_c(\delta)$ hits $\rho_c(\infty-)$ for some $\delta < \infty$ is also depicted in \cref{fig:phase_diagram_right}.

\begin{figure}[tpb]
  \centering
  \begin{subfigure}[b]{0.49\textwidth}
    \centering
    \includegraphics[width=\linewidth]{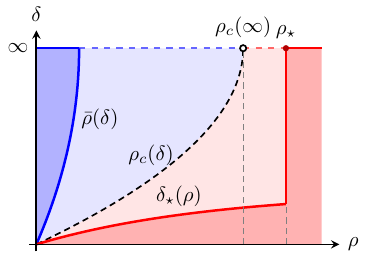}
    \caption{The case \(\rho_c(\infty-)=\rho_c(\infty)\).}
    \label{fig:phase_diagram_left}
  \end{subfigure}
  \hfill
  \begin{subfigure}[b]{0.49\textwidth}
    \centering
    \includegraphics[width=\linewidth]{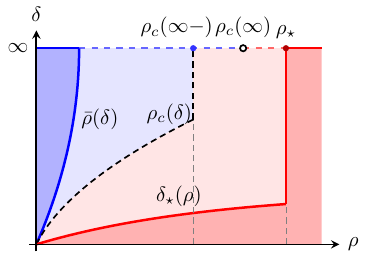}
    \caption{The case \(\rho_c(\infty-)<\rho_c(\infty)\).}
    \label{fig:phase_diagram_right}
  \end{subfigure}
  \caption{We illustrate two possible phase diagrams compatible with our results.
  Blue and red indicate, respectively, the extinction and survival phases. 
  Darker shades represent the extinction and survival regions established by our theorems; the lighter shades illustrate possible extrapolations up to the critical curve $\rho_c(\delta)$.
  We expect that the actual phase diagram should resemble the one illustrated on the left panel, with a continuous and strictly increasing critical curve.
  In the right-hand side panel, densities lying strictly between $\rho_c(\infty-)$ and $\rho_c(\infty)$ yield survival for all finite healing rates ($\delta<\infty$) but extinction at instantaneous healing ($\delta = \infty$).}
  \label{fig:phase_diagram}
\end{figure}  

Transitioning from a system of independent random walks as in~\cite{KestenSidoravicius06} to the zero-range process poses major challenges.
Because jump rates depend on local site occupancy, environment particles dynamically obstruct or accelerate one another, destroying the independent path decompositions and complicating the use of techniques that are central to the analysis of the IRW, such as Poisson thinning.

Moreover, beyond the change from the IRW to the ZRP medium, our setting differs from the one in~\cite{KestenSidoravicius06} in its microscopic transmission rule: there, the infection can only be transmitted at the exact moment an environment particle jumps.
This allows for particles to heal even when they share a site with other infected particles.
Under this rule, the authors establish a non-trivial phase transition in the healing rate for every density (see~\cite[Theorem 1]{KestenSidoravicius06} therein).
This cannot hold for our model, due to \cref{thm:contagion_high_rho}, so it is natural that \cref{thm:phase_transition_delta} holds for a bounded interval of densities.

Although the change in the infection mechanism may appear to be a minor detail, as noted in their work, this technical restriction is critical for their proof of extinction at low healing rates, and a careful inspection reveals that it is indispensable for their renormalization procedure.
Indeed, in~\cite[page 3]{KestenSidoravicius06} the authors address the difficulty of adopting the alternative mechanism by stating: ``\emph{This creates some sort of singularity in the model which we are unable to handle at the moment}.''
We employ a multi-scale approach capable of handling the setting where transmission occurs instantaneously whenever an infected and a healthy particle coincide, bypassing this difficulty when the density belongs to $(0,\rho_c(\infty-))$, with the additional advantage of considering a more general environment.
Moreover, it is worth mentioning that our techniques could also be applied to their setting, allowing one to extend~\cite[Theorem 1]{KestenSidoravicius06} to the infection model on the ZRP.

Another novelty in our work consists of enhanced versions of decoupling inequalities for the zero-range process.
They are fundamental tools for the multi-scale renormalization scheme used to prove the above results, as we discuss further in \cref{sec:sketch}.
While certain versions have become available recently in the literature~\cite{BaldassoTeixeira20, ABHS25}, they fall short of being applicable to our setting.
We expect that our new versions can be applied to the study of other models where the ZRP plays the role of a dynamic random environment (such as random walks and detection problems).

Before we state our decoupling results, first we establish some notation regarding space-time boxes. 
Given two space-time boxes $B^1$ and $B^2$, their horizontal and vertical distances are defined as
\begin{equation}
\label{eq:distances}
  \begin{split}
    \dist_V(B^1, B^2) & = \inf \Big\{  |s-t|:
      \begin{array}{c} (x,s) \in B^1 \text{ and } (y,t) \in B^2, \\ \text{ for some } x,y \in \Z^d
    \end{array} \Big\} \\
    \dist_{H}(B^1, B^2) & = \inf \Big\{  \normI{x-y}:
      \begin{array}{c} (x,s) \in B^1 \text{ and } (y,t) \in B^2, \\ \text{ for some } s,t \in \R
    \end{array} \Big\}.
  \end{split}
\end{equation}

Let $\eta = \left(\eta_t(x) \colon (x,t) \in \mathbb{Z}^d \times \mathbb{R}_+\right)$, where each $\eta_t(x)$ belongs to $\mathbb{N}_0$, denote a trajectory.
Assume $t \mapsto \eta_t \in \mathbb{N}_0^{\mathbb{Z}^d}$ is càdlàg.
For a fixed rate function $g$, we denote by $\mathbb{P}_\rho$ the probability measure on the canonical Skorokhod space under which $\eta$ is distributed as a zero-range process initialized according to the density-$\rho$ invariant product measure $\mu_\rho$.
We write $\mathbb{E}_{\rho}$ and $\cov_{\rho}$ to denote the corresponding expectation and covariance with respect to $\mathbb{P}_\rho$.

Following~\cite[Equation (1.5)]{BaldassoTeixeira20}, we formalize the notion of support as follows:
A function $f: \mathbb{N}_0^{\mathbb{Z}^d \times \mathbb{R}_+} \to \mathbb{R}$ is said to have
\emph{support} in a space-time region $B \subset \mathbb{Z}^d \times \mathbb{R}_+$ if 
\begin{equation}
    \label{eq:support}
\text{for all $\eta, \bar\eta$, $f(\eta) = f(\bar\eta)$ whenever $\eta_t(x) = \bar\eta_t(x)$ for all $(x, t) \in B$.}
\end{equation}

We begin with the horizontal decoupling inequality, which provides a covariance bound for functions of the zero-range process supported on space-time boxes that are well-separated spatially. Intuitively, this ensures that events occurring in regions with a sufficient spatial separation evolve almost independently.

\newconstant{c:horiz_decoupling}
\newbigconstant{C:horiz_decoupling}
\begin{theorem}[Horizontal decoupling]\label{thm:horizontal_decoupling} There are constants
  $\ubc{C:horiz_decoupling}, \uc{c:horiz_decoupling}>0$ which depend only on $\Gamma_+$ and the dimension, such that the following holds.

  Let $B^1, B^2 \subset \Z^d \times [0, T]$ be space-time boxes with spatial projections $S_1, S_2$
  satisfying $\dist_H = \dist_H(B^1,B^2) \ge 8$ and $\diam(S_j) \leq w$ for $j = 1, 2$ and some positive constant $w$.
  Let $f_1, f_2: \N_0^{\Z^d \times \R_+} \to [0, 1]$ be measurable functions with supports in $B^1$ and $B^2$, respectively. Then, for all $\rho > 0$, $T \ge 1$,
  \begin{equation}\label{eq:horiz_decoupling}
    \abs{\cov_\rho(f_1(\eta), f_2(\eta))} \le \ubc{C:horiz_decoupling} w^{d - 1} \rho T^{d/2}
    e^{-\uc{c:horiz_decoupling} \dist_H \log(1+\dist_H/T)}.
  \end{equation}
\end{theorem}

For the one-dimensional setting, in the particular case where $B^1$ and $B^2$ are space-time squares of side length $s$ and their horizontal separation satisfies $\dist_H \ge s^\alpha$ for some $\alpha > \frac{1}{2}$, the bound \eqref{eq:horiz_decoupling} simplifies to
  \begin{equation}
  \label{eq:horizontal_decoup_1d}
    \abs{\cov_\rho(f_1(\eta), f_2(\eta))} \le \ubc{C:horiz_decoupling}\rho\sqrt{s}
    e^{-\uc{c:horiz_decoupling} s^\alpha \log(1+s^{\alpha-1})}.
  \end{equation}
  The polynomial factor $\sqrt{s}$ can be absorbed into the exponential by adjusting the positive constant $\uc{c:horiz_decoupling}$.
  This yields an upper bound in terms of $s$ that is  stretched-exponential, exponential or super-exponential depending on the value of $\alpha$.

  Let us compare \eqref{eq:horizontal_decoup_1d} to the horizontal decoupling for the one-dimensional zero-range process derived in~\cite[Proposition 1.6]{BaldassoTeixeira20}.
  There the upper bound on the covariance is only super-polynomial.
  Moreover, the horizontal separation between the boxes needs to grow linearly with their sizes.
  In addition, they have to restrict their density $\rho$ to range over a previously fixed compact interval $[\rho_-, \rho_+]$, and the constants corresponding to $\ubc{C:horiz_decoupling}$ and $\uc{c:horiz_decoupling}$ (called respectively $c_3$ and $c_3^{-1}$ there) depend implicitly on $\rho_-$ and $\rho_+$.

  We now turn to the vertical decoupling inequality. 
Unlike its horizontal counterpart, the vertical decoupling handles temporal separation and is not framed as a direct covariance bound. Instead, it requires a sprinkling argument, which entails changing the reference measure by slightly perturbing the underlying particle density $\rho$.
Here we state it for increasing functions, but an analogous result can also be obtained for decreasing ones. We say a function $f$ of the trajectories is increasing if, for any two trajectories $\xi$ and $\eta$,
\begin{equation}
    \xi \preceq \eta \text{ implies } f(\xi) \leq f(\eta),
\end{equation}
where $\xi \preceq \eta$ means that $\xi_{t}(x) \leq \eta_{t}(x)$, for all $(x, t) \in \Z^{d} \times \R$. Analogously, $f$ is decreasing if $-f$ is increasing.

\newbigconstant{C:vertical_decoupling}
\newbigconstant{C:V}
\newconstant{c:vertical_decoupling}
\begin{theorem}[Vertical decoupling with sprinkling]
  \label{thm:vertical_decoupling}
  There exist positive constants $\ubc{C:vertical_decoupling}$, $\uc{c:vertical_decoupling}$ and $\ubc{C:V}$ depending only on the dimension and $\Gamma_+$ and $\Gamma_-$, such that, for any two space-time boxes $B^1,B^2$ of side length $w>0$ with vertical distance $\dist_V = \dist_V(B^1, B^2) \ge \ubc{C:V}$, any two non-decreasing bounded measurable functions $f_1, f_2: \N_0^{\Z^d \times \R_+} \to [0, 1]$ with supports in $B^1$ and $B^2$, any $\rho >0$, and any sprinkling increment $\varepsilon \in (0,1)$, we have
  \begin{equation}
  \label{eq:vertical_decoupling_bound}
    \begin{split}
    \E_\rho & [f_1(\eta) f_2(\eta)] - \E_{\rho(1 + \varepsilon)}[f_1(\eta)] \E_{\rho(1 + \varepsilon)}[f_2(\eta)] \\
    &\quad \leq \ubc{C:vertical_decoupling} (w + \dist_V)^{d + 1} \Big(\rho \exp\braces{-\uc{c:vertical_decoupling} \dist_V^{1/3}} + \exp\braces{-\uc{c:vertical_decoupling} \rho\varepsilon^2 \dist_V^{1/3}}\Big).
    \end{split}
  \end{equation}
\end{theorem}

A vertical decoupling for the zero-range process was obtained in~\cite[Theorem 1.3]{BaldassoTeixeira20}. 
The main difference between their result and \cref{thm:vertical_decoupling} lies in the uniformity of the constants $\ubc{C:vertical_decoupling}$ and $\uc{c:vertical_decoupling}$ (which correspond, respectively, to the constants $C_1$ and $c_3$ therein) on the density parameter $\rho$.
In fact, as made explicit in their statement, there these constants depend implicitly on a previously fixed upper bound $\rho_{+}$ for the set of possible densities $\rho$.
In contrast, we provide a new version where $\ubc{C:vertical_decoupling}$ and $\uc{c:vertical_decoupling}$ can be chosen uniformly over $\rho_{+}$.
This is in fact crucial for us as it allows for the analysis of high-density regimes (for instance in \cref{thm:contagion_high_rho}), by enabling us to drive the density to sufficiently high values without losing control on the term appearing in the right-hand side in \eqref{eq:vertical_decoupling_bound}.

\subsection{Motivation and Background}
\label{sec:motivation}

This study is partially inspired by the papers of Kesten and Sidoravicius~\cite{KestenSidoravicius05, KestenSidoravicius06, KestenSidoravicius08} on the  spread of an infection (or a rumor) when the underlying population evolves as an independent Poisson system of continuous-time simple symmetric random walks.
In the absence of a recovery mechanism, the spread of the infection behaves as a stochastic growth model reminiscent of the  Richardson model.
In~\cite{KestenSidoravicius08}, an asymptotic shape theorem was established under the assumption that both infected ($A$) and susceptible ($B$) particles share the same diffusion rate ($D_A = D_B$).
When infected and susceptible particles jump at different positive rates ($D_A \neq D_B$), a linear growth rate was established by Dauvergne and Sly~\cite{DauvergneSly23}.
Furthermore, the case when the underlying environment consists of biased random walks was investigated by Baldasso and Stauffer~\cite{BaldassoStauffer20}.
In the case where particles perform independent random walks on random conductances in $\mathbb{Z}^d$, Gracar and Stauffer~\cite{GracarStauffer19} proved a linear lower bound for the displacement of the front.
Linear upper and lower bounds for the position of the front of the infection when the IRW environment is replaced by the ZRP in one dimension were established by Baldasso and Teixeira~\cite{BaldassoTeixeira20}.

In addition to the expansion of the front, in the presence of recovery, the interest shifts to the study of the survival-extinction phase transition.
For the IRW environment, this was first treated by Kesten and Sidoravicius in~\cite{KestenSidoravicius06}, and subsequently 
by Gracar and Stauffer~\cite{GracarStauffer19}, and by Baldasso and Stauffer~\cite{BaldassoStauffer21}.
When the graph is the Sierpinski gasket, Drewitz, Gallo, and Gracar~\cite{DGG26} proved that the infection survives if the recovery rate is sufficiently low.
A variation involving the removal of recovered particles (the SIR model) was studied by Dauvergne and Sly~\cite{DauvergneSly22}, and a similar setup where particles move via independent Brownian motions with removal, was developed by Grimmett and Li~\cite{GrimmettLi22}.
To the best of our knowledge, the present manuscript is the first to study the survival-extinction phase transition for the ZRP environment instead of IRW.

Broadly, the aforementioned infection models can be viewed as members of a family of processes featuring the $A+B \to 2A$ transition rule whose appearance in the mathematics literature dates back at least to~\cite{BCDFLS86}.
It encompasses a variety of systems including the so-called frog-model (or combustion process), the multi-particle diffusion-limited aggregation (mDLA), and the activated random-walks.
Within this framework, the physical interpretation often changes according to the particular variations in the underlying particle transport mechanism, and the microscopic rules of evolution (e.g.\ the relative diffusion rates assigned to each particle type).
However, each one of these processes can be regarded essentially as a model for the spread of an infection or rumor; the $A$ and $B$ particles representing infected and susceptible (or healthy), respectively.
In most of these, the underlying dynamics does not simply act as a passive background. 
Rather, it serves simultaneously as a mechanism for displacing the infection and as a source of spatial inhomogeneities to which the epidemic can be highly sensitive.
Together with the complicated space-time correlations that arise from the motion of the particles, this not only results in a remarkably rich spectrum of behaviors, but also makes this class of processes challenging to study.
Given the vastness of the literature on this topic, the exposition that we present next is necessarily non-exhaustive.
For alternative perspectives and a more comprehensive list of references, we refer the reader to the introductory sections of the works cited herein.

Most of the works so far were confined to the case when the particles of the medium move as independent random walks diffusing with rates $D_A$ and $D_B$ depending on their $A/B$ state. 
As already mentioned, when $D_A > 0$ and $D_B > 0$, the model represents the spread of an infection or rumor through a mobile population, as analyzed in~\cite{KestenSidoravicius05, KestenSidoravicius08, DauvergneSly23}.
If one instead freezes the target particles by setting $D_B = 0$, the process becomes the frog model~\cite{AMP02} or combustion process~\cite{BerardRamirez10, CQR07, RamirezSidoravicius02, RamirezSidoravicius04}.
Reversing this regime by setting $D_A = 0$ yields multi-particle diffusion-limited aggregation (mDLA) in which the  environment consists of mobile particles that attach to a growing aggregate at the moment they attempt to jump onto its boundary~\cite{KestenSidoravicius08a, Sly21} (see also~\cite{SidoraviciusStauffer19, Voss84} where, instead of IRW, particles move as a symmetric simple exclusion process).

The introduction of recovery as in~\cite{KestenSidoravicius06, BaldassoStauffer21} corresponds to the addition of an extra transition of the type $A\to B$.
This has also been considered in the case when $D_B=0$, leading to the model known as activated random walks. 
 Rolla and Sidoravicius~\cite{RollaSidoravicius12} showed that this model undergoes an absorbing-state phase transition. 
 Much of the interest in this model stems from the fact that it features self-organized criticality (see also~\cite{Rolla20} for a survey, and  Hoffman, Johnson, and Junge~\cite{HJJ24} and Forien~\cite{Forien25} for recent developments).
 
Also related to our interest are models of contact processes on dynamic random environments.
One recent example is the interchange-and-contact process recently introduced by Hilário, Ungaretti, Valesin, and Vares~\cite{HUVV25}. In that model, sites can be vacant, healthy, or infected, with transmission occurring across neighbors and the particle population evolving dynamically according to an underlying interchange process at rate $\nu$. 
Apart from the fact that the environment is conservative, transmissions in their model do not occur instantaneously upon contact as they do in our setup. 
In the interchange-and-contact process, the transmission rate $\lambda$ is treated as a free parameter alongside the density $\rho$ and the interchange rate $\nu$, with the goal being to characterize the critical transmission threshold in the rapid-interchange limit ($\nu \to \infty$). 
Conversely, in the present paper, we fix the intrinsic evolution rate of the ZRP and focus on the interplay between the healing rate and the background density, obtaining a global characterization of the corresponding phase diagram.

\subsection{Sketch of the proofs}
\label{sec:sketch}

To establish the extinction regime of \cref{thm:immunity_low_rho} and the survival regimes of \cref{thm:contagion_high_rho,thm:contagion_low_delta}, we rely on multi-scale renormalization arguments.
In broad terms, we show that the occurrence of a specific ``bad event'' in a large space-time region entails the occurrence of similar bad events in two well-separated regions at a smaller scale.
These bad events are suitably defined depending on whether the objective is to prove extinction or survival. 
For the extinction phase, a bad event corresponds to the infection crossing a large space-time box---namely, that when initialized along the boundary, the infection manages to span the box either in the temporal direction or along one of the spatial coordinates. 
Conversely, to prove survival, we map the spread of the infection to a multi-scale oriented percolation model, defining a bad event as the failure of the infection to traverse a given space-time block.

In both cases, by applying the appropriate decoupling bounds, we are able to establish a recursive inequality of the form $p_{k + 1} \leq C p_k^2$, where $p_k$ denotes the probability of a bad event occurring at scale $k$.
Once such recurrence is in force, verifying that the initial probability $p_0$ is sufficiently small guarantees that $p_k$ decays fast, and therefore the bad event will not occur for sufficiently large scales, establishing our results.

At this stage, the problem reduces to bounding the probability $p_0$ to initiate the renormalization cascade, a step commonly referred to as \emph{triggering}.
This is accomplished by tuning the parameters of the model according to the regime under consideration.

To establish \cref{thm:immunity_low_rho}, we choose the parameters so that the infection is likely to die out locally before spanning the base-scale box.

\emph{\cref{thm:immunity_low_rho}:} 
By taking the density $\rho$ to be sufficiently small, we show that particles at the base-scale box spend most of their time in small clusters that do not interact with one another. 
Assuming this, we prove that the infection dies out exponentially fast inside such clusters and therefore no crossings occur.

To establish survival (\cref{thm:contagion_high_rho,thm:contagion_low_delta}), we must ensure that an initial infection spreads across the base-scale block with high probability:

\emph{\cref{thm:contagion_high_rho}:} 
By taking the density $\rho$ to be large, we can assume that all sites in the base-scale box contain at least two particles at all times.
Because particles can only heal when alone on a site, this crowding prevents recoveries and the infection easily crosses the appropriate region.

\emph{\cref{thm:contagion_low_delta}:} 
By choosing a large enough base-scale length and duration, we can ensure that starting from an initial infected region the infection will successfully spread with high probability, provided that no healing marks occur. 
This gives that the probability of crossing a box at the base-scale is sufficiently high when $\delta=0$.
However, since the length and height of the box are fixed, a finite-size criterion argument allows us to slightly raise $\delta$ to a positive value still keeping the spread probability sufficiently high.

To obtain the recursive inequalities we heavily rely on space-time decoupling estimates, whose proofs follow distinct strategies depending on the geometry of the separation:

\emph{\cref{thm:horizontal_decoupling}:} 
The horizontal decoupling is obtained by introducing a priority-based graphical representation, where particles starting in different regions are assigned distinct priorities.
We show that, with high probability, particles from well-separated spatial regions never interact.
On this non-interaction event, the trajectories in these regions are entirely independent, allowing us to bound the covariance of locally supported functions.

\emph{\cref{thm:vertical_decoupling}:} 
For the vertical decoupling, we employ a coupling strategy between two processes initialized with slightly different densities.
By applying an iterative matching procedure in space-time blocks, we show that uncoupled particles eventually meet and couple due to the diffusive properties of the zero-range random walks.
This establishes domination of the lower-density process by the slightly higher-density one with high probability after a sufficient time separation.

Finally, let us remark that, while the proofs here are carried out for the environment given by the zero-range process, the main requirement for the environment is that one has access to correlation decay estimates like the ones from \cref{thm:horizontal_decoupling,thm:vertical_decoupling}. In particular, one can use the techniques developed here to analyze infection processes on top of other conservative particle systems like the exclusion process.

\subsection*{About constants}
Throughout the text, $C, C', \dots$ and $c, c', \dots$ are used to denote generic constants appearing during the proofs, whose value may change from one proof to another.
Numbered constants such as $C_1, C_2, \dots$ and $c_1, c_2, \ldots$ have their value fixed at their first appearance, and their values remain fixed throughout the text.
All the constants may depend on the dimension and on fixed parameters of the model like $\Gamma_+$ and $\Gamma_-$.
For some of these constants we may indicate the dependence on extra parameters of the model, for instance, we may write $C_j=C_j(\rho)$ to indicate that $C_j$ is a constant whose value depends on the parameter $\rho$ (and possibly on the dimension and fixed parameters of the model).

\subsection*{Acknowledgments}
The research of RB was partially supported by CNPq grants ``Projeto Universal'' (402952/2023-5), (408529/2025-3), ``Produtividade em Pesquisa'' (308018/2022-2) and ``Jovem Cientista do Nosso Estado'' (204.276/2025) from FAPERJ.
MH was partially supported by CNPq grant ``Projeto Universal'' (401314/2025-1), ``Produtividade em Pesquisa'' (312566/2023-9) and FAPEMIG grant ``Demanda Universal'' (APQ-01214-21).
IO thanks FAPEMIG and PPGMAT UFMG.
IO also thanks PUC Rio for its hospitality during a visit where part of this collaboration took place. 
The authors thank Weberson Arcanjo, Alexander Glazman, Renato dos Santos and Alexandre Stauffer for helpful discussions.

\section{The zero-range and the infection processes}
\label{sec:zrp}

The main goal of this section is to provide the mathematical construction of the infection process introduced in \cref{sec:intro}.
This is done in \cref{sec:representation_infection}.
The rest of the section is divided as follows.
In \cref{sec:zr_invariant_measures} we present the standard construction of the zero-range process based on its infinitesimal generator, and characterize the invariant measures.
\cref{sec:zr_slot_representation} is dedicated to presenting a particular graphical representation for the zero-range process that will be convenient for our purposes.
\cref{sec:zr_bounds} collects useful deviation bounds for the invariant measures of the zero-range process and for its site occupation.

\subsection{Zero-range process and invariant measures}
\label{sec:zr_invariant_measures}

The zero-range process (ZRP) is a continuous-time Markov process $(\eta_t)_{t \geq 0}$ taking values in the configuration space $\mathbb{N}_0^{\mathbb{Z}^d}$.
For an element $\eta = (\eta(x))_{x \in \mathbb{Z}^d}$, the coordinate $\eta(x)$ is interpreted as the number of particles at site $x$.
The dynamics of the process are governed by the Markov generator $L$, which acts on local functions $f: \N_0^{\Z^d} \to \R$ as
\begin{equation}\label{eq:zrp_gen}
  Lf(\eta) = \sum_{x \in \Z^d} g(\eta(x)) \sum_{y: |y-x|=1} \frac{1}{2d} \bigl[ f(\eta^{x, y}) - f(\eta) \bigr],
\end{equation}
where $\eta^{x, y} \in \mathbb{N}_0^{\mathbb{Z}^d}$ is the configuration obtained from $\eta$ by displacing a single particle from site $x$ to site $y$:
\begin{equation}
  \eta^{x, y}(z) =
  \begin{cases}
    \eta(x) - 1 & \text{if } z = x, \\
    \eta(y) + 1 & \text{if } z = y, \\
    \eta(z) & \text{otherwise}.
  \end{cases}
\end{equation}

As already mentioned, we will always assume that the jump rate function $g:\N_0\to\R_+$ satisfies $g(0)=0$ and
\begin{equation}
\label{eq:g_cond}
  \Gamma_- \le g(k) - g(k - 1) \le \Gamma_+, \qquad k \in \N,
\end{equation}
for constants $0 < \Gamma_- \le 1 \le \Gamma_+$.
This condition on the increments of $g$ ensures well-posedness of the process  (see~\cite{Andjel82}). 
Furthermore, for every parameter $\phi > 0$, there exists an invariant product measure $\nu_\phi$ with i.i.d.\ marginal distributions given by
\begin{equation}
  \label{eq:nu_phi}
  \nu_\phi ( \eta(0) = n) = \frac{1}{Z(\phi)} \frac{\phi^n}{g(n)!},
\end{equation}
where $g(0)! = 1$, $g(n)! = g(1) \cdots g(n)$, and $Z(\phi) = \sum_{k=0}^{\infty} \frac{\phi^k}{g(k)!}$ is the partition function, which is finite for all $\phi > 0$ under \eqref{eq:g_cond}.
We refer the reader to~\cite{Andjel82} for more information on this family of invariant measures.

The mean density of particles under $\nu_\phi$ is given by the strictly increasing bijection $R(\phi) = \E_{\nu_\phi}[\eta(0)]$. 
It is often more convenient to parameterize the invariant measures by their density $\rho > 0$ rather than by $\phi$. 
Thus, we define $\mu_\rho = \nu_{R^{-1}(\rho)}$ to be the invariant product measure with density $\rho$.

\subsection{The slot representation for the zero-range process}
\label{sec:zr_slot_representation}

Alongside the infinitesimal generator, important tools to analyze an interacting particle system are the graphical representations. 
They provide a way to sample the system through a space-time percolation-type structure based on the arrivals of Poisson point processes.
We present a graphical representation for the zero-range process that introduces an infinite collection of \emph{slots} at each site. 
While equivalent to the representation found in~\cite[Section 2.1]{BaldassoTeixeira20}, this approach makes explicit the ordering of particles at every site.
In \cref{sec:slot_priority} we modify this construction to obtain the \emph{slot-priority} representation which will allow us to derive good decoupling estimates for the zero-range process.

Let us start by defining the notion of slots.
\begin{definition}[Slot configuration]
  A \emph{slot} is a pair of coordinates $(x,i) \in \Z^d \times \N$. The sequence of slots of the form $\{(x,i)\}_{i \geq 1}$ represents a stack of positions built over site $x$.
  A \emph{slot configuration} is an element
  \begin{equation}
    \sigma = (\sigma(x, i))_{x \in \Z^d, i \in \N} \in \{0, 1\}^{\Z^d \times \N},
  \end{equation}
  assigning to each slot either $0$ (vacant) or $1$ (occupied).
  Denote by $n(\sigma, x) = \sum_{i=1}^\infty \sigma(x, i)$ the number of particles placed in slots over $x$.

  A configuration $\sigma$ is called \emph{packed} if, for every $x \in \Z^d$, the occupied slots are exactly those at the bottom of the stack, that is,
  \begin{equation}
    \sigma(x, i) = \textbf{1}_{ \{ i \le n(\sigma, x) \} } = 
    \begin{cases}
      1, & i \le n(\sigma, x),\\
      0, & i > n(\sigma, x).
    \end{cases}
  \end{equation}
  Therefore, when $\sigma$ is packed, $n(\sigma, x)$ represents the \emph{height} of the stack of particles at $x$.
  Denote by $\Sigma_{\text{packed}} \subset \{0, 1\}^{\Z^d \times \N}$ the state space of all
  packed slot configurations. Given a configuration $\eta \in \N_0^{\Z^d}$, its \emph{canonical slot embedding} is the unique packed slot configuration $\sigma$ with $n(\sigma, x)=\eta(x)$, for every $x\in \Z^d$.
\end{definition}

We now move to the graphical representation of the zero-range process using slots. For any $\sigma \in \Sigma_{\text{packed}}$, slot $(x,i)$ with $1 \le i \le n(\sigma, x)$ and a target site $y \in \Z^d$, let $\sigma^{(x, i) \to y} \in \Sigma_{\text{packed}}$ denote the
configuration resulting from moving the particle at $(x,i)$ to the top of the stack at site $y$.
That is, the departing particle is removed from slot $(x,i)$, and placed at $(y, n(\sigma, y) + 1)$. The particles at $x$ which occupy slots $(x,j)$ for $j > i$ are shifted down to $(x,j - 1)$. All other particles remain in place.
We now move to the definition of the dynamics using slots.

\begin{definition}[Slot dynamics]
  \label{def:slot_dynamics}
  Assign independently to each slot $(x,i)$ and each nearest neighbor $y$ with $|y-x|=1$ a Poisson clock $\mathcal{P}_{x,y}^i$ on $\R_+$ with rate
  \begin{equation}
    \lambda_i = \frac{1}{2d}\bigl(g(i) - g(i - 1)\bigr).
  \end{equation}
  The \emph{slot dynamics} is defined as follows. 
  Starting from a packed slot configuration $\sigma_0$, whenever the Poisson clock $\mathcal{P}_{x,y}^i$ rings at time $t$:
  \begin{itemize}
    \item If $i \le n(\sigma_{t-}, x)$, the particle at $(x,i)$ jumps to the target site $y$. The system transitions to the new state $\sigma_t = (\sigma_{t-})^{(x, i) \to y}$.
    \item If $i > n(\sigma_{t-}, x)$, the slot $(x,i)$ is empty. 
    The clock ring is suppressed and nothing happens ($\sigma_t = \sigma_{t-}$).
  \end{itemize}
\end{definition}

The infinitesimal generator $L^{\text{slot}}$ of the slot representation acts on local functions $f: \Sigma_{\text{packed}} \to \R$ as
\begin{equation}\label{eq:ig_slot}
    L^{\text{slot}}f(\sigma) = \sum_{x \in \Z^d} \sum_{i = 1}^{n(\sigma, x)} \lambda_i \sum_{y\sim x} \bigl[ f(\sigma^{(x,i) \to  y}) - f(\sigma) \bigr].
\end{equation}
Assume now $\eta_0 \in \N_0^{\Z^d}$ is an initial configuration drawn from $\mu_\rho$, and $(\sigma_t)_{t\ge 0}$ is the slot dynamics of \cref{def:slot_dynamics} started from the canonical slot embedding of $\eta_0$. 
We now prove that the process $\eta_t(x) = n(\sigma_t, x)$ is the nearest-neighbor zero-range process with rate function $g$ and density $\rho$.

\begin{proposition}[Equivalence with the zero-range process]
\label{prop:slot_is_zrp} 
Let $\eta_0\in\N_0^{\Z^d}$ be an initial configuration drawn from $\mu_\rho$ for some $\rho>0$, and let $(\sigma_t)_{t\ge 0}$ be the slot dynamics (as in \cref{def:slot_dynamics}) started from the canonical slot embedding of $\eta_0$.
Define $\eta_t(x) = n(\sigma_t, x)$ for each $x \in \Z^d$ and $t \ge 0$.
Then $(\eta_t)_{t \ge 0}$ has the law of the zero-range process on $\Z^d$ with rate function $g$ and initial condition $\eta_0$.
\end{proposition}

\begin{proof}
  Consider the projection $\pi: \Sigma_{\text{packed}} \to \N_0^{\Z^d}$ defined by $\pi(\sigma)(x) =
  n(\sigma, x)$. For any local function $f: \N_0^{\Z^d} \to \R$, a straightforward computation gives
  $L^{\text{slot}}(f \circ \pi) = (Lf) \circ \pi$, where $L$ is the zero-range generator. Indeed,
  since $\pi(\sigma^{(x, i) \to y}) = \pi(\sigma)^{x, y}$ for any $i \in \{1, \ldots, n(\sigma,
  x)\}$, we can sum the rates over $i$ to obtain
  \begin{equation}
      \begin{split}        
    L^{\text{slot}} (f\circ\pi)(\sigma)
    &= \sum_{x \in \Z^d} \sum_{i = 1}^{n(\sigma, x)} \lambda_i \sum_{y\sim x} \bigl[ f(\pi(\sigma)^{x, y}) - f(\pi(\sigma)) \bigr] \\
    &= \sum_{x \in \Z^d} g(n(\sigma, x)) \sum_{y\sim x} \frac{1}{2d} \bigl[ f(\pi(\sigma)^{x, y}) - f(\pi(\sigma)) \bigr] \\
    &= (L f)(\pi(\sigma)).
    \end{split}
  \end{equation}
  This relation implies that the projected process $\eta_t = \pi(\sigma_t)$ is a Markov process with
  generator $L$. Since the initial configuration is matched by construction, $(\eta_t)$ is the
  zero-range process.
\end{proof}

\begin{remark}
  While the algebraic relation $L^{\text{slot}}(f \circ \pi) = (Lf) \circ \pi$ holds for all finite
  configurations, proving that the respective infinite-volume processes share the same law requires
  the corresponding martingale problems to be uniquely well-posed. For the zero-range process,
  well-posedness on the infinite lattice holds whenever the initial conditions do not grow too fast as the spatial coordinates go to infinity, see~\cite{Andjel82}. 
  Since we draw $\eta_0$ from the invariant product
  measure $\mu_\rho$, the initial configuration $\eta_0$ is almost surely inside the regime of
  uniqueness defined therein.
\end{remark}

We denote by $\P^{\eta_0}$ the law of the slot representation started from the canonical slot embedding of a configuration $\eta_0 \in \N_0^{\Z^d}$, and $\E^{\eta_0}$ the corresponding expectation operator. When the initial configuration is drawn from the invariant measure $\mu_\rho$, we denote the resulting law by $\P_\rho$ and its expectation by $\E_\rho$.

\subsection{Healing marks and the infection process}
\label{sec:representation_infection}

We augment the slot representation of the zero-range process by adding an independent Poisson point process $\mathcal{R}_x$ of rate $\delta \geq 0$ to every site $x\in\Z^d$.
The arrival marks of $\mathcal{R}_x$ are called the \emph{healing marks} at $x$.
The infection process is defined using this augmented graphical representation through the so-called infection paths, as we define next.

Following the terminology in~\cite{KestenSidoravicius05} we start with the definition of genealogical paths and $J$-paths.
See \cref{fig:jump-paths}.

\begin{figure}[tpb]
  \centering
  \begin{subfigure}[t]{0.48\textwidth}
    \centering
    \includegraphics[width=\linewidth]{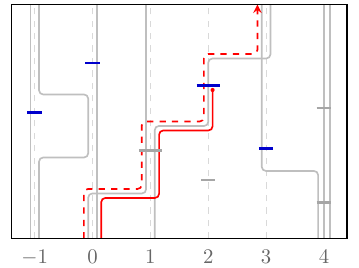}
    \caption{A $J$-path is drawn as a dashed red line along with its longest corresponding genealogical path in solid red.}
    \label{fig:jump-paths}
  \end{subfigure}\hfill
  \begin{subfigure}[t]{0.48\textwidth}
    \centering
    \includegraphics[width=\linewidth]{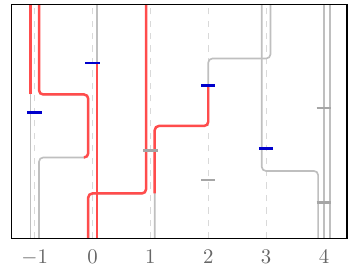}
    \caption{The infected particles for the infection process started at the origin.}
    \label{fig:infection}
  \end{subfigure}
  \caption{Graphical representation for the zero-range process with particles depicted as solid gray lines, effective healing marks as blue horizontal lines and ineffective healing marks as gray horizontal lines.}
\end{figure}

\begin{definition}[Genealogical paths]
  \label{def:jump_path}
  Given a healing rate $\delta \in [0, \infty]$, a \emph{genealogical path} is a càdlàg function
  $\gamma:[s_{0},s_{1}] \to \Z^d$ ($0 \leq s_{0} \leq s_{1}$) such that:
  \begin{enumerate}[label=(\roman*)]
    \item whenever $\gamma(t)\neq\gamma(t-)$ for $t \in [s_{0}, s_{1}]$, a particle at position $\gamma(t-)$ jumps to $\gamma(t)$ at
      time $t$;
    \item $\eta_t(\gamma(t)) \ge 1$ for all $t \in [s_{0}, s_{1}]$;
    \item for every $t \in [s_0,s_1]$ such that $\eta_{t}(\gamma(t)) = 1$, $t$ is not an arrival of the Poisson process $\mathcal{R}_{\gamma(t)}$.
  \end{enumerate}
\end{definition}

\begin{remark}
\label{rmk:healing_suppression}
  If a healing mark coincides with a site occupied by exactly one particle, we call it an \emph{effective healing mark}. Condition (iii) states that the path must avoid all effective healing marks. This captures the requirement that healings are only effective at isolated particles, since any healing attempt in the presence of other particles is ignored due to instantaneous reinfection. In particular, when $\delta = \infty$, the set of healing marks is $\R_+ \times \Z^d$. Thus, the infection can only be sustained if the path strictly avoids isolated particles, resulting in the requirement that $\eta_t(\gamma(t)) \ge 2$ for all $t$.
\end{remark}

\begin{definition}[$J$-path]
  \label{def:weak_jump_path}
  A \emph{$J$-path} is a càdlàg function that satisfies conditions (i) and (ii) of \cref{def:jump_path}.
\end{definition}

Unlike a genealogical path, a $J$-path does not take into account the healing marks, requiring only the presence of at least one particle along its trajectory.
In particular, every genealogical path is also a $J$-path.

\begin{definition}[Infection path]\label{def:infection_path} An \emph{infection path} starting at a set
  $A\subset\Z^d$ at time $0$ is a genealogical path $\gamma:[0,T]\to\Z^d$ with $\gamma(0)\in A$.
\end{definition}

Since the conditions appearing in the definition of genealogical paths (\cref{def:jump_path}) are monotone in $\eta$, the existence of a
genealogical path traversing a given space-time region is an increasing event in $\rho$ and a decreasing event in $\delta$.

Furthermore, due to the local nature of these paths, if there are two genealogical paths
$\gamma_1: [s_0, s_1] \to \Z^d$ and $\gamma_2: [s_1, s_2] \to \Z^d$ such that $\gamma_1(s_1) =
\gamma_2(s_1)$, then their concatenation $\gamma: [s_0, s_2] \to \Z^d$ is also a
genealogical path.

We consider the zero-range process started from the invariant product measure $\mu_\rho$ of density $\rho$, and with all particles at the origin (if any) declared infected at time $0$ --- that is, we consider $A=\{0\}$.

We say a site $y$ is \emph{infected at time $t$} if there exists an infection path starting from $A=\{0\}$ at time $0$ that reaches $y$ at time $t$.
See \cref{fig:infection}.
We say that the infection \emph{dies out} if there exists $T < \infty$ such that no
site is infected at time $T$.
Naturally, the infection dies out already at time $t=0$ if the initial configuration assigns no particle to the origin, that is if $\eta_0(0)=0$.

We denote by $\P_{\delta}^{\eta_0}$ the joint law of the slot representation and the infection process with healing rate $\delta$, starting from a fixed particle configuration $\eta_0$. When the initial configuration is drawn from the invariant measure, we integrate over $\mu_\rho$ and write
\begin{equation}
    \P_{\rho, \delta} (\cdot) = \int \P_{\delta}^{\eta_0} (\cdot) \, \textnormal{d} \mu_{\rho} (\eta_0).
\end{equation}

We can simultaneously couple all healing rates $\delta \ge 0$ in a single graphical representation by using an independent Poisson point process on $\R_+ \times \R_+$ with Lebesgue intensity for each site $x$, and defining $\mathcal{R}_x$ as the projection onto the time coordinate of marks below level $\delta$. Because of this natural coupling, we will frequently suppress the parameter $\delta$ from our notation, writing simply $\P^{\eta_0}$ and $\P_\rho$ when the specific healing rate is clear from context.

\subsection{Bounds for the zero-range process}
\label{sec:zr_bounds}

In this section, we collect several bounds for the zero-range process. 
The proofs of these bounds are deferred to \cref{sec:deferred_proofs} and \cref{sec:exponential_families}.

We begin with two static bounds concerning the particle distribution under the invariant measure $\mu_\rho$. Since the invariant measure is a product measure, the particle counts at any $n \ge 1$ distinct sites are distributed as independent copies $X_1, \dots, X_n$ of the marginal distribution $X$ given as in \eqref{eq:nu_phi}. 
We denote their empirical mean by $\bar{X} = \frac{1}{n}\sum_{i=1}^n X_i$. For these static bounds, it will be convenient to define the ratio of the jump-rate bounds as $\Gamma = \Gamma_+ / \Gamma_-$.

\begin{proposition}[Poisson domination]
\label{lem:poisson_domination}
For any density $\rho > 0$, the sum $\sum_{i = 1}^n  X_i$ is stochastically dominated by a Poisson random variable with mean $\Gamma \rho n$.
\end{proposition}

We also need standard concentration inequalities to bound the probability of observing moderate deviations from the mean. In \cref{sec:exponential_families} we verify that the invariant measures $\mu_{\rho}$ form an exponential family and derive the following tail bounds.

\begin{proposition}[Tail bounds]\label{lem:tail_bounds}
  For any relative deviation $\varepsilon \in (0, 1)$, the following upper and lower tail bounds hold:
  \begin{equation}
    \begin{gathered}
    \P\big(\bar{X} \ge \rho(1+\varepsilon)\big) \le \exp\braces*{-n \frac{\rho \varepsilon^2}{3\Gamma}},\\
     \P\big(\bar{X} \le \rho(1-\varepsilon)\big) \le \exp\braces*{ -n \frac{\rho \varepsilon^2}{2\Gamma}}.
     \end{gathered}
  \end{equation}
\end{proposition}

Next, we establish bounds on the dynamic behavior of the process. 
In the slot representation, at any given instant each particle occupies a single slot. It follows from~\eqref{eq:g_cond} that each particle has a jump rate bounded from above by $\Gamma_+$ and from below by $\Gamma_-$. Furthermore, a particle's slot may change over time, but the memoryless property of the exponential distribution ensures that, after each
such reassignment, the residual waiting time for the particle's next jump is still an exponential random variable of rate at most $\Gamma_+$. Therefore the number of jumps performed by each single particle within a time interval of length $T$ is stochastically dominated by $Y \sim
\poisson(\Gamma_+ T)$, regardless of the behavior of the other particles. Since all jumps are to nearest neighbors, the maximal displacement $\sup_{0 \le s \le T} \|X_s - X_0\|$---where $X_t$ denotes the particle position at time $t$---is stochastically dominated by $\max_{0 \le k \le Y} \|S_k\|$, where $S$ denotes the discrete-time simple random walk.

This individual displacement bound can be used to obtain a bound on the probability that any particle from a distant site reaches the origin, as stated below for an arbitrary dimension.

\newbigconstant{C:displacement}
\newconstant{c:displacement}
\newbigconstant{C:crossing}
\newconstant{c:crossing}
\begin{proposition}[Displacement bound and crossing probability]\label{lem:crossing_bound}
  There exist constants $\ubc{C:displacement}(\Gamma_+, d), \uc{c:displacement}(\Gamma_+, d) > 0$ such that, if $Y \sim \poisson(\Gamma_+ T)$ is independent of the random walk $S$ on $\Z^d$,
  \begin{equation}\label{eq:displacement_bound}
    \P\Big( \max_{0 \le k \le Y} \|S_k\| \ge s \Big)
    \le \ubc{C:displacement} e^{-\uc{c:displacement} \, s \log(1+s/T)},
    \qquad T \ge 1.
  \end{equation}
  In particular, there exists constants $\ubc{C:crossing}(\Gamma_+, d), \uc{c:crossing}(\Gamma_+, d) > 0$ such that
  \begin{equation}\label{eq:crossing_bound}
    \P_\rho \Bigg(
      \begin{array}{c}
        \text{some particle starting at a site $x \in \Z^d$}\\
        \text{with $\normI{x} \ge s$ reaches $0$ by time $T$}
    \end{array} \Bigg)
    \le \ubc{C:crossing} \rho T^{d/2} e^{-\uc{c:crossing} s\log(1 + s/T)},
  \end{equation}
  for any $T \geq 1$ and $s \ge 1$.
\end{proposition}

The displacement bounds provided above govern movement of individual particles. We now provide a bound estimating the probability that two particles meet.
For this purpose, we define the dimension-dependent exponent $\gamma_d \ge 0$ as $\gamma_1 = 0$, $\gamma_2 = 1$, and $\gamma_d = d-2$ for $d \ge 3$.

\newbigconstant{C:meeting}
\newbigconstant{C:meeting_distance}
\begin{proposition}[Particle meeting]\label{lem:meeting}
  Let $(W_t)_{t \ge 0}$ be a continuous-time nearest-neighbor symmetric random walk on $\Z^d$, absorbed at the origin, with a time-dependent total jump rate that is bounded below by $2\Gamma_- > 0$ almost surely. There exist constants $\ubc{C:meeting_distance}(\Gamma_-, d) \in (0, 1)$ and $\ubc{C:meeting}(d) >0$ such that for any initial position $W_0$ with $\|W_0\| \le R$ where $R \ge 1$, and any time $T > 0$ satisfying $R \le \sqrt{\ubc{C:meeting_distance} T}$, the walk hits the origin by time $T$ with probability
  \begin{equation}
    \P\Big( \inf_{t \in [0, T]} \|W_t\| = 0 \Big) \ge \ubc{C:meeting} R^{-\gamma_d}.
  \end{equation}
\end{proposition}

Also needed are bounds on the extremal number of particles a site can contain in a continuous time interval.

\newconstant{c:occupation}\begin{proposition}[Continuous-time occupation bounds]\label{lem:occupation_bounds}
  There exists a constant $\uc{c:occupation}(\Gamma_+) > 0$ such that, for any time horizon $T \ge 0$ and any site $x$, the occupation at $x$ satisfies the following bounds.
  
  For any threshold $M \ge 1$ and any choice of $K \le \uc{c:occupation} M$:
  \begin{equation}
    \P_\rho\Big( \sup_{t \in [0, T]} \eta_t(x) \ge M \Big) \le (T + 1) \Big[ \P_\rho\big( \eta_0(x) \ge K \big) + e^{-\uc{c:occupation} M} \Big].
  \end{equation}
  
  For any threshold $m \ge 1$ and any choice of $K \ge \uc{c:occupation}^{-1} m$:
  \begin{equation}
    \P_\rho\Big( \inf_{t \in [0, T]} \eta_t(x) \le m \Big) \le (T + 1) \Big[ \P_\rho\big( \eta_0(x) \le K \big) + e^{-\uc{c:occupation} K} \Big].
  \end{equation}
\end{proposition}

\section{Horizontal decoupling}\label{sec:horizontal_decoupling}

In this section we present a graphical representation for the zero-range process, called the
\emph{slot-priority representation}, and use it to derive a decoupling inequality for events 
depending on the process inside regions that are well-separated in the space coordinate. 
The main idea is to assign priority labels to the particles that compose the environment in a way that the dynamics of higher-priority particles is unaffected by the dynamics of the lower-priority ones.
That property, combined with moderate deviation bounds on the displacement of individual particles from \cref{lem:crossing_bound}, yields the desired decoupling inequality.

\subsection{Priorities and order dependence}
\label{sec:slot_priority}

We modify the slot representation by assigning \emph{priority} levels to individual particles,
and defining the \emph{slot-priority dynamics} in which particles within each stack are chosen to jump based on their level.
The central result of this section (\cref{prop:order_dependence}) states that
the trajectory of any particle with a given priority is entirely determined by the initial
positions of the particles with higher priorities and the Poisson clocks that they use. 
In particular, it does not depend on the positions and the displacements of lower-priority particles.

\begin{definition}[Priority assignment and dynamics]\label{def:priority} A \emph{priority assignment} for
  a configuration $\eta_0$ is a map that assigns to every particle a value in $\R$. We write $\xi(x,j) \in \R \cup \{\varnothing\}$ for the priority of the particle at the slot $(x,j)$, assigning priority $\varnothing$ to empty slots, and adopt the convention that smaller values correspond to higher priority and that $r < \varnothing$ for every $r \in \R$.

  Let $n(\xi, x) = \sum_{i=1}^\infty \charf{\xi(x, i) \neq \varnothing}$ denote the \emph{height} of
  the stack at $x$, and for $r \in \R$, let $n_{\le r}(\xi, x) =
  \sum_{i=1}^\infty \charf{\xi(x, i) \le r}$ denote the number of particles at $x$ with priority at most $r$.

  Let $\Xi_{\text{packed}} \subset (\R \cup \{ \varnothing \})^{\Z^d \times \N}$ denote the space of
  packed and priority-sorted configurations. 
  By being priority-sorted we mean that for every site $x$  the inequalities
  \begin{equation}
    \xi(x, 1) \le \xi(x, 2) \le \cdots \le \xi(x, n(\xi, x))
  \end{equation}
  hold, and $\xi(x, j) = \varnothing$ for all $j > n(\xi, x)$. 
  This implies that the particles with priority at
  most $r$ occupy slots $1, \ldots, n_{\le r}(\xi, x)$, a requirement which is compatible with any packed initial configuration $\eta_0$, since it only prescribes the ordering within each stack.

  The \emph{slot-priority dynamics} is given by the Markov process on $\Xi_{\text{packed}}$
  generated by
  \begin{equation}
    L^{\text{prio}} H(\xi) = \sum_{x \in \Z^d} \sum_{i = 1}^{n(\xi, x)}  \lambda_i \sum_{y\sim x} \bigl[ H(\xi^{(x, i) \to y}) - H(\xi) \bigr],
  \end{equation}
  acting on local functions $H: \Xi_{\text{packed}} \to \R$. Here $\xi^{(x, i) \to y}$ is the
  priority configuration resulting from the particle in $(x, i)$ jumping to site $y$, defined as
  follows:
  \begin{itemize}
    \item The departing particle with priority $r = \xi(x, i)$ is removed from slot $(x, i)$.
    \item Every remaining particle at site $x$ in a slot $(x, j)$ for $j > i$ shifts down to slot $j - 1$, and we set $\xi(x, n(\xi, x)) = \varnothing$.
    \item At the target site $y$, the arriving particle is placed in slot $(y, k + 1)$, where
      $k=n_{\le r}(\xi, y)$. The particles originally at slots $(y, j)$ for $j \in \{k+1,
      \ldots, n(\xi, y)\}$ shift one slot up to $(y, j + 1)$.
  \end{itemize}

  Equivalently, in terms of the graphical representation, we associate to each slot $(x, i)$ and each nearest neighbor $y$ with $|y-x|=1$ an
  independent Poisson clock $\mathcal{P}_{x,y}^i$ with rate
  $\lambda_i$. When the clock $\mathcal{P}_{x,y}^i$ rings, if $i \le n(\xi, x)$, the system jumps to $\xi^{(x, i) \to y}$.
\end{definition}

\begin{remark}\label{rmk:priority_packed} The slot-priority dynamics preserves the packed and sorted
  property: a departure from slot $i$ shifts all higher-indexed particles one slot down, and an
  arrival inserts the particle into the correct position to maintain sorting, shifting
  lower-priority particles up by one.
\end{remark}

\begin{remark} When all particles share the same priority, the arrival rule reduces to placing the
  particle on top of the stack, recovering the standard slot dynamics of \cref{def:slot_dynamics}.
\end{remark}

\begin{proposition}[Order dependence]\label{prop:order_dependence} Let $(\xi_t)_{t\ge 0}$ be the
  slot-priority dynamics. Fix a priority level $r^\star\in \R$, and let $\pi^\star:
  \Xi_{\text{packed}} \to \N_0^{\Z^d}$ be the projection that counts particles with priority at most
  $r^\star$, defined as $\pi^\star(\xi)(x) = n_{\le r^\star}(\xi, x)$. The projected process
  $\zeta_t = \pi^\star(\xi_t)$ satisfies:
  \begin{enumerate}[label=(\alph*)]
    \item $(\zeta_t)_{t\ge 0}$ is a zero-range process.
    \item For each realization, the trajectory of $(\zeta_t)_{t\ge
      0}$ is a deterministic function of the initial configuration $\zeta_0$ and the full set of
      Poisson clocks $(\mathcal{P}_{x,y}^i)_{x\in\Z^d, y\sim x, i\in\N}$.
  \end{enumerate}
\end{proposition}

\begin{proof}
  Since the slot-priority dynamics preserves the sorting property (\cref{rmk:priority_packed}), the
  high-priority particles at $x$ always occupy exactly the bottom $\zeta_t(x)$ slots. 
  To prove (a), consider a local function $F$ depending on $\xi$ only through the high-priority counts, meaning $F
  = f \circ \pi^\star$ for some local function $f: \N_0^{\Z^d} \to \R$. Under the transition of a
  particle from $(x, i)$ to $y$, the projection changes if and only if $i \le \pi^\star(\xi)(x)$, in
  which case it becomes $\pi^\star(\xi)^{x, y}$. The action of $L^{\text{prio}}$ on $F$ thus
  simplifies to:
  \begin{equation}
  \begin{split}
    L^{\text{prio}} (f \circ \pi^\star)(\xi) &= \sum_{x \in \Z^d} \sum_{i = 1}^{n(\xi, x)} \lambda_i \sum_{y\sim x} \bigl[ f(\pi^\star(\xi^{(x, i) \to y})) - f(\pi^\star(\xi)) \bigr] \\
    &= \sum_{x \in \Z^d} \sum_{i = 1}^{\pi^\star(\xi)(x)} \lambda_i \sum_{y \sim x} \bigl[ f(\pi^\star(\xi)^{x, y}) - f(\pi^\star(\xi)) \bigr] \\
    &= (L f)(\pi^\star(\xi)),
  \end{split}
  \end{equation}
  where $L$ is the standard zero-range generator from \Cref{eq:zrp_gen}. Since the generator
  projects properly, it follows that $(\zeta_t)_{t\ge 0}$ is a zero-range process.

  To prove (b), simply note that in the graphical representation, if $\mathcal{P}_{x,y}^i$ rings at time $t$:
  \begin{itemize}
    \item If $i \le \zeta_{t-}(x)$, then $\zeta_t(x)$ decreases by one and $\zeta_t(y)$ increases by one.
    \item If $i > \zeta_{t-}(x)$, the transition leaves $\zeta_t$ unchanged.
  \end{itemize}
  This shows that $\zeta_t$ is entirely determined by $\zeta_0$ and the clocks up to time $t$.
\end{proof}

The same proof also holds for the projections given by $n_{< r^\star}(\xi, x)$, which counts the number of particles with priority strictly less than $r^\star \in \R \cup \{\varnothing\}$. In particular, taking
$r^\star = \varnothing$ we have $n_{< \varnothing}(\xi, x) = n(\xi, x)$, which proves the following corollary:

\begin{corollary}[Priority dynamics is a zero-range process] Let $\eta_0 \in \N_0^{\Z^d}$ be an
  initial configuration drawn from $\mu_\rho$ for some $\rho > 0$, and let $(\xi_t)_{t \ge 0}$ be
  the slot-priority dynamics of \cref{def:priority} started from any priority assignment compatible
  with $\eta_0$. Define $\eta_t(x) = n(\xi_t, x)$ as the number of occupied slots at site $x$ at
  time $t$.

  Then $(\eta_t)_{t \ge 0}$ has the law of the zero-range process on $\Z^d$ with rate function $g$ and initial condition $\eta_0$.
\end{corollary}

\subsection{Decoupling geometry}

Let $B^1, B^2 \subset \Z^d \times [0, T]$ be space-time regions with spatial projections $S_1, S_2$ satisfying $\diam(S_j) \le w$ for $j=1,2$, and $T \ge 1$.
Recall the \emph{horizontal separation} $\dist_H = \dist_H(B^1, B^2)$ defined in \eqref{eq:distances}.
We assume henceforth that $\dist_H \ge 8$ and write $d_0 = \floor*{\frac{\dist_H}{4}}$.

For a point $x \in \Z^d$ and a set $S \subset \Z^d$, we denote their distance in the infinity norm by $\dist_\infty(x, S) = \inf_{y \in S} \normI{x - y}$. We define the $d_0$-neighborhoods of $S_1$ and $S_2$ as the sets
\begin{equation}
  A_1 = \{x \in \Z^d : \dist_\infty(x, S_1) \le d_0\}, \qquad A_2 = \{x \in \Z^d : \dist_\infty(x, S_2) \le d_0\}.
\end{equation}
Since $\dist_H \ge 4d_0$, the sets $A_1$ and $A_2$ are disjoint and separated by a distance of at least $2d_0$.

\begin{definition}[Priority partition]\label{def:priority_partition} Given an initial configuration
  $\eta_0$ drawn from $\mu_\rho$, assign priorities to particles as follows.
  \begin{itemize}
    \item Priority 1 (highest): all particles with initial position inside $A_1$.
    \item Priority 2: all particles with initial position inside $A_2$.
    \item Priority 3 (lowest): all particles with initial position outside $A_1 \cup A_2$.
  \end{itemize}
  We write $\eta^{(k)}_t$ for the configuration restricted to particles of priority $k$, so that
  $\eta_t = \eta^{(1)}_t + \eta^{(2)}_t + \eta^{(3)}_t$.
\end{definition}

We establish horizontal decoupling by splitting the space into distinct regions and defining an event under which particles originating in different regions never interact. Since the priority-based dynamics prevents lower-priority particles from influencing higher-priority ones, such a non-interaction event allows us to decouple observations from the two regions.

For this purpose, define the boundaries $M_1 = \{x \in \Z^d : \dist_\infty(x, S_1) = 2d_0\}$ and $M_2 = \{x \in \Z^d : \dist_\infty(x, S_2) = 2d_0\}$. Let $V_1$ and $V_2$ denote the interiors bounded by $M_1$ and $M_2$, respectively. Since $\dist_\infty(S_1, S_2) \ge 4d_0$, the regions $V_1$ and $V_2$ are disjoint.

\begin{definition}[Splitting event]\label{def:split_event}
  Define the event
  \begin{equation}
    \S = \S_1 \cap \S_2 \cap \S_3,
  \end{equation}
  where
  \begin{equation}
  \begin{gathered}
    \S_1 = \bigl\{\text{no priority-1 particle reaches } M_1 \text{ during }[0,T]\bigr\},\\
    \S_2 = \bigl\{\text{no priority-2 particle reaches } M_2 \text{ during }[0,T]\bigr\},\\
    \S_3 = \bigl\{\text{no priority-3 particle enters }S_1 \cup S_2\text{ during }[0,T]\bigr\}.        
  \end{gathered}
  \end{equation}
\end{definition}
See \cref{fig:horizontal_decoupling_setup} for an illustration of these events.

On the event $\S$, the priority-1 particles never leave $V_1$, the priority-2 particles never leave $V_2$, and the priority-3 particles never reach the base regions $S_1$ and $S_2$.

\newbigconstant{C:split_event}
\begin{proposition}\label{prop:split_event} There exists a constant $\ubc{C:split_event} =
  \ubc{C:split_event}(\Gamma_+, d) > 0$ such that, for all $\rho > 0$, $T \ge 1$, and $\dist_H \ge 8$,
  \begin{equation}\label{eq:split_event_bound}
    \P_\rho\bigl[\S^c\bigr] \le \ubc{C:split_event} w^{d-1} \rho T^{d/2} e^{-\uc{c:crossing}
    \dist_H \log(1+\dist_H/T)}.
  \end{equation}
\end{proposition}

\begin{proof}
  By a union bound, $\P_\rho[\S^c] \le \P_\rho[\S_1^c] + \P_\rho[\S_2^c] + \P_\rho[\S_3^c]$. We bound each term separately.

  The event $\S_1^c$ requires some priority-1 particle to reach $M_1$. Every priority-1 particle starts inside $A_1$, meaning its initial distance to $M_1$ is at least $d_0$.
  By taking a union bound over all sites on $M_1$, and applying \cref{lem:crossing_bound} with $s = d_0$, we obtain $\P_\rho[\S_1^c] \le C (w+2d_0)^{d-1} \rho T^{d/2} e^{-\uc{c:crossing} d_0 \log(1+d_0/T)}$. Since $d_0 = \floor{\dist_H/4} \ge \dist_H/8$, and any polynomial factor in $d_0$ can be absorbed into the exponential decay, this gives the required bound. A symmetric argument bounds $\P_\rho[\S_2^c]$.

  For $\S_3^c$, a priority-3 particle starting outside $A_1 \cup A_2$ must enter $S_1 \cup S_2$.
  Such a particle must travel an initial distance of at least $d_0$ to reach $S_1 \cup S_2$. Taking a union bound over the boundaries of $S_1$ and $S_2$, and applying \cref{lem:crossing_bound} provides $\P_\rho[\S_3^c] \le C w^{d-1} \rho T^{d/2} e^{-\uc{c:crossing} d_0 \log(1+d_0/T)}$. Summing the three bounds and relabeling the constant proves \eqref{eq:split_event_bound}.
\end{proof}

On the splitting event $\S$, functions supported on $B^1$ and $B^2$ can be replaced by independent counterparts.

\begin{figure}[tpb]
  \centering
  \begin{subfigure}[t]{0.48\textwidth}
    \centering
    \includegraphics[width=\linewidth]{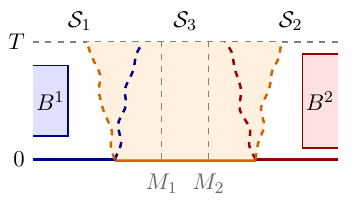}
    \caption{On the event $\S$, priority-1 (blue) and priority-2 (red) particle displacements are bounded by $M_1$ and $M_2$. Priority-3 (orange) particles do not reach the boxes $B^1$ and $B^2$.}
    \label{fig:horizontal_decoupling_setup}
  \end{subfigure}\hfill
  \begin{subfigure}[t]{0.48\textwidth}
    \centering
    \includegraphics[width=\linewidth]{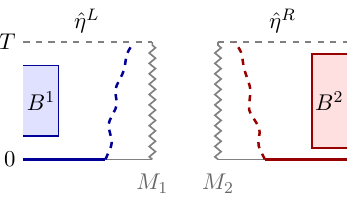}
    \caption{Conditioned on $\S$, we couple the dynamics to independent processes $\hat{\eta}^L$ and $\hat{\eta}^R$, isolated by absorbing boundary conditions at $M_1$ and $M_2$.}
    \label{fig:horizontal_decoupling_decoupled}
  \end{subfigure}
  \caption{Illustration of the horizontal decoupling scheme. The boundaries $M_1$ and $M_2$ divide the space to prevent information from propagating between the regions containing $B^1$ and $B^2$.}
  \label{fig:horizontal_decoupling}
\end{figure}

\begin{definition}[Decoupled processes] Consider the slot-priority dynamics with the priority partition of
  \cref{def:priority_partition}. Define two auxiliary processes.
  \begin{enumerate}[label=(\roman*)]
    \item $\hat\eta^L_t$ is the zero-range process on $V_1$ with an absorbing boundary condition at $M_1$ (particles that jump to $M_1$ are removed), started from
      \[
        \hat\eta^L_0(x) = \eta^{(1)}_0(x) \text{ for } x \in V_1.
      \]

    \item $\hat\eta^R_t$ is the zero-range process on $V_2$ with an absorbing boundary condition at $M_2$, started from
      \[
        \hat\eta^R_0(x) = \eta^{(2)}_0(x) \text{ for } x \in V_2.
      \]
  \end{enumerate}
  Both processes are driven by the same Poisson clocks as the original process, restricted to their respective regions.
\end{definition}

\begin{remark}[Independence of the decoupled processes]\label{rmk:independence_decoupled} The
  process $\hat\eta^L$ depends on the initial data $(\eta_0(x))_{x \in A_1}$ and the internal clocks
  $(\mathcal{P}_{x,y}^i)_{x \in V_1, i \in \N}$, while $\hat\eta^R$ depends on $(\eta_0(x))_{x \in A_2}$ and $(\mathcal{P}_{x,y}^i)_{x \in V_2, i \in \N}$. The initial measure $\mu_\rho$ is a product
  measure and the sets $A_1$ and $A_2$ are disjoint, so the two groups of
  initial conditions are independent. The Poisson clocks internal to $V_1$ and $V_2$ are
  independent as well, since these are disjoint sets. Therefore $\hat\eta^L$ and $\hat\eta^R$ are
  independent.
\end{remark}

\begin{lemma}[Agreement on $\S$]
\label{lem:agreement} 
Let $f_1, f_2:\N_0^{\Z^d\times\R_+}\to\R$ be functions with support in $B^1$ and $B^2$ respectively. Define
  \begin{equation}
    \hat f_1 = f_1\bigl(\hat\eta^L\bigr), \qquad \hat f_2 = f_2\bigl(\hat\eta^R\bigr),
  \end{equation}
  where $f_1(\hat\eta^L)$ means $f_1$ evaluated on the trajectory $\hat\eta^L$ extended by zero
  outside $V_1$, and similarly for $f_2$.

  Then, on the event $\S$,
  \begin{equation}
    f_1(\eta) = \hat f_1 \quad\text{and}\quad f_2(\eta) = \hat f_2.
  \end{equation}
\end{lemma}

\begin{proof}
  We first prove $f_1(\eta) = \hat f_1$ on $\S$. 
  Since $f_1$ has support in $B^1$ and $S_1 \subset V_1$, it suffices to show that $\eta_t(x) =
  \hat\eta^L_t(x)$ for all $(x, t) \in B^1$ on the event $\S$.

  We first observe that, on $\S$, only priority-1 particles contribute to the occupation numbers in $B^1$. The event $\S_3$ prevents any priority-3 particle from entering $S_1$. As for priority-2 particles: each one starts inside $A_2$, and any priority-2 particle reaching $S_1$ would have to traverse $M_2$ to leave $V_2$. But $\S_2$ ensures that no priority-2 particle reaches $M_2$ during $[0,T]$, so no priority-2 particle can enter $B^1$ on $\S$. Therefore, on $\S$,
  \begin{equation}\label{eq:only_p1}
    \eta_t(x) = \eta^{(1)}_t(x) \quad \text{for all } (x, t) \in B^1.
  \end{equation}

  Next, by \cref{prop:order_dependence} applied with $r^\star=1$, the site trajectory of every priority-1 particle is a deterministic function of the initial positions of priority-1 particles and the Poisson clocks. In particular, the configuration $\eta^{(1)}_t$ restricted to any spatial region is the same whether or not priority-2 and priority-3 particles are present in the system.

  We now compare $\eta^{(1)}$ with $\hat\eta^L$. Since the process $\hat\eta^L$ is driven by the same Poisson clocks (restricted to $V_1$) and starts from the same priority-1 particles in the same positions, the only potential divergence would be due to the different boundary conditions. But, on $\S_1$, no priority-1 particle reaches $M_1$ during $[0,T]$, so the absorbing boundary condition of $\hat\eta^L$ is never triggered, and no particle is removed.

  Combining these observations, on $\S$ the priority-1 particles in $S_1$ follow identical trajectories under $\eta$ and $\hat\eta^L$:
  \begin{equation}
    \eta_t(x) = \eta^{(1)}_t(x) = \hat\eta^L_t(x)
    \qquad \text{for all } (x, t) \in B^1,
  \end{equation}
  and hence $f_1(\eta)=f_1(\hat\eta^L)=\hat f_1$.

  By applying a completely symmetric argument, on $\S$, priority-1 and priority-3 particles never enter $S_2$, and the priority-2 particles are unaffected by the absorbing boundary condition of $\hat\eta^R$, yielding identical trajectories in $B^2$ and thus $f_2(\eta) = \hat f_2$.
\end{proof}

\subsection{Proof of Theorem~\ref{thm:horizontal_decoupling}}

We now combine the previous results into the main decoupling inequality.

\begin{proof}[Proof of \cref{thm:horizontal_decoupling}]
  Let $\hat f_1$ and $\hat f_2$ be as in \cref{lem:agreement}, and let $\S$ be the splitting event of
  \cref{def:split_event}. Since $\hat f_1$ and $\hat f_2$ are independent
  (\cref{rmk:independence_decoupled}), we have $\E_\rho[\hat f_1 \hat f_2] = \E_\rho[\hat f_1]
  \E_\rho[\hat f_2]$. On the event $\S$, \cref{lem:agreement} gives $f_j = \hat f_j$ for $j = 1, 2$.
  Since $|f_j|, |\hat f_j| \le 1$, it follows that
  \begin{equation}\label{eq:product_approx}
    \abs{\E_\rho[f_1 f_2] - \E_\rho[\hat f_1 \hat f_2]} \le \P_\rho[\S^c]
  \end{equation}
  and
  \begin{equation}\label{eq:f_replacement}
    \abs{\E_\rho[f_j] - \E_\rho[\hat f_j]} \le \P_\rho[\S^c], \qquad j = 1, 2.
  \end{equation}
  Combining \eqref{eq:product_approx} and \eqref{eq:f_replacement} with the triangle inequality, and
  using that all expectations are bounded by $1$ in absolute value, we can bound
  \begin{align*}
    \abs{\cov_\rho(f_1, f_2)} &= \abs{\E_\rho[f_1 f_2] - \E_\rho[f_1] \E_\rho[f_2]}\\
    &\le \abs{\E_\rho[f_1 f_2] - \E_\rho[\hat f_1 \hat f_2]}\\
    &\quad + \abs{\E_\rho[\hat f_1] \E_\rho[\hat f_2] - \E_\rho[f_1] \E_\rho[f_2]}\\
    &\le \P_\rho[\S^c] + \abs{\E_\rho[\hat f_2] - \E_\rho[f_2]}
    + \abs{\E_\rho[\hat f_1] - \E_\rho[f_1]}\\
    &\le 3 \, \P_\rho[\S^c].
  \end{align*}
  The statement follows from \cref{prop:split_event} by setting $\ubc{C:horiz_decoupling} = 3 \,
  \ubc{C:split_event}$ and $\uc{c:horiz_decoupling} = \uc{c:crossing}$.
\end{proof}

\begin{remark}\label{rmk:horizontal-graphical}
    Since our proof does not alter the slot representation besides assigning priorities, and by construction the graphical representation within a box is independent for disjoint boxes, the theorem still holds even if $f_i$ is a function which also depends on the slot representation (with healing marks) within $B^i$.
\end{remark}

\section{Vertical decoupling}
\label{sec:improved_vertical_decoupling}

In this section we prove \cref{thm:vertical_decoupling}, which provides decay of correlations
for events supported on vertically separated space-time boxes. This complements the horizontal
decoupling established in \cref{thm:horizontal_decoupling}.

The proof generalizes the coupling strategy introduced in~\cite{BaldassoTeixeira18} for the exclusion process and
applied to the one-dimensional zero-range process in~\cite{BaldassoTeixeira20}, where the density was restricted
to a bounded interval. 
Here we extend the argument to arbitrary dimensions and remove the restriction to bounded densities. 
This latter extension is needed for the high-density survival regime.

\subsection{Coupling strategy and finite-volume reduction}

Fix two space-time boxes $B^1, B^2$ with spatial side length $w>0$ and temporal length $w$. Denote by $\dist_V = \dist_V(B^1, B^2)$ their vertical distance. We assume, without loss of generality, that
\begin{equation}
  B^{1} \subset \Z^{d} \times \R_{-} \quad \text{and} \quad B^{2} = [-w/2, w/2]^{d} \times [\dist_{V}, \dist_{V}+w].
\end{equation}

Consider two independent initial configurations $\xi$ and $\xi'$ with densities $\rho$ and $\rho(1+\varepsilon)$, respectively. Suppose we construct a coupling of the evolutions $\xi_{t}$ and $\xi'_{t}$ in which $(\xi'_{t})_{t \geq 0}$ is independent of $\xi_{0}$. Define the domination event
\begin{equation}\label{eq:domination_event}
  D = \big\{ (\xi, \xi'): \xi_{t}(x) \leq \xi'_{t}(x) \text{ for all } (x,t) \in B^{2} \big\}.
\end{equation}
Since $B^1$ lies entirely before $B^2$ in the time coordinate, the Markov property gives
\begin{equation}
  \begin{split}
    \E_{\rho}[f_{1}f_{2}] &= \E[ f_{1}(\xi)f_{2}(\xi)] = \E\big[ f_{1}(\xi) \E[f_{2}(\xi) \mid \xi_{0}] \big] \\
    & \leq \E\big[ f_{1}(\xi) \E[f_{2}(\xi') + \charf{D^{c}}(\xi, \xi') \mid \xi_{0}] \big] \\
    & \leq \E[f_{1}(\xi)] \E[f_{2}(\xi')] + \P \big( (\xi, \xi') \in D^{c} \big) \\
    & \leq \E_{\rho}[f_{1}] \E_{\rho(1+\varepsilon)}[f_{2}] + \P \big( (\xi, \xi') \in D^{c} \big).
  \end{split}
\end{equation}

It remains to construct a coupling achieving~\eqref{eq:domination_event} and to bound $\P(D^c)$.
We begin by restricting to a finite spatial window. Let
\begin{equation}\label{eq:shaded}
  S = [-2w-\dist_{V}, 2w+\dist_{V}]^{d},
\end{equation}
and let $\mathcal{I}$ be the event that some particle of $\xi$ lying outside $S$ at time $\dist_{V}$ enters $B^{2}$ during the interval $[\dist_{V}, \dist_{V}+w]$. Then
\begin{equation}\label{eq:domination_split}
  \P \big( (\xi, \xi') \in D^{c} \big) \leq \P\big( \xi_{\dist_{V}}(x) > \xi'_{\dist_{V}}(x) \text{ for some } x \in S \big) + \P \big( \mathcal{I} \big).
\end{equation}

\newbigconstant{C:invasion}
\newconstant{c:invasion}
\begin{lemma}[Finite-volume restriction]\label{lem:invasion}
  There exist constants $\ubc{C:invasion} > 0$ and $\uc{c:invasion} > 0$ such that
  \begin{equation}\label{eq:invasion}
    \P \big( \mathcal{I} \big) \leq \ubc{C:invasion}\rho \exp\braces{-\uc{c:invasion} (w+\dist_{V})}.
  \end{equation}
\end{lemma}
\begin{proof}
  To enter $B^2$ during the interval $[\dist_V, \dist_V + w]$, a particle must cross the boundary of its spatial projection $S_2 = [-w/2, w/2]^d$. By translation invariance in time, the probability is bounded by considering particles starting at time $0$ outside $S$ and entering $S_2$ by time $w$. Since the distance from the complement of $S$ to $S_2$ is strictly greater than $w + \dist_V$, taking a union bound over $x \in \partial S_2$ and applying \cref{lem:crossing_bound} yields
  \begin{equation}
    \begin{split}
      \P \big( \mathcal{I} \big)  &\leq \sum_{x \in \partial S_2} \P\parens*{
        \begin{array}{c}
          \text{some particle at $(y, 0)$ with $\normI{y - x} \geq w+\dist_{V}$}\\
          \text{reaches $x$ in the time interval $[0, w]$}
      \end{array}} \\
      &\leq 2d w^{d - 1} \ubc{C:crossing} \rho w^{d/2}e^{-\uc{c:crossing} (w+\dist_V)\log(1+(w+\dist_V)/w)} \leq \ubc{C:invasion}\rho e^{-\uc{c:invasion} (w+\dist_{V})}. \qedhere
    \end{split}
  \end{equation}
\end{proof}

We couple $(\xi, \xi')$ up to time $\dist_V$ and use identical graphical constructions from that time onward. Combining~\eqref{eq:domination_split} and~\eqref{eq:invasion} then gives
\begin{equation}\label{eq:bad_coupling}
  \P \big( (\xi, \xi') \in D^{c} \big) \leq \P \big( \mathcal{D}^{c} \big) + \ubc{C:invasion}\rho e^{-\uc{c:invasion} (w+\dist_{V})},
\end{equation}
where $\mathcal{D} = \{ \xi_{\dist_V}(x) \leq \xi'_{\dist_V}(x) \text{ for all } x \in S \}$.

\subsection{Block decomposition and the matching coupling}

To bound $\P(\mathcal{D}^c)$, we partition space and time into blocks. Fix a time step $T \in [1, \dist_V]$ and set $L = \lfloor \sqrt{\kappa T} \rfloor$, where we set $\kappa = \ubc{C:meeting_distance} / d$ and $\ubc{C:meeting_distance} \in (0, 1)$ is the specific constant from \cref{lem:meeting}.
By setting the constant $\ubc{C:V} = (d / \ubc{C:meeting_distance})^{\frac{d+2+\gamma_d}{2}}$ (which depends only on $d$, $\Gamma_+$, and $\Gamma_-$), the assumption $\dist_V \ge \ubc{C:V}$ from \cref{thm:vertical_decoupling} guarantees that $L \ge 1$.

We partition the lattice into blocks of side length $L$:
\begin{equation}
  H(i) = iL + [0,L)^{d}, \qquad i \in \Z^{d}.
\end{equation}
Define the extended spatial region $S_{\mathrm{ext}} = [-2w-4\dist_{V}, 2w+4\dist_{V}]^{d}$, and let $I$ denote the set of all block indices intersecting it:
\begin{equation}
  I = \big\{ i \in \Z^d : H(i) \cap S_{\mathrm{ext}} \neq \emptyset \big\}.
\end{equation}
The region where we require domination is $\bar{S} = \bigcup_{i \in I} H(i)$, and $|I| \leq C \big(\frac{w+\dist_V}{L}\big)^{d}$.

We say that a block $i \in I$ \emph{fails domination} if $\sum_{x \in H(i)} \xi(x) > \sum_{x \in H(i)} \xi'(x)$, and that the pair $(\xi, \xi')$ \emph{fails domination globally} if some block $i \in I$ fails domination.

\newbigconstant{C:init_dom}
\newconstant{c:init_dom}
\begin{lemma}[Initial block domination]\label{lem:initial_domination}
  Assuming $\varepsilon \in (0, 1)$, there exist constants $\ubc{C:init_dom} > 0$ and $\uc{c:init_dom} > 0$ such that
  \begin{equation}\label{eq:no_domination}
    \P\big( (\xi_{0}, \xi'_{0}) \text{ fails domination} \big) \leq \ubc{C:init_dom} \Big(\frac{w+\dist_V}{L}\Big)^{d} e^{-\uc{c:init_dom} L^{d}\rho \varepsilon^{2}}.
  \end{equation}
\end{lemma}
\begin{proof}
  Since $\xi$ and $\xi'$ have densities $\rho$ and $\rho(1+\varepsilon)$, domination fails in a block $H(i)$ only if the total mass of $\xi$ in $H(i)$ exceeds that of $\xi'$. We compare both sums to the midpoint $\rho(1 + \varepsilon/2)|H(i)|$.

  The relative deviation needed for $\xi$ to reach this midpoint is $\varepsilon/2$, while for $\xi'$ to fall to it the required relative deviation is $\frac{\varepsilon/2}{1+\varepsilon} \ge \varepsilon/4$, where the last inequality uses $\varepsilon < 1$.

  Set $N_0 = \sum_{x \in H(0)} \xi(x)$, and define $N_0'$ analogously. A union bound over $i \in I$ together with the tail bounds of \cref{lem:tail_bounds} gives
  \begin{equation}
    \begin{aligned}
      \P\big( (\xi_{0}, \xi'_{0}) \text{ fails domination} \big)
      &\leq |I| \P\bigg( N_0 \geq |H(0)| \rho (1+\frac{\varepsilon}{2}) \bigg) \\
      &\quad + |I| \P\bigg( N_0' < |H(0)| \rho(1+\varepsilon)(1-\frac{\varepsilon}{4}) \bigg) \\
      &\leq \ubc{C:init_dom} \Big(\frac{w+\dist_V}{L}\Big)^{d} e^{-\uc{c:init_dom} L^{d}\rho \varepsilon^{2}},
    \end{aligned}
  \end{equation}
  where the constants from \cref{lem:tail_bounds} have been absorbed into $\uc{c:init_dom}$.
\end{proof}

Set $s_{k} = k T$ for $k = 0, 1, \dots, N$, where $N = \lfloor \dist_V / T \rfloor$. We now describe the coupling of $(\xi, \xi')$ on the interval $[0, \dist_V]$.

Introduce two independent graphical constructions $\mathcal{P}^{1}$ and $\mathcal{P}^{2}$, each as in \cref{def:slot_dynamics}. The process $\xi'$ always evolves according to $\mathcal{P}^{2}$, so it is independent of $\xi_{0}$. The process $\xi$ alternates between the two sets of clocks according to matching rules defined as follows.

Sample $(\xi_{0}, \xi'_{0})$ independently. If the pair fails global domination, evolve $\xi$ entirely with $\mathcal{P}^{1}$. Otherwise, pair the particles of $\xi_{0}$ in $\bar{S}$ to particles of $\xi'_{0}$ in any deterministic manner which respects the following rules:
\begin{itemize}
  \item \emph{Local matching:} first, match particles that share the same site.
  \item \emph{Regional matching:} any remaining particles of $\xi_{0}$ at a site are matched to available particles of $\xi'_{0}$ within the same block $H(i)$.
\end{itemize}

Given this pairing, the evolution of $\xi$ proceeds as follows:
\begin{itemize}
  \item Matched particles that occupy the same site are placed at the bottom of the stack and both evolve according to $\mathcal{P}^{2}$, so they jump together. Upon landing, they are again placed at the bottom of the stack.
  \item A particle of $\xi$ that does not share a site with its match evolves according to $\mathcal{P}^{1}$ and lands on top of the stack.
\end{itemize}
At each time $s_{k}$, the matching is redefined according to the same rules.

\subsection{Proof of Theorem~\ref{thm:vertical_decoupling}}

\begin{proof}[Proof of \cref{thm:vertical_decoupling}]
  Recall the event $\mathcal{D}^c$ from~\eqref{eq:bad_coupling}. Under our coupling, if $(\xi, \xi') \in \mathcal{D}^{c}$, then at least one of the following must occur:
  \begin{enumerate}
    \item Some particle lies outside $\bar{S}$ at some time $s_{k}$ and enters $S$ by time $\dist_{V}$.
    \item The pair $(\xi_{s_{k}}, \xi'_{s_{k}})$ fails global block domination for some $k \leq N$.
    \item Some particle of $\xi$ that remains inside $\bar{S}$ at all times $s_k \le \dist_V$ fails to meet its match during all $N$ intervals.
  \end{enumerate}

  The first event is bounded by the same argument as in \cref{lem:invasion}. The second is bounded by \cref{lem:initial_domination} applied at each time $s_k$, using the Markov property. For the third, a pair of matched particles at time $s_{k}$ belongs to the same block $H(i)$ and is therefore separated by a distance of at most $\sqrt{d}\, L$. Their relative displacement evolves as a continuous-time random walk with jump rate at least $2\Gamma_-$. We apply \cref{lem:meeting} with spatial separation $R = \sqrt{d} L$. The condition $R \le \sqrt{\ubc{C:meeting_distance} T}$ is satisfied since $L = \lfloor \sqrt{\kappa T} \rfloor$ implies $\sqrt{d} L \le \sqrt{d \kappa T} = \sqrt{\ubc{C:meeting_distance} T}$. The lemma ensures they meet during $[s_k, s_{k+1}]$ with probability at least $\ubc{C:meeting} (\sqrt{d} L)^{-\gamma_d}$. By absorbing the dimension-dependent factor $d^{-\gamma_d / 2}$ into $\ubc{C:meeting}$, we can bound the probability they do not meet by $1 - \ubc{C:meeting} L^{-\gamma_d}$. The strong Markov property then bounds the probability of the third event by
  \begin{equation}
    (1 - \ubc{C:meeting} L^{-\gamma_d})^N \leq \exp\braces{-\ubc{C:meeting} N L^{-\gamma_d}} \leq \exp\braces{-c \dist_V T^{-1 - \gamma_d/2}}.
  \end{equation}

  A union bound over the three contributions above gives \eqref{eq:union_bound_D_c} below. The first term follows from taking a union bound over the $N \le \dist_V$ time steps and the $C(w+\dist_V)^{d-1}$ sites on the boundary of $S$. The second term is obtained by summing the bound of \cref{lem:initial_domination} over the $N$ steps. The final term bounds the probability that at least one of the expected $\rho |\bar{S}| < C \rho (w+\dist_V)^d$ particles fails to meet its match.
  \begin{equation}
  \label{eq:union_bound_D_c}
    \begin{split}
      \P \big( \mathcal{D}^{c} \big) & \leq C \rho \dist_V (w+\dist_V)^{d-1} e^{-c \dist_V} \\
      & \quad + C \frac{\dist_V}{T} \Big(\frac{w+\dist_V}{\sqrt{T}}\Big)^{d} e^{-c T^{d/2}\rho \varepsilon^{2}} \\
      & \quad + \rho (w+\dist_V)^{d} e^{-c \dist_V T^{-1 - \gamma_d/2}}.
    \end{split}
  \end{equation}

  We choose $T$ so that the second and third exponential terms have the same order. Equating $T^{d/2} = \dist_V T^{-1 - \gamma_d/2}$ and solving gives
  \begin{equation}
    T = \dist_V^{\frac{2}{d + 2 + \gamma_d}}.
  \end{equation}
  Writing $\beta = \frac{d}{d + 2 + \gamma_d}$, the common value of both exponents is
  \begin{equation}
    \dist_V T^{-1 - \gamma_d/2} = T^{d/2} = \dist_V^{\beta}.
  \end{equation}

  A direct check gives $\beta \ge 1/3$ in every dimension.  After absorbing the polynomial factors into the constants, we obtain
  \begin{equation}
    \begin{split}
      \P(\mathcal{D}^{c}) &\leq C (w + \dist_V)^{d + 1} \Big(\rho \exp\braces{-c \dist_V^{1/3}} + \exp\braces{-c \rho\varepsilon^2 \dist_V^{1/3}}\Big).
    \end{split}
  \end{equation}

  Recall the definition of the event $D$ in~\eqref{eq:domination_event}. Comparing the estimate above with~\eqref{eq:bad_coupling} yields
  \begin{equation}
    \begin{split}
      \P\big( (\xi, \xi') \in D^{c}) & \leq \ubc{C:vertical_decoupling} (w + \dist_V)^{d + 1} \Big(\rho \exp\braces{-\uc{c:vertical_decoupling} \dist_V^{1/3}} + \exp\braces{-\uc{c:vertical_decoupling} \rho\varepsilon^2 \dist_V^{1/3}}\Big),
    \end{split}
  \end{equation}
  simply adjusting the constants. This completes the proof.
\end{proof}

\begin{remark}\label{rmk:vertical-graphical}
    Note that the above construction only specifies a graphical representation for the particles within the time interval $[0, \dist_V]$, so it is compatible for the slot representation to be used for both $\xi$ and $\xi'$ outside that interval. This shows that the theorem also holds if the functions $f_i$ are allowed to depend on the slot representation within $B_i$, similar to \cref{rmk:horizontal-graphical}.
\end{remark}

\section{Immunity regime}\label{sec:immunity_regime}

In this section, we prove \cref{thm:immunity_low_rho}.
The proof is based on a multi-scale renormalization argument inspired by the framework developed in~\cite{FMUV23, HUVV22} to investigate the renewal contact process and contact processes in dynamic random environments.
For simplicity, in this section we opted to carry out our proofs for spatial dimension $d = 1$, but they generalize straightforwardly to arbitrary $d \ge 1$.

Let us briefly outline the mechanism.
We define a sequence of space-time boxes whose side lengths grow geometrically, and show that the probability of an infection path crossing these boxes (in either space or time) at each scale satisfies a recursive contraction inequality.
This recursion is established by relating the occurrence of a crossing at scale $k+1$ to the coexistence of a pair of crossings within well-separated boxes at scale $k$.
The key tools for upper bounding the probability of these simultaneous crossings are the horizontal and vertical decoupling inequalities of \cref{sec:horizontal_decoupling,sec:improved_vertical_decoupling}, respectively.
Once the recursive inequality is established, it remains to prove that the probability of a half-crossing at the baseline scale is arbitrarily small.
\cref{thm:immunity_low_rho} is then proved by choosing the particle density $\rho$ sufficiently close to zero.

\subsection{Multi-scale geometry and half-crossing events}

We start by introducing a base-scale parameter $\ell_0 \in \N$ and a scaling parameter $\alpha = 8$. This value is simply chosen to be large enough for our geometric arguments to hold, and any larger even integer would suffice. The even restriction guarantees that all boxes will be positioned at integer coordinates.

The constant $\ell_0$ serves as the base height scale, and we define the corresponding base width scale as $w_0 = \ell_0^2$.

Having defined the base-scale dimensions, we define the higher-scale dimensions recursively for $k \ge 0$ by
\begin{equation}\label{eq:length_scales}
      \ell_{k+1} = \alpha \ell_k \qquad\text{ and }\qquad w_{k+1} = \alpha w_k.
\end{equation}
Under these definitions, the temporal and spatial scales grow exponentially fast according to $\ell_k = \alpha^k \ell_0$ and $w_k = \alpha^k w_0$.

The space-time rectangle
\begin{equation}\label{eq:scale_k_box}
  B_k = [-w_k, w_k] \times [0,\ell_k],
\end{equation}
and its translates $(x_0, t_0) + B_k = [x_0 - w_k, x_0 + w_k] \times [t_0, t_0 + \ell_k]$ for $(x_0, t_0) \in \mathbb{Z} \times \mathbb{R}_+$, are referred to as scale-$k$ boxes.
We assume these boxes are subsets of $\mathbb{Z} \times \mathbb{R}$ (discrete in space and continuous in time).
Since $w_k = \ell_k \ell_0$, the aspect ratio of every scale-$k$ box is uniformly given by $2\ell_0$.

\begin{figure}[tpb]
  \centering
  \begin{subfigure}[t]{0.48\textwidth}
    \centering
    \includegraphics[width=0.8\linewidth]{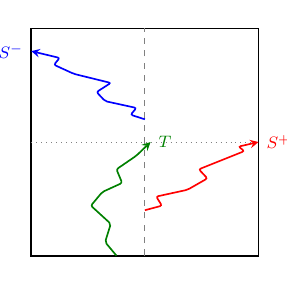}
    \caption{Half-crossing events.}
    \label{fig:half-crossings}
  \end{subfigure}\hfill
  \begin{subfigure}[t]{0.48\textwidth}
    \centering
    \includegraphics[width=0.8\linewidth]{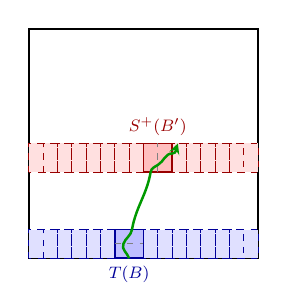}
    \caption{Cascading of half-crossings.}
    \label{fig:cascading}
  \end{subfigure}
  \caption{Illustration of the three half-crossing events (a) and the cascading decomposition (b).}
\end{figure}

We now define the half-crossing events, which were originally introduced in~\cite{FMUV23}.
Throughout this section, these are treated as \emph{bad events}. 
Our objective is to prove that their probabilities decay  fast as the scale index $k$ increases. These events are depicted in \cref{fig:half-crossings}.

\begin{definition}[Half-crossing events]
  \label{def:half_crossing}

  Let $B=[-w,w]\times[s,s+h]$ be a space-time box, and let $B^+=[0,w]\times[s,s+h]$ and $B^-=[-w,0]\times[s,s+h]$ denote its right and left halves, respectively.

  We define the following events with respect to the genealogical paths introduced in \cref{def:jump_path}:

  \begin{description}
    \item[Right spatial half-crossing] $S^+(B)$ is the event that there exists a genealogical path starting from $\{0\}\times[s,s+h]$ that reaches $\{w\}\times[s,s+h]$ while remaining entirely inside $B^+$.

    \item[Left spatial half-crossing] $S^-(B)$ is the event that there exists a genealogical path starting from $\{0\}\times[s,s+h]$ that reaches $\{-w\}\times[s,s+h]$ while remaining entirely inside $B^-$.

    \item[Temporal half-crossing] $T(B)$ is the event that there exists a genealogical path starting from $[-w,w] \times \{s\}$ that reaches $[-w,w] \times \{s+h/2\}$ while remaining entirely inside $B$.
  \end{description}
  The \emph{half-crossing event} of $B$ is defined as the union
  \begin{equation}\label{eq:half_crossing_union}
    H(B) = S^+(B) \cup S^-(B) \cup T(B).
  \end{equation}
\end{definition}

The events $S^\pm(B)$ and $T(B)$ are formulated explicitly in terms of genealogical paths and depend exclusively on the realization of the zero-range process and the Poisson healing/jump clocks inside $B$.
Therefore, $H(B)$ is supported in the box $B$ (see the definition around \cref{eq:support}).

These half-crossing events serve as a natural local proxy for the global survival of the infection, as formalized by the following lemma.

\begin{lemma}[Survival implies half-crossings]
\label{lem:survival_implies_H} 
If the infection initialized at the origin survives past time $\ell_k/2$, then the event $H(B_k)$ must occur.
\end{lemma}

\begin{proof}
  Suppose the infection survives past time $\ell_k/2$.
  Then there exists a genealogical path $\gamma:[0,T] \to \mathbb{Z}$ with $\gamma(0)=0$ and $T > \ell_k/2$.

  If $\gamma$ remains inside the spatial interval $[-w_k,w_k]$ up to time $\ell_k/2$, then the restriction $\gamma|_{[0,\ell_k/2]}$ is a genealogical path starting from $[-w_k,w_k]\times\{0\}$ that is contained entirely within $B_{k}$ and reaches height $\ell_k/2$. This directly witnesses the occurrence of the temporal half-crossing event $T(B_k)$.

  Otherwise, let $t_1 \leq \ell_k/2$ be the first time at which $\gamma$ hits the boundary set $\{-w_k, w_k\}$.
  Assume without loss of generality that this exit occurs through the right boundary, so that $\gamma(t_1) = w_k$. Let $t_0 = \sup\{t \leq t_1 : \gamma(t) = 0\}$ be the last time the path is at the origin prior to time $t_{1}$. 
  The discrete nature of the jump dynamics implies the existence of some $t_0' < t_0$ such that $\gamma(t) = 0$ for all $t \in [t_0', t_0)$. Thus, the restricted path $\gamma|_{[t_0',t_1]}$ constitutes a valid genealogical path that starts at the origin at time $t_0'$, remains within $[0,w_k]$, and reaches $w_k$ at time $t_1$. This guarantees the occurrence of the right spatial half-crossing event $S^+(B_k)$.

  A symmetric argument yields the occurrence of $S^-(B_k)$ if the exit occurs through $-w_k$.
  Hence, in all cases, the event $H(B_k)$ occurs.
\end{proof}

In view of \cref{lem:survival_implies_H}, to prove extinction it suffices to show that the probability of $H(B_k)$ vanishes as $k \to \infty$, justifying why these are viewed as \emph{bad events}.

A crucial step for bounding the probabilities of the half-crossing events from above relies on the decoupling inequalities derived in the next subsection. For this purpose, we will require the following monotonicity property.

\begin{lemma}[Monotonicity of half-crossings]
\label{lem:monotonicity} 
For any space-time box $B$, the event $H(B)$ is non-decreasing in the following sense: if $(\eta_t)_{t \ge 0}$ and $(\bar\eta_t)_{t \ge 0}$ are coupled such that $\eta_t(x) \le \bar\eta_t(x)$ for all $(x, t) \in B$, every jump in $\eta$ also occurs in $\bar\eta$, and both processes share the same configuration of healing marks, then the occurrence of $H(B)$ under $\eta$ implies its occurrence under $\bar\eta$.
\end{lemma}

\begin{proof}
    Since the healing marks are the same in both processes, and $\eta_t(x) \le \bar \eta_t(x)$ for all $(x, t) \in B$, the effective healing marks of the second process are a subset of the effective healing marks of the first process. This implies that any witness for the half-crossing event in $\eta$ is also a witness for the event in $\bar \eta$.
\end{proof}

Another important property that we will exploit to bound the half-crossing probabilities is the so-called cascading property: a half-crossing at scale $k+1$ implies that two distinct half-crossings at scale $k$ occur within well-separated boxes. This is the content of \cref{lem:temporal_cascading} below, which adapts~\cite[Lemma 2.5]{HUVV22} to our current setting.

We begin by establishing this property for temporal half-crossings through the following lemma, which is depicted in \cref{fig:cascading}.

\begin{lemma}[Cascading property for temporal half-crossings]
  \label{lem:temporal_cascading}
  Fix $k \ge 0$.
  There exist collections $\mathcal{B}_0$ and $\mathcal{B}_0'$ of scale-$k$ boxes such that
  \begin{equation}
    T(B_{k+1}) \subset \bigcup_{(B,B') \in \mathcal{B}_0 \times \mathcal{B}_0'} \left( H(B) \cap H(B') \right).
  \end{equation}
  Moreover, $|\mathcal{B}_0| = |\mathcal{B}_0'| = 2\alpha+3$, and the vertical separation between any box in $\mathcal{B}_0$ and any box in $\mathcal{B}_0'$ is at least $\ell_{k+1}/4$.
\end{lemma}

\begin{proof}
  The event $T(B_{k+1})$ requires a genealogical path $\gamma: [0, \ell_{k + 1}/2] \to \Z$ such that $\gamma(t) \in I_{k + 1} = [-w_{k+1}, w_{k+1}]$ for all $t$. In particular, $\gamma$ is defined both
  during the initial interval $[0, \ell_k]$ and during the later interval $[\ell_{k+1}/2 - \ell_k, \ell_{k+1}/2]$. The vertical gap between these two intervals is
  \begin{equation}
    \frac{\ell_{k+1}}{2} - 2\ell_k = \Bigl(\frac{\alpha}{2}-2\Bigr) \ell_k \geq \frac{\alpha}{4} \ell_{k} = \frac{1}{4}\ell_{k+1},
  \end{equation}
  since $\alpha = 8$.

  We now show that the path produces a half-crossing of a scale-$k$ box during each time interval, by applying an argument similar to the one in \cref{lem:survival_implies_H}.

  We cover $I_{k+1}$ with the $2\alpha + 3$ windows $W_j = [(j-1)w_k, (j + 1)w_k]$ for $j = -\alpha - 1, -\alpha, \ldots, \alpha + 1$. Each $W_j$ has width $2w_k$, consecutive windows are spaced $w_k$ apart, and $\bigcup_j W_j \supset I_{k + 1}$.

  Consider the time interval $[0, \ell_k]$. In case $\gamma|_{[0,\ell_k]}$ does not intersect 4 distinct $W_j$, then $\gamma|_{[0, \ell_k]} \subset [(j - 1)w_k, (j + 1)w_k]$ for some $j$, and its trace lies entirely within $W_j\times[0,\ell_k]$, providing a temporal half-crossing $T(B)$ for the scale-$k$ box $B=W_j\times[0,\ell_k]$.

  If instead it intersects four distinct windows $W_{j}, W_{j+1}, W_{j + 2}, W_{j + 3}$, let $t_1 \leq \ell_k$ be the smallest time for which $\gamma|_{[0, t_1]}$ still does so. Observe that in this event $\gamma|_{[0, t_1]}$ must cross at least one of the intervals $[jw_k, (j + 1)w_k]$ and $[(j + 1)w_k, (j+2)w_k]$. Suppose it crosses from $(j+1)w_k$ to $(j+2)w_k$ and let $t_0 = \sup\{t \leq t_1: \gamma(t) = (j+1)w_k\}$. The discrete nature of the jumps implies there exists some time interval $[t_0', t_0)$ for which $\gamma(t) = (j+1)w_k$ for all $t \in [t_0', t_0]$. The restriction $\gamma|_{[t_0',t_1]}$ starts at the centerline $(j+1)w_k$ of $B=W_{j+1} \times [0,\ell_k]$ at time $t_0'$ and reaches the right boundary $(j+2)w_k$ at time $t_1$ while remaining in the right half $[(j+1)w_k,(j+2)w_k]$, yielding $S^+(B)$. The symmetric case gives $S^-(B)$. In all cases $H(B)$ occurs for some $B$ in the collection $\mathcal{B}_0 = \{W_j \times [0, \ell_k] : j = -\alpha - 1, \ldots, \alpha + 1\}$.

  An identical argument applied to the time interval $[\ell_{k+1}/2 - \ell_k, \ell_{k+1}/2]$ produces a
  second collection $\mathcal{B}_0'$. Every pair $(B, B')$ with $B \in \mathcal{B}_0$ and $B' \in
  \mathcal{B}_0'$ has vertical separation at least $\ell_{k+1}/2 - 2\ell_k \geq \ell_{k + 1}/4$, and both
  collections have cardinality $2\alpha + 3$, proving the lemma.
\end{proof}

\begin{lemma}[Cascading property for spatial half-crossings]
  \label{lem:spatial_cascading}
  Fix $k \ge 0$.
  There exist collections
  $\mathcal{B}_1$ and $\mathcal{B}_1'$ of scale-$k$ boxes such that
  \begin{equation}
    S^+(B_{k+1}) \subset \bigcup_{(B, B') \in \mathcal{B}_1 \times \mathcal{B}_1'} H(B)
    \cap H(B').
  \end{equation}
  Moreover, $|\mathcal{B}_1| = |\mathcal{B}_1'| = 2\alpha + 3$, and any pair of boxes in $\mathcal{B}_1$ and $\mathcal{B}_1'$ have horizontal separation at least $w_{k+1} / 2$. An analogous statement holds for $S^-(B_{k+1})$ with collections $\mathcal{B}_2$ and $\mathcal{B}_2'$.
\end{lemma}

\begin{proof}
  On $S^+(B_{k+1})$, a genealogical path $\gamma$ crosses the right half $B_{k+1}^+=[0,
  w_{k+1}] \times [0, \ell_{k+1}]$ from left to right. In particular, $\gamma$ crosses both the left
  strip $[0, 2w_k] \times [0, \ell_{k+1}]$ and the right strip $[w_{k+1}-2w_k, w_{k+1}] \times [0,
  \ell_{k+1}]$. The horizontal distance between these strips is $w_{k+1} - 4w_k \geq w_{k+1}/2$, and a reasoning analogous to the previous lemma proves the statement.
\end{proof}

\cref{lem:temporal_cascading,lem:spatial_cascading} imply that on each of the half-crossing events $S^+(B_{k+1})$, $S^-(B_{k+1})$, and $ T(B_{k+1})$ that entail the occurrence of $H(B_{k+1})$ one can find a pair of well-separated $k$-scale boxes where half-crossings take place.
The number of choices for such a pair is at most $(2\alpha + 3)^2$ for each one of these $3$ half-crossing events, accounting for a total of $3(2\alpha + 3)^2$ choices.
The following lemma records the result where the entropy constant $\ubc{C:entropy}=\ubc{C:entropy}(\alpha)$ captures the total number of candidate pairs.

\newbigconstant{C:entropy}
\begin{lemma}[Cascading half-crossings]
  \label{lem:cascading}
  For any $k \geq 0$,
  \begin{equation}\label{eq:cascading_union}
    H(B_{k+1}) \subset \bigcup_{j = 0}^{2} \bigcup_{(B, B') \in \mathcal{B}_j \times
    \mathcal{B}_j'} H(B) \cap H(B').
  \end{equation}
  In words, whenever $H(B_{k+1})$ occurs, there exists a pair of scale-$k$ boxes $B,B'$ that are separated vertically
  by at least $\ell_{k+1}/4$ or horizontally by at least $w_{k+1}/2$ such that $H(B)$ and $H(B')$ occur.
  The total number of candidate pairs of boxes $B$, $B'$ is at most
  \begin{equation}\label{eq:entropy_bound}
    \ubc{C:entropy} = 3(2\alpha + 3)^2.
  \end{equation}
\end{lemma}

\subsection{Sprinkling and recursive estimates}

We now derive the recursive inequality for the half-crossing probabilities. To do so, we bound from above the probability of two joint half-crossing events in well-separated boxes by employing decoupling inequalities. 
When these boxes are separated horizontally, we apply horizontal decoupling (\cref{thm:horizontal_decoupling}), whereas for vertically separated boxes, we use vertical decoupling (\cref{thm:vertical_decoupling}). 
This latter step requires a slight increase in the density---a sprinkling procedure---when descending from scale $k+1$ to scale $k$, which is instrumental in achieving the decoupling. 
Indeed, the covariance of the pair of half-crossings decays much slower than needed here. This slow decay stems from the conservation of particles, as a fraction of them can potentially visit both boxes and introduce long-range dependencies.

Fix the sprinkling parameter $s = 1/10$ and the limit density $\rho_\infty = \ell_0^{-s}$. Define the sprinkling increments
\begin{equation}
\label{eq:sprinkling}
  \delta_k = \ell_k^{-s}, \qquad k \geq 0.
\end{equation}
Since $\ell_k=\alpha^k \ell_0$ grows exponentially, the series $\sum_{k=0}^\infty \delta_k$ converges.
We define the \emph{decreasing} sequence of densities $(\rho_k)_{k=0,\ldots,\infty}$ as
\begin{equation}\label{eq:density_sequence}
  \rho_k = \rho_\infty + \sum_{j = k}^\infty \delta_j.
\end{equation}
In particular, the initial density is given by
\newbigconstant{C:rho0}
\begin{equation}\label{eq:rho0_bound}
  \rho_0  = \rho_\infty + \ubc{C:rho0} \ell_0^{-s} = \Bigl(1 +\frac{1}{1 - \alpha^{-s}}\Bigr) \ell_0^{-s}.
\end{equation}

For each $k \geq 0$, define the half-crossing probability at density $\rho_k$ by
\begin{equation}\label{eq:pk_def}
  p_k = \P_{\rho_k} \big( H(B_k) \big).
\end{equation}

We wish to show that the sequence $p_k$ vanishes as $k$ increases.
The first step towards that goal is to obtain the following recursive inequality.
\newconstant{c:eps_decay}
\newbigconstant{C:eps_decay}
\begin{proposition}[Recursive inequality]
  \label{prop:recursive} For each $k \ge 0$
  \begin{equation}
    \label{eq:recursive}
    p_{k+1} \le \ubc{C:entropy} p_k^2 + \ubc{C:entropy} \varepsilon_k,
  \end{equation}
  where the decoupling error $\varepsilon_k$ is given by
  \begin{equation}\label{eq:eps_decay}
    \varepsilon_k = \ubc{C:eps_decay} \ell_k^4 \Big( \rho_0 \exp\braces{-\uc{c:eps_decay} \ell_k^{1/3}} + \exp\braces{-\uc{c:eps_decay} \rho_0^{-1} \ell_k^{1/3 - 2s}} \Big)
  \end{equation}
  for constants $\ubc{C:eps_decay} = \ubc{C:eps_decay}(s, \Gamma_+, \Gamma_-)$ and $\uc{c:eps_decay} = \uc{c:eps_decay}  (s, \Gamma_+, \Gamma_-)>0$ depending only on the parameters indicated.
\end{proposition}

\begin{proof}
  Using \cref{lem:cascading} and a union bound, we get
  \begin{equation}
    \label{eq:union_bound_cascading}
    p_{k+1} = \P_{\rho_{k+1}}\big( H(B_{k+1}) \big)
    \le \ubc{C:entropy} \max_{(B, B')} \P_{\rho_{k+1}} \big( H(B) \cap H(B') \big),
  \end{equation}
  where the maximum ranges over the collection of at most $\ubc{C:entropy}$ pairs of scale-$k$ boxes $(B, B')$ separated vertically by at least distance $\ell_{k+1} / 4$ or horizontally by at least distance $w_{k+1} / 2$.
  Since $H(B)$ is non-decreasing (see \cref{lem:monotonicity}) and $\mu_\rho$ is stochastically increasing in $\rho$, the monotone coupling gives $\P_{\rho_{k+1}}(H(B)) \leq \P_{\rho_k}(H(B)) = p_k$ for every scale-$k$ box $B$.

  We now bound the right-hand side of \eqref{eq:union_bound_cascading} by dividing it into two cases depending on the separation between $B$ and $B'$.

  For pairs that are vertically well-separated, we apply \cref{thm:vertical_decoupling}.
  The theorem requires a multiplicative sprinkling increment, so we set $\varepsilon = \delta_k / \rho_{k+1}$, which gives $\rho_{k+1}(1 + \varepsilon) = \rho_{k+1} + \delta_k = \rho_k$.
  To bound the error terms, we note that $\rho_{k+1} \le \rho_0 \le \rho_\infty + \ubc{C:rho0}$, uniformly over all choices of $k$ and $\ell_0$.
  Applying \cref{thm:vertical_decoupling} with the spatial width parameter $w = 2w_k$ and vertical distance $\ell_{k+1}/4 \leq \dist_V \leq \ell_{k+1}$, we obtain
  \begin{equation}
    \P_{\rho_{k+1}} \big( H(B)\cap H(B') \big) \leq p_k^2 + \varepsilon_k^{(V)},
  \end{equation}
  where, since $\delta_k = \ell_k^{-s}$ and $\ell_{k+1} = \alpha \ell_k$,
  \begin{equation}
    \begin{aligned}\label{eq:eps_V}
      \varepsilon_k^{(V)} &= \ubc{C:vertical_decoupling} (2w_k + \ell_{k+1})^2 \Big( \rho_{k+1} \exp\braces{-\uc{c:vertical_decoupling} (\ell_{k+1}/4)^{1/3}} \\
      &\hspace{100pt}+ \exp\braces{-\uc{c:vertical_decoupling} \rho_{k+1} \varepsilon^2 (\ell_{k+1}/4)^{1/3}} \Big) \\
      &\le C \ell_k^4 \Big( \rho_0 \exp\braces{-c \ell_k^{1/3}} + \exp\braces{-c \rho_{k+1} (\delta_k/\rho_{k+1})^2 \ell_k^{1/3}} \Big) \\
      &\le C \ell_k^4 \Big( \rho_0 \exp\braces{-c \ell_k^{1/3}} + \exp\braces{-c\, \rho_0^{-1} \ell_k^{1/3-2s}} \Big).
    \end{aligned}
  \end{equation}

  For horizontally separated pairs, the horizontal distance satisfies $\dist_H\ge w_{k+1}/2$ and the temporal projections lie in an interval of length at most $\ell_{k+1}$. Applying \cref{thm:horizontal_decoupling} gives
  \begin{equation}
    \P_{\rho_{k+1}} \big( H(B)\cap H(B') \big) \leq p_k^2 + \varepsilon_k^{(H)},
  \end{equation}
  where, since $\rho_{k+1} \le \rho_0$ and $w_{k+1} = \alpha \ell_0 \ell_k \ge \ell_k$,
  \begin{equation}\label{eq:eps_H}
    \begin{aligned}
      \varepsilon_k^{(H)} &= \ubc{C:horiz_decoupling} \rho_{k+1} \sqrt{\ell_{k+1}} \exp\braces{-\uc{c:horiz_decoupling} (w_{k+1}/2) \log(1+w_{k+1}/(2\ell_{k+1}))} \\
      &\le C \rho_0 \ell_k^{1/2} \exp\braces{-c \ell_k}.
    \end{aligned}
  \end{equation}

  Finally, we can define constants $\ubc{C:eps_decay}, \uc{c:eps_decay} > 0$ so that the sum $\varepsilon_k^{(V)} + \varepsilon_k^{(H)}$ is bounded by $\varepsilon_k$ as defined in \eqref{eq:eps_decay}. This is possible since the exponent $\ell_k$ in the horizontal error grows much faster than $\ell_k^{1/3}$, meaning the horizontal error is easily absorbed into the first term of $\varepsilon_k^{(V)}$.
\end{proof}

\begin{lemma}[Contraction]\label{lem:contraction} Suppose $p_0 \leq 1/(4 \ubc{C:entropy})$ and
  $\ubc{C:entropy}^2 \varepsilon_k \leq (1/2)^{k + 4}$ for all $k \geq 0$. Then
  $p_k \leq (1/2)^{k + 2}/\ubc{C:entropy}$ for every $k \geq 0$.
\end{lemma}

\begin{proof}
  Set $q_k = \ubc{C:entropy} p_k$. Multiplying \eqref{eq:recursive} by $\ubc{C:entropy}$ gives
  $q_{k+1} \leq q_k^2 + \ubc{C:entropy}^2 \varepsilon_k$. The hypothesis $p_0 \leq 1/(4
  \ubc{C:entropy})$
  gives $q_0 \leq 1/4 = (1/2)^2$. Suppose inductively that $q_k \leq (1/2)^{k + 2}$. Then
  $q_k^2 \le (1/2)^{2k + 4} \leq (1/2)^{k + 4}$, and together with
  $\ubc{C:entropy}^2 \varepsilon_k \leq (1/2)^{k + 4}$ we obtain
  \begin{equation}
    q_{k+1} \leq (1/2)^{k + 4} + (1/2)^{k + 4} = (1/2)^{k + 3}.
  \end{equation}
  This concludes the proof by induction on $k$.
\end{proof}

\subsection{Local cluster structure and the triggering event for small densities}

In the induction procedure from \cref{lem:contraction}  the assumption that the probability of a half-crossing event in the base-scale $p_0$ is small was given as an input.
To complete the proof that the sequence $p_k$ contracts exponentially fast in $k$, we now need to show that $p_0$ can indeed be taken small by driving the density at the base-scale to zero.
This is however a delicate procedure due to the fact that particles are needed in order for the sprinkling in the vertical decoupling to be carried out.
In other words, on one hand we would like to have very few particles, but on the other hand we would like to have enough of them to be able to perform the sprinkling.

Our argument relies on understanding the initial particle configuration under $\mu_{\rho_{0}}$.
For controlling these configurations we introduce the cluster structure and the triggering event.
We start by fixing some parameters whose values will be specified during the proof of \cref{thm:immunity_low_rho} in \cref{sec:proof_immunity_low_rho}:
\begin{itemize}
   \item $\beta \in (0, 1)$, a connectivity exponent for the cluster;
  \item $m \in \mathbb{N}$, a cluster capacity threshold;
\end{itemize}
\begin{definition}[Clusters and triggering event]\label{def:clusters_good_event}
  Given an initial configuration $\eta_0$, we partition the occupied sites into \emph{clusters}. 
  Two occupied sites $x$ and $x'$ belong to the same cluster if they can be connected by a chain of occupied sites $x=x_0< x_1 < \cdots < x_k=x'$ where consecutive
  sites are at distance at most $\ell_0^\beta$, that is $|x_{i+1}-x_{i}| < \ell_0^{\beta}$.

  Let $I_{\mathrm{ext}} = [-4w_0, 4w_0]$ be the extended base-scale box and let $I_{\mathrm{in}} = [-2w_0, 2w_0]$.
  The \emph{triggering event} is the intersection of four events
  \begin{equation}
    \T = \Ecluster \cap \Eout \cap \Ein \cap \Eheal,
  \end{equation}
  where:
  \begin{itemize}
    \item $\Ecluster$: every cluster intersecting $I_{\mathrm{ext}}$ contains strictly less than $m$
      particles;
    \item $\Eout$: no particle starting outside $I_{\mathrm{ext}}$ visits $I_{\mathrm{in}}$ during the time interval
      $[0, \ell_0^{\beta/2}]$;
    \item $\Ein$: every particle initially in $I_{\mathrm{ext}}$ displaces by at most $\ell_0^{\beta/2}$ from its
      starting position during $[0, \ell_0^{\beta/2}]$;
    \item $\Eheal$: every particle in any cluster initially intersecting $[-w_0, w_0]$ fully heals by time
      $\ell_0^{\beta/2}$.
  \end{itemize}
\end{definition}

\begin{figure}[tpb]
  \centering
  \includegraphics[width=\linewidth]{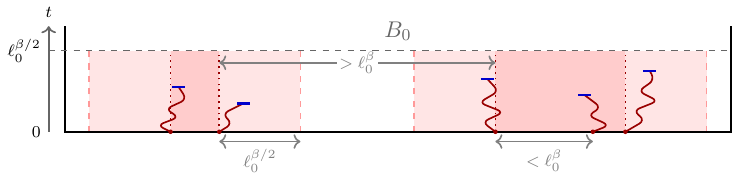}
  \caption{Cluster structure and the triggering event at the base-scale.}
  \label{fig:cluster-healing}
\end{figure}

\begin{remark}\label{rmk:triggering-event}
The event $\T$ (depicted in \cref{fig:cluster-healing}) is contained in the event $T(B_0)^c$: on $\T$, all particles in any cluster initially intersecting $[-w_0, w_0]$ heal by time $T = \ell_0^{\beta/2} < \ell_0/2$. Moreover, $\Ein$ ensures that no particle in these clusters leaves $I_{\mathrm{in}}$ or shares a site with a particle from another cluster that intersects $I_{\mathrm{ext}}$ during $[0, T]$, while $\Eout$ prevents any particle from outside $I_{\mathrm{ext}}$ entering $I_{\mathrm{in}}$. Any genealogical path starting from $[-w_0,w_0]\times\{0\}$ is therefore completely isolated and killed before time $T$.
\end{remark}

The previous remark implies that bounding $\T^c$ is sufficient for bounding $T(B_0)$. This will be done over the following lemmas.  

\newbigconstant{C:big_cluster}
\begin{lemma}[Small clusters]\label{lem:big_cluster} There exists a constant $\ubc{C:big_cluster}>0$
  such that
  \begin{equation}\label{eq:big_cluster}
    \P_{\rho_0} \big( \Ecluster^c \big)
    \le \ubc{C:big_cluster} \ell_0^{2 + (m - 1)\beta - ms}.
  \end{equation}
\end{lemma}

\begin{proof}
  On the event $\Ecluster^c$, some cluster intersecting $I_{\mathrm{ext}} = [-4w_0, 4w_0]$ has at least $m$ particles.
  This requires $m$ particles $x_1 \le \cdots \le x_m$ with $x_{i + 1} - x_i \le \ell_0^\beta$ for each
  $1 \le i \le m - 1$, so their span is at most $m \ell_0^\beta$. Since the cluster intersects $I_{\mathrm{ext}}$,
  these $m$ particles must lie within the enlarged interval $I = [-4w_0 - m \ell_0^\beta, 4w_0 + m \ell_0^\beta]$.

  Set $w = 2m \ell_0^\beta$ and cover $I$ with overlapping intervals $W_j$ of length $w$, staggered by
  $w/2$. The number of such intervals is $N_w \le 2|I|/w + 2 \le C \ell_0^{2 - \beta}$. Any set of
  diameter at most $w/2$ is contained in at least one $W_j$, so $\Ecluster^c$ implies that some
  $W_j$ contains at least $m$ particles.

  By \cref{lem:poisson_domination}, the particle count in $W_j$ is dominated by a Poisson random
  variable $Z$ with mean $\lambda = \Gamma \rho_0 w \le C \ell_0^{\beta - s}m$. The tail bound
  $\P(Z \ge m) \le \lambda^m/m! \le C \ell_0^{m\beta - ms}$ and a union bound over the $N_w$ intervals
  gives
  \begin{equation}
    \P_{\rho_0} \big( \Ecluster^c \big)
    \le N_w \cdot C \ell_0^{m\beta - ms}
    \le \ubc{C:big_cluster} \ell_0^{2 + (m - 1)\beta - ms}. \qedhere
  \end{equation}
\end{proof}

The next lemma ensures that particles from outside the extended interval do not interfere with the
local dynamics.

\newbigconstant{C:outside}
\begin{lemma}[Isolation from the outside]\label{lem:outside_particles} There exists a constant
  $\ubc{C:outside}$ such that
  \begin{equation}\label{eq:outside_bound}
    \P_{\rho_0} \big( \Eout^c \big)
    \le \ubc{C:outside} \rho_0 \ell_0^{\beta/4} e^{-\uc{c:crossing} \ell_0^2}.
  \end{equation}
\end{lemma}

\begin{proof}
  The event $\Eout^c$ requires that some particle initially outside $I_{\mathrm{ext}} = [-4w_0, 4w_0]$ reaches $I_{\mathrm{in}} = [-2w_0, 2w_0]$ by time $T = \ell_0^{\beta/2}$. We consider the contribution from the right half-line $(4w_0, \infty)$. A particle starting at a site $y > 4w_0$ must travel a distance of at least $y - 2w_0 \ge 2w_0$ to reach $2w_0$.

  Applying \cref{lem:crossing_bound} with $s = 2w_0 = 2\ell_0^2$ bounds the crossing probability:
  \begin{equation}
    \P_{\rho_0} \Big(
      \begin{array}{c} \text{some particle from } (4w_0, \infty) \\ \text{ reaches } 2w_0 \text{ before time } T
    \end{array} \Big)
    \le \ubc{C:crossing} \rho_0\sqrt{T} e^{-\uc{c:crossing} 2\ell_0^2\log(1+2\ell_0^2/T)}.
  \end{equation}

  A symmetric reasoning applies to the probability of a particle from $(-\infty, -4w_0)$ reaching $-2w_0$. Adding both contributions and simplifying the exponent yields the claimed bound.
\end{proof}

\newbigconstant{C:confinement}
\begin{lemma}[Particle confinement]\label{lem:particle_confinement}
  There exists a constant $\ubc{C:confinement}>0$ such that
  \begin{equation}\label{eq:confinement_bound}
    \P_{\rho_0} \big(\Ein^c)
    \le \ubc{C:confinement} \rho_0 \ell_0^{2 + \beta/4} e^{-\uc{c:displacement} \ell_0^{\beta/2}}.
  \end{equation}
\end{lemma}

\begin{proof}
  We condition on $\eta_0$, the number of particles initially in $I_{\mathrm{ext}} = [-4w_0, 4w_0]$ is $N = \sum_{x \in I_{\mathrm{ext}}} \eta_0(x)$. The event $\Ein^c$ requires at least one particle initially in $I_{\mathrm{ext}}$ to displace by more than $T = \ell_0^{\beta/2}$ during $[0, T]$. By a union bound and the displacement estimate
  \eqref{eq:displacement_bound} applied with $s = T$ and time horizon $T$,
  \begin{equation}
    \P^{\eta_0} \big( \Ein^c \big) \le N \ubc{C:displacement} \sqrt{T} e^{-\uc{c:displacement} T}.
  \end{equation}
  Taking expectations gives
  \begin{equation}
    \P_{\rho_0} \big( \Ein^c \big)
    \le \ubc{C:displacement} \sqrt{T} e^{-\uc{c:displacement} T} \E_{\rho_0} \big[N \big] \le \ubc{C:confinement} \rho_0 \ell_0^{2 + \beta/4} e^{-\uc{c:displacement} \ell_0^{\beta/2}},
  \end{equation}
  concluding the proof.
\end{proof}

\newconstant{c:single_cluster}
\begin{lemma}[Local cluster healing]\label{lem:local_cluster_healing}
  Consider a zero-range process on $\Z$ starting from an initial configuration $\eta_0$ with $k \le m$ particles. Assume all particles are infected at time 0. Then there exists a constant $\uc{c:single_cluster}>0$, depending only on $m$, $\Gamma_\pm$, and the healing rate $\delta > 0$, such that, for all $T \geq 1$,
  \begin{equation}
    \P^{\eta_0}\big( \text{infection survives up to time $T$} \big) \le e^{-\uc{c:single_cluster} T}.
  \end{equation}
\end{lemma}

\begin{proof}
  We first show that from any initial configuration of $k \le m$ particles, the probability that the
  infection completely heals during a unit time interval is uniformly bounded from below.

  For any initial configuration with $k$ particles, we can construct a target configuration in which all $k$ particles occupy distinct sites and such that this configuration can be reached from the original one with a sequence of at most $r \le m^2$ specific nearest-neighbor jumps. Indeed, one can move each particle sequentially to an isolated site by performing at most $k - 1$ steps per particle.

  If $r = 0$, the particles are already at distinct sites and we simply require no jumps during $[0, 1/2]$, which has probability at least $e^{-m \Gamma_+/2}$.

  Assume now $r \ge 1$. In the slot representation, each particle is individually tracked and follows a Poisson
  trajectory. The rate at which a given particle jumps is bounded from below by $\Gamma_-$ and from
  above by $\Gamma_+$. We divide the time interval $[0,1/2]$ into $r$ equal sub-intervals of length
  $\Delta = \frac{1}{2r}$. Let $\mathcal{A}_{\mathrm{jump}}$ be the event that during the $j$-th
  sub-interval, the $j$-th required jump is performed by the designated particle, and
  no other particles jump.

  In any sub-interval of length $\Delta$, the designated particle jumps exactly once with
  probability at least $\Gamma_- \Delta e^{-\Gamma_+ \Delta}$. Given a jump, the particle moves
  in the required direction with probability $1/2$. Thus, the prescribed move occurs with
  probability at least $\frac{1}{2} \Gamma_- \Delta e^{-\Gamma_+ \Delta}$. The remaining $k - 1$
  particles do not jump at all with probability at least $e^{-(k - 1) \Gamma_+ \Delta}$. Multiplying
  these gives a conditional lower bound for the required outcome in each sub-interval:
  \begin{equation}
    \frac{1}{2} \Gamma_- \Delta e^{-k \Gamma_+ \Delta} \ge \frac{\Gamma_-}{4r}
    e^{-m \Gamma_+/(2r)}.
  \end{equation}
  By the independence of increments of the Poisson processes, the probability of the full successful
  sequence over the $r$ sub-intervals is at least
  \newcommand{\Ajump}{\mathcal{A}_{\mathrm{jump}}}
  \newcommand{\Aheal}{\mathcal{A}_{\mathrm{heal}}}
  \begin{equation}\label{eq:jump-occupied}
    \P(\Ajump)
    \ge \Big( \frac{\Gamma_-}{4r} e^{-m \Gamma_+/(2r)} \Big)^r
    = \Big(\frac{\Gamma_-}{4r}\Big)^r e^{-m \Gamma_+/2}.
  \end{equation}
  Taking the minimum of \cref{eq:jump-occupied} over $1 \le r \le m^2$ along with the previously obtained bound for $r = 0$, this probability is uniformly lower-bounded by a constant $p_{\mathrm{jump}} > 0$ depending only on $m$ and $\Gamma_\pm$.

  On $\Ajump$, the $k$ particles occupy mutually distinct sites at time $1/2$, so the healing
  mechanism is active.

  Let $\Aheal$ be the event that, during $[1/2, 1]$, no particle jumps and every healing clock
  associated with the $k$ occupied sites rings at least once. The probability of no jumps in $[1/2,
  1]$ is at least $e^{-m \Gamma_+/2}$. Independently, each healing clock rings with probability $1 -
  e^{-\delta/2}$. Thus,
  \begin{equation}
    \P( \Aheal \mid \Ajump)
    \ge e^{-m \Gamma_+/2} \big( 1 - e^{-\delta/2} \big)^m = p_{\mathrm{heal}} > 0.
  \end{equation}
  On $\Ajump \cap \Aheal$, all particles heal by time $1$. Let $p = p_{\mathrm{jump}} \cdot
  p_{\mathrm{heal}} > 0$. This uniform lower bound ensures that the probability of the infection
  surviving a unit interval is at most $1 - p$.

  By the Markov property, the probability that the infection survives up to time $T$ requires it to
  survive across $\lfloor T \rfloor$ consecutive unit intervals, each bounded conditionally by $1 -
  p$. Therefore, since $T \geq 1$,
  \begin{equation}
    \P\big( \text{infection survives up to time } T\ \big) \le (1 - p)^{\lfloor T \rfloor} \le e^{-\frac{p}{2} T}.\qedhere
  \end{equation}
\end{proof}

With the local healing mechanism understood, we bound the probability that healing fails for some
cluster.

\newconstant{c:healing_decay}
\begin{lemma}[Healing under the triggering event]\label{lem:healing_good_event} There exists a
  constant $\uc{c:healing_decay} > 0$ such that for all sufficiently large $\ell_0$,
  \begin{equation}\label{eq:healing_good_event}
    \P_{\rho_0}\big( \Eheal^c \cap \Ein \cap \Eout \cap \Ecluster\big) \le e^{-\uc{c:healing_decay} \ell_0^{\beta/2}}.
  \end{equation}
\end{lemma}

\begin{proof}
  As in the previous proofs, set $T=\ell_0^{\beta/2}$. We condition on the initial configuration
  $\eta_0 \in \Ecluster$. Let $\mathcal{C}$ be the set of clusters that initially intersect $[-w_0,
  w_0]$. Since $[-w_0, w_0] \subset I_{\mathrm{ext}}$, under $\Ecluster$, each cluster $K \in \mathcal{C}$ has size at
  most $m-1$ and is initially contained in $I_0 = [-w_0 - m\ell_0^\beta, w_0 + m\ell_0^\beta]
  \subset I_{\mathrm{in}} = [-2w_0, 2w_0]$. Furthermore, these clusters are separated from each
  other by gaps strictly greater than $\ell_0^\beta$.

  For each cluster $K \in \mathcal{C}$, let $\Eheal^K$ be the event that cluster $K$ heals by time
  $T$. A union bound gives
  \begin{equation}
    \P^{\eta_0}\big( \Eheal^c \cap \Ein \cap \Eout\big) \le \sum_{K \in \mathcal{C}}
    \P^{\eta_0}\big( (\Eheal^K)^c \cap \Ein \cap \Eout\big).
  \end{equation}
  On $\Ein \cap \Eout$, all particles in $K$ never share a site with particles outside of $K$ up to time $T$. This holds because internal inter-cluster gaps exceed $\ell_0^\beta$, particles in $I_{\mathrm{ext}}$ displace by at most $T < \ell_0^\beta/2$, and particles from outside $I_{\mathrm{ext}}$ are prevented by $\Eout$ from ever reaching $K \subset I_{\mathrm{in}}$. Consequently, the evolution of $K$ coincides with the \emph{isolated process} $\eta_0^K$ where only the particles of cluster $K$ are present. Then
  \begin{equation}
    \P^{\eta_0}\big((\Eheal^K)^c \cap \Ein \cap \Eout \big)
    \le \P^{\eta_0^K}\big( (\Eheal^K)^c \big).
  \end{equation}
  Since $K$ has fewer than $m$ particles, \cref{lem:local_cluster_healing} bounds this isolated
  failure probability by $e^{-\uc{c:single_cluster} T}$. Summing over the number of clusters $|\mathcal{C}|$ gives
  \begin{equation}
    \P^{\eta_0} \big( \Eheal^c \cap \Ein \cap \Eout \big) \le |\mathcal{C}| e^{-\uc{c:single_cluster} T}.
  \end{equation}
  Multiplying by $\charf{\Ecluster}$ and integrating over the initial measure $\mu_{\rho_0}$ gives
  \begin{equation}
    \P_{\rho_0} \big( \Eheal^c \cap \Ein \cap \Eout \cap \Ecluster \big)
    \le \E_{\rho_0}[|\mathcal{C}| \charf{\Ecluster}] e^{-\uc{c:single_cluster} T}.
  \end{equation}
  Since all clusters in $\mathcal{C}$ are contained in $I_0$, the number of such clusters is bounded by the number of particles in this interval. Thus
  $\E_{\rho_0}[|\mathcal{C}| \charf{\Ecluster}] \le 3 \rho_0 \ell_0^{2}$,
  and~\eqref{eq:healing_good_event} follows for $\ell_{0}$ large enough.
\end{proof}

\subsection{Proof of Theorem~\ref{thm:immunity_low_rho}}
\label{sec:proof_immunity_low_rho}

We now prove extinction for fixed healing rate $\delta$ and sufficiently low density $\rho$, \cref{thm:immunity_low_rho}.

\begin{proof}[Proof of \cref{thm:immunity_low_rho}]
  Recall $s = \frac{1}{10}$. We start by fixing the values for the parameters $\beta$ and $m$ in a way to satisfy the constraint $2 + (m - 1)\beta - ms < 0$ required so that the exponent in \cref{lem:big_cluster} is negative.
  The choice
  \begin{equation}
    \text{$m=40$ and $\beta=1/40$}
  \end{equation}
  satisfies this condition.
  The parameter $\ell_0$ remains free and will be chosen below.

  We first bound the temporal half-crossing probability using \cref{rmk:triggering-event}, so that $\P_{\rho_0} (T(B_0)) \le \P_{\rho_0} (\T^c)$. Decompose the complement $\T^c$ as
  \begin{equation}
    \begin{split}
      \T^c = \Ecluster^c \cup \Eout^c \cup \Ein^c
      \cup \bigl(\Eheal^c \cap \Ein \cap \Eout \cap \Ecluster\bigr).
    \end{split}
  \end{equation}
  By Lemmas~\ref{lem:big_cluster},~\ref{lem:outside_particles},~\ref{lem:particle_confinement}, and~\ref{lem:healing_good_event}, we
  have
  \begin{equation}
  \begin{aligned}
    \P_{\rho_0}(T(B_0)) \le \ubc{C:big_cluster} \ell_0^{2 + (m - 1)\beta - ms}
    + \ubc{C:outside} \rho_0 \ell_0^{\beta/4} e^{-\uc{c:crossing} \ell_0^2} \\
    + \ubc{C:confinement} \rho_0 \ell_0^{2+\beta/4} e^{-\uc{c:displacement} \ell_0^{\beta/2}}
    + e^{-\uc{c:healing_decay} \ell_0^{\beta/2}}.
  \end{aligned}
  \end{equation}
  Since $\rho_0 = \ubc{C:rho0} \ell_0^{-s}$, and our parameter choices ensure that the first exponent of $\ell_0$ is strictly negative while the remaining terms decay super-polynomially, we have $\P_{\rho_0}(T(B_0)) \to 0$ as $\ell_0$ grows.

  To bound the probability of the spatial half-crossings $S^\pm(B_0)$, we use a path-counting
  argument for the number of $J$-paths over the half-box $B_0^+ = [0, w_0] \times [0, \ell_0]$. Note that every genealogical path is also a $J$-path (see \cref{def:weak_jump_path}).

  For any fixed site $x \in [0, w_0]$, we first bound the maximum number of particles up to time
  $\ell_0$ using the occupation bounds of \cref{lem:occupation_bounds}:
  \begin{equation}
    \P_{\rho_0} \Big( \sup_{t \in [0, \ell_0]} \eta_t(x) \ge \log^2 \ell_0 \Big)
    \le (\ell_0 + 1) \Big[ \P_{\rho_0}\big( \eta_0(x) \ge \uc{c:occupation} \log^2 \ell_0 \big) + e^{-\uc{c:occupation} \log^2 \ell_0} \Big].
  \end{equation}
  Since $\rho_0 \le \ubc{C:rho0}$ is bounded, by \cref{lem:tail_bounds} the probability $\P_{\rho_0}(\eta_0(x) \ge \uc{c:occupation} \log^2 \ell_0)$ decays super-exponentially in $\log^2 \ell_0$. By a union bound over the $\ell_0^2 + 1$ sites in $[0, w_0]$, we obtain the global occupation bound
  \begin{equation}\label{eq:occupation_bound}
    \P_{\rho_0} \Bigg( \sup_{\substack{x \in [0,w_0] \\ t \in [0,\ell_0]}} \eta_t(x) \ge \log^2 \ell_0 \Bigg) \le C \ell_0^3 e^{-\uc{c:occupation} \log^2 \ell_0}.
  \end{equation}

  On the complement of this event, the jump rate from any site inside $[0, w_0]$ is uniformly
  bounded by $\lambda = \Gamma_+ \log^2 \ell_0$ up to time $\ell_0$. As the event $S^+(B_0)$
  requires a $J$-path to start at $0$ and reach $w_0$, by erasing any spatial loops formed
  by the infection trajectory, we can extract a self-avoiding spatial path of length $n \ge
  \ell_0^2$. Since the infection only advances when a particle jumps, the probability of traversing a specific path of length $n$ is bounded by the probability that a sequence of $n$ independent Poisson clocks of rate $\lambda$ ring sequentially in the time interval $[0, \ell_0]$, which is exactly $\P(\poisson(\lambda \ell_0) \ge n)$.

  The number of self-avoiding paths of length $n$ starting at $0$ is at most $2^n$. Taking a union
  bound over all such paths and summing over all possible lengths $n \ge \ell_0^2$, the probability
  of any spatial crossing is at most (in $d=1$, there is exactly one such path, but we proceed with the general bounding argument to accommodate $d \ge 1$):
  \begin{gather}\label{eq:path_bound}
    \begin{aligned}
      \sum_{n = w_0}^\infty 2^n \P\big( \poisson(\Gamma_+ \ell_0 \log^2 \ell_0) \ge n \big)
      &\le \sum_{n = w_0}^\infty 2^n \left( \frac{e \Gamma_+ \ell_0 \log^2 \ell_0}{n} \right)^n\\
      &\le \sum_{n = w_0}^\infty \left( \frac{2e \Gamma_+ \log^2 \ell_0}{\ell_0} \right)^n \\
      &\le (1/2)^{\ell_0^2 - 1},
    \end{aligned}
  \end{gather}
  where the last inequality holds since the base of the power is bounded by $1/2$ for all
  sufficiently large $\ell_0$.

  Combining the occupation bound \eqref{eq:occupation_bound} with the path bound
  \eqref{eq:path_bound}, and using the same bounds for $S^-(B_0)$ by symmetry, we obtain
  \begin{equation}
    \P_{\rho_0} \big( S^+(B_0) \cup S^-(B_0) \big) \le 2 \Big( C \ell_0^3 e^{-\uc{c:occupation} \log^2 \ell_0} + (1/2)^{\ell_0^2 - 1} \Big).
  \end{equation}

  Finally, since $H(B_0) = T(B_0) \cup S^+(B_0) \cup S^-(B_0)$, a union bound implies that $p_0 \to
  0$ as $\ell_0$ grows.

  By \cref{lem:contraction}, it suffices to verify $p_0 \le \frac{1}{4 \ubc{C:entropy}}$ and
  $\ubc{C:entropy}^2 \varepsilon_k \le (1/2)^{k + 5}$ for all $k \ge 0$. By
  \cref{prop:recursive}, the error $\varepsilon_k$ consists of exponential terms bounded by
  $\ubc{C:eps_decay} \ell_k^4 \exp\braces{ -\uc{c:eps_decay} \rho_0^{-1} (\alpha^k\ell_0)^{1/3 - 2s} }$. Since $1/3 - 2s > 0$, this tends to zero
  super-exponentially in $k \ge 0$, so by further increasing $\ell_0$ as needed, we can handle both
  the remaining terms and the bound on $p_0$.

  This establishes the recursive contraction $p_k \le (1/2)^{k + 2}/\ubc{C:entropy}$, so $p_k \to 0$
  as $k \to \infty$. For any fixed density $\rho \in (0, \rho_\infty)$, we have $\rho \le \rho_k$
  for all $k$. By the monotonicity of $H(B_k)$ (\cref{lem:monotonicity}),
  \begin{equation}
    \P_\rho(H(B_k)) \le \P_{\rho_k}(H(B_k)) \le p_k \to 0.
  \end{equation}
  By \cref{lem:survival_implies_H}, the infection dies out almost surely for every $\rho<
  \rho_{\infty}$, concluding the proof.
\end{proof}

\section{Contagion regime}\label{sec:survival}

\newcommand{\Uin}{\In}
\newcommand{\Uout}{\Out}
\newcommand{\Vin}{V_{\text{in}}}
\newcommand{\Vout}{V_{\text{out}}}
\newcommand{\Vblocks}{V_{\text{blocks}}}
\newcommand{\Vsink}{V_{\text{sink}}}
\newcommand{\Elocal}{\mathcal{E}_{\text{local}}}
\newcommand{\Edense}{\mathcal{E}_{\text{dense}}}
\newcommand{\Ephaseone}{{E}_{1}}
\newcommand{\Ephasetwo}{{E}_{2}}
\newcommand{\Eempty}{\mathcal{E}_{\text{empty}}}
\newcommand{\Nmax}{N_{\text{max}}}
\newcommand{\Tacc}{T_{\text{acc}}}
\newcommand{\pmeet}{p_{\text{meet}}}
\newcommand{\phit}{p_{\text{hit}}}
\newcommand{\plimit}{p_{\text{limit}}}

In this section we establish conditions under which the infection survives with positive
probability. We consider both the high-density regime, where the infection propagates indefinitely
for any healing rate, and the low-healing regime, where the infection survives at any given density
provided the healing rate is sufficiently small. The two regimes correspond to
\cref{thm:contagion_high_rho,thm:contagion_low_delta}, respectively; both share the renormalization
machinery developed below and differ only in how the base-scale probability is controlled.

We apply a multi-scale oriented percolation scheme to establish the survival of the infection.
Rather than tracking the infection directly across arbitrary distances, the strategy is to establish
that good scale-0 blocks---within which the infection is locally guaranteed to spread---percolate
across space-time. We do this by iteratively bounding the probability that a connected path of good
blocks fails to cross space-time regions of increasing sizes. We begin in \cref{subsec:geometry} by
defining the multi-scale geometry and establishing the transitivity properties of good blocks.
\cref{subsec:recursive_bounds} is devoted to proving a recursive inequality that bounds the failure
probability at scale $k+1$ in terms of the failure probability at scale $k$. Finally, in \cref{subsec:seeding_lemma,subsec:high_density}, we provide the triggering mechanisms,
proving that the base-scale failure probability can be made arbitrarily small by properly tuning the free parameters.

\subsection{Multi-scale geometry and block transitivity}\label{subsec:geometry}

Fix a sufficiently large base-scale $L_0 \in \N$ and set $T_0 = L_0^2$. For $k \ge 0$, define the
sequence of scale multipliers $\ell_k = 2\floor{L_k^{1/2}/2} + 1$, which ensures $\ell_k$ is always an odd
integer, and define the spatial and temporal dimensions recursively by
\begin{equation}
  L_{k+1} = \ell_k L_k \quad \text{and} \quad T_{k+1} = 5 \ell_k T_k.
\end{equation}

The scale-$k$ space-time block is given by $B_k = [-3L_k, 3L_k]^d \times [0, T_k]$. To define the
connectivity and percolation of good blocks through $B_k$, we define input and output sets, which we
call interfaces. The \emph{input interface} at time $t=0$ is
\begin{equation}
  \Uin(B_k) = \parens*{ [-L_k/2, L_k/2)^d \cap \Z^d } \times \{0\}.
\end{equation}
At time $t=T_k$, we define $3^d$ \emph{output interfaces}, indexed by displacement vectors $v =
(v_1, \dots, v_d) \in \{-1, 0, 1\}^d$:
\begin{equation}
  \Uout(B_k, v) = \parens*{v L_k + \parens*{[-L_k/2, L_k/2)^d \cap \Z^d }} \times \{T_k\}.
\end{equation}

\begin{figure}[tpb]
  \centering
  \includegraphics[width=\linewidth]{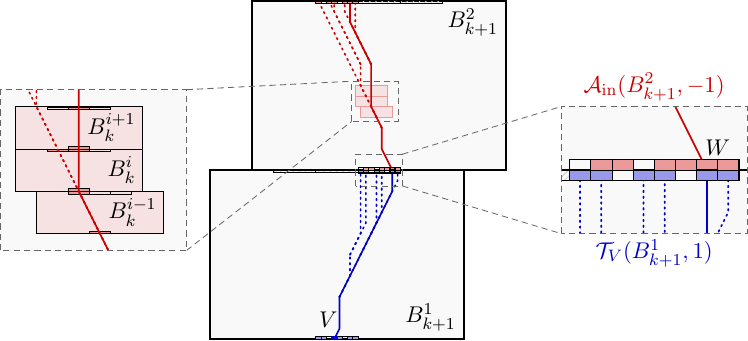}
  \caption{For two successive good blocks $B_{k+1}^1$ and $B_{k+1}^2$, the transitivity property
    guarantees that most sub-interfaces of $\Uin(B_{k+1}^1)$ are connected to most sub-interfaces of
    $\Uout(B_{k+1}^2, v)$ by a chain of good scale-$k$ blocks. Here we construct such a
    chain by first choosing any $V \in \A_{in}(B_{k+1}^1, 1)$, finding an element $W$ of
    $\T_V(B_{k + 1}^1, 1) \cap \A_{in}(B_{k + 1}^2, -1)$, and picking any good chain connecting
    $V$ to $W$, along with any good chain connecting $W$ to $\Uout(B_{k+1}^2, -1)$. This
    construction can be iterated to obtain a chain of good scale-0 blocks that
  traverses both blocks.}
  \label{fig:multi_scale_geometry}
\end{figure}

For $k \ge 1$, the interfaces of a scale-$k$ block $B_k$ naturally partition into $\ell_{k-1}^d$
disjoint translations of the scale-$(k-1)$ interfaces. We denote these collections of sub-interfaces
by $\L(\Uin(B_k))$ and $\L(\Uout(B_k, v))$, respectively.

\begin{definition}[Block connectivity]
  For two space-time blocks $B_k^1$ and $B_k^2$ of the same scale $k$, we say $B_k^2$ is a \emph{successor} of $B_k^1$ if an output interface of $B_k^1$ coincides with the input interface of $B_k^2$. A \emph{chain}
  of scale-$k$ blocks is a (possibly infinite) sequence $(B_k^1, B_k^2, \dots, B_k^n)$ where $B_k^{i + 1}$ is a successor of $B_k^{i}$ for all $i$. Such a path \emph{connects} an input interface $V$ of $B_k^1$ to an output interface $W$ of $B_k^n$ if it starts at $V$ and ends at $W$.
\end{definition}

\begin{definition}[Base-scale good block]
  Let $E(B)$ be an increasing event supported on the scale-$0$ block $B$. We say the scale-0 block $B$ is \emph{good} if the event $E(B)$ occurs.
\end{definition}

\begin{remark}
    In the proof of our results, it is important that the event $E(B)$ might be allowed to depend not only on the occupation of the zero-range process within $B$ but also on the graphical representation of this process. Fortunately, due to \cref{rmk:horizontal-graphical,rmk:vertical-graphical}, we can still apply the decoupling estimates for such events.
\end{remark}

\begin{definition}[Recursive good block]
  A scale-$k$ block $B$ is \emph{good} if, for each of its $3^d$ output interfaces $\Uout(B, v)$
  (where $v \in \{-1, 0, 1\}^d$), the following connectivity event holds: there exists a subset of input
  sub-interfaces $\A_{in}(B, v) \subseteq \L(\Uin(B))$ of size $|\A_{in}(B, v)| \ge \frac{3}{4}
  \ell_{k-1}^d$ such that, for every $V \in \A_{in}(B, v)$, there exists a target set $\T_V(B, v)
  \subseteq \L(\Uout(B, v))$ of size $|\T_V(B, v)| \ge \frac{3}{4} \ell_{k-1}^d$, such that $V$ is
  connected to each target $W \in \T_V(B, v)$ by a chain of good scale-$(k-1)$
  sub-blocks contained entirely within $B$.
\end{definition}

\begin{lemma}[Transitivity of connectivity]\label{lem:transitivity} For any $0 \le k' < k$, any
  chain of good scale-$k$ blocks contains a chain of good
  scale-$k'$ blocks traversing the same time interval.
\end{lemma}
\begin{proof}
  By induction on $k - k'$, we only need to show this for $k' = k-1$. Let $(B_k^{1}, B_k^2, \dots)$ be a chain of good scale-$k$ blocks, where the connection from $B_k^{i}$ to $B_k^{i+1}$ has displacement $v_i$.

  Fix $V_1 \in \A_{in}(B_k^{1}, v_1)$. Since $B_k^{1}$ is good, we have a target set $\T_{V_1}(B_k^{1}, v_1)$ of size at least $\frac{3}{4}\ell_{k-1}^d$. The next block $B_k^{2}$ has an input set $\A_{in}(B_k^{2}, v_2)$ of size at least $\frac{3}{4}\ell_{k-1}^d$. Since both are subsets of a set of size $\ell_{k-1}^d$, their intersection contains at least $\frac{1}{2}\ell_{k-1}^d \ge 1$ sub-interfaces.

  Fix $V_2 \in \T_{V_1} \cap \A_{in}$. There is a chain of good scale-$(k-1)$ blocks in $B_k^{1}$ connecting $V_1$ to $V_2$. Because $V_2 \in \A_{in}(B_k^{2}, v_2)$, we can repeat this at the boundary between $B_k^{2}$ and $B_k^{3}$ to find $V_3$, and so on. Continuing in this manner, we obtain a chain of good scale-$(k-1)$ blocks through the scale-$k$ sequence (see \cref{fig:multi_scale_geometry} for an illustration of this construction).
\end{proof}

\subsection{Recursive bounds on failure probabilities}\label{subsec:recursive_bounds}

To use decoupling inequalities for decreasing events, we introduce a sequence of increasing densities. Fix a sprinkling parameter $s = 1/10$. Given the initial base density $\rho_0 > 0$, define
\begin{equation}
  \rho_{k+1} = \rho_k(1 + L_k^{-s}).
\end{equation}
Since $L_k$ grows super-exponentially, the product $\prod_{k=0}^\infty (1+L_k^{-s})$ converges, and thus the sequence $\rho_k$ converges to some finite limit density $\rho_\infty < \infty$. Let $p_k$ denote the probability that a scale-$k$ block is bad (i.e.\ not good) under the invariant product measure at density $\rho_k$. Because goodness is an increasing event, a block being bad is a decreasing event, and its probability is decreasing with respect to the particle density.

For $k \ge 1$, the failure of a scale-$k$ block implies the existence of a large number of bad scale-$(k-1)$ blocks within it. In this section, we establish the recursive bound for the failure probability $p_{k+1}$ at density $\rho_{k+1}$ in terms of $p_k$.

\subsubsection{Deterministic bounds on defects}

Consider a directed lattice graph $G^*$ with vertex set $([-2m, 2m]^d \cap \mathbb{Z}^d) \times \{0, 1, \dots, H\}$. Directed edges are drawn from $(x, t)$ to $(x', t+1)$ if and only if $\|x - x'\|_\infty \le 1$. Let $S$ and $T$ be arbitrary spatial boxes within $[-2m, 2m]^d \cap \mathbb{Z}^d$ of side length $m$. We define the source set $\Vin = S \times \{0\}$ and the target set $\Vout = T \times \{H\}$, and assume the maximal coordinate distance between any $x \in S$ and $y \in T$ is bounded by $H$.

Let $G$ be a spanning subgraph of $G^*$ formed by designating each vertex as either ``good'' (retaining all its outgoing edges) or ``bad'' (having all its outgoing edges deleted). A vertex $v \in \Vin$ is called \emph{conductive} if it can reach at least $\frac{3}{4}m^d$ vertices in $\Vout$ via a directed path in $G$. Otherwise, it is \emph{non-conductive}.

\begin{lemma}[Graph defect multiplicity]\label{lem:graph_defect}
  If $G$ contains at least $\frac{1}{4}m^d$ non-conductive vertices, then the total number of bad vertices in $G$ is at least $C_d m^d$, where $C_d = 2^{-(d+4)}$.
\end{lemma}

\begin{proof}
  Let $N \subseteq \Vin$ denote the set of non-conductive vertices, and let $B$ denote the set of bad vertices in $G$. Assume $|N| \ge \frac{1}{4}m^d$. We bound the number of disconnected pairs $M$ between $\Vin$ and $\Vout$ via a family of deterministic paths. By definition, each $v \in N$ fails to reach at least $\frac{1}{4}m^d$ distinct destination vertices in $\Vout$. Summing over all $v \in N$ yields the lower bound:
  \begin{equation}\label{eq:lower_bound_M}
    M \ge |N| \left( \frac{1}{4}m^d \right) \ge \frac{1}{16}m^{2d}.
  \end{equation}

  To establish an upper bound, construct a deterministic path $\gamma_{x,y}$ in $G^*$ for every pair $(x, y) \in \Vin \times \Vout$. Let $\gamma(t)$ denote the spatial coordinate of the path at layer $t$, initialized at $\gamma(0) = x$. For each dimension $i \in \{1, \dots, d\}$, the path steps greedily toward the target:
  \begin{equation}
    \gamma(t+1)_i = \gamma(t)_i + \operatorname{sgn}(y_i - \gamma(t)_i).
  \end{equation}
  Since the maximal coordinate distance is bounded by $H$, the sequence $\gamma(t)$ reaches $y$ at or before layer $H$ and remains there. Thus, $\gamma_{x,y}$ is a path in $G^*$.

  If $(x, y)$ is disconnected in $G$, its canonical path $\gamma_{x,y}$ must be obstructed by some bad vertex $b = (z, s) \in B$. We now bound the number of paths passing through a fixed $(z, s)$. Analyzing component-wise: for a fixed dimension $i$, if the path has not reached its target ($y_i \neq z_i$), the starting coordinate $x_i$ is uniquely determined by $y_i$ since the path is monotonically stepping towards it. This yields at most $m$ pairs for this dimension. If $y_i = z_i$, $y_i$ is fixed and there are at most $m$ choices for $x_i$. Thus, there are at most $2m$ pairs $(x_i, y_i)$ for each dimension, leading to at most $(2m)^d$ total paths obstructed by $(z, s)$.

  Since every disconnected path must be blocked by at least one bad vertex, the total number of disconnected paths is bounded by:
  \begin{equation}
    M \le |B| \cdot 2^d m^d.
  \end{equation}
  Combining this with~\eqref{eq:lower_bound_M} yields $|B| \ge \frac{M}{2^d m^d} \ge 2^{-(d+4)} m^d = C_d m^d$.
\end{proof}

\subsubsection{The directed graph of sub-blocks}

For a fixed $k \geq 1$, we cover the spatial extent $[-2L_{k+1}, 2L_{k+1}]^d$ of the block $B_{k+1}$ by a grid of scale-$k$ sub-blocks centered at horizontal positions $x L_k$ for $x \in \{-2\ell_k, \dots, 2\ell_k\}^d$. We partition the temporal extent $[0, T_{k+1}]$ into $5\ell_k$ consecutive layers of duration $T_k$, indexed by vertical positions $y \in \{0, \dots, 5\ell_k - 1\}$.

This arrangement defines a discrete lattice $\Vblocks = \{-2\ell_k, \dots, 2\ell_k\}^d \times \{0, \dots, 5\ell_k - 1\}$. We append a layer of sink vertices $\Vsink = \{-2\ell_k, \dots, 2\ell_k\}^d \times \{5\ell_k\}$ to represent the output interfaces at the final time $T_{k+1}$. Let $V = \Vblocks \cup \Vsink$. Each vertex $(x, y) \in \Vblocks$ corresponds to the space-time sub-block
\begin{equation}
  B^{(x,y)}_k = (xL_k, yT_k) + B_k.
\end{equation}
A vertex $(x, y) \in \Vblocks$ is \emph{good} if the corresponding block $B^{(x,y)}_k$ is good; otherwise, it is \emph{bad}. Let $B$ denote the set of all bad vertices.

The output interfaces of $B^{(x,y)}_k$ coincide with the input interfaces of $B^{(x+v, y+1)}_k$ for $v \in \{-1, 0, 1\}^d$. We define the directed graph $G = (V, E)$ by placing directed edges from each good vertex $(x, y) \in \Vblocks$ to $(x+v, y+1)$ for all $v \in \{-1, 0, 1\}^d$. Bad vertices have all outgoing edges removed.

The input interface $\Uin(B_{k+1}) = ([-L_{k+1}/2, L_{k+1}/2)^d \cap \Z^d) \times \{0\}$ corresponds to the set of \emph{input vertices} at the base layer $y=0$:
\begin{equation}
  \Vin = \braces*{(x, 0) \in \Vblocks : x \in [-\ell_k/2, \ell_k/2)^d \cap \mathbb{Z}^d}.
\end{equation}
Since $L_{k+1} = \ell_k L_k$, the input interfaces of the sub-blocks centered at these vertices tile the macroscopic input interface. Thus, each vertex in $\Vin$ corresponds bijectively to a scale-$k$ sub-interface in $\L(\Uin(B_{k+1}))$.

If the block $B_{k+1}$ is bad, it must fail the connectivity requirement for at least one of its $3^d$ output interfaces. Let $v \in \{-1, 0, 1\}^d$ be the displacement vector of a failing output interface. This output interface $\Uout(B_{k+1}, v)$ corresponds to the following set of \emph{target vertices} at the terminal layer $y=5\ell_k$:
\begin{equation}
  \Vout(v) = \braces*{(x, 5\ell_k) \in \Vsink : x \in \parens*{v\ell_k + [- \ell_k/2, \ell_k/2)^d} \cap \mathbb{Z}^d}.
\end{equation}
Similarly, each vertex in $\Vout(v)$ corresponds bijectively to a scale-$k$ sub-interface in $\L(\Uout(B_{k+1}, v))$.

\begin{lemma}[Defect multiplicity]\label{lem:defect_multiplicity}
  Let $G = (V, E)$ be the directed grid graph defined above, and $B \subset \Vblocks$ be the set of bad vertices. If the block $B_{k+1}$ is bad, then $|B| \geq C_d \ell_k^d$.
\end{lemma}
\begin{proof}
  By the exact definition of a recursive good block, if $B_{k+1}$ is bad, it must fail the connectivity requirement for at least one output interface displacement $v \in \{-1, 0, 1\}^d$. This means the number of conductive vertices in $\Vin$ (with respect to $\Vout(v)$) is strictly less than $\frac{3}{4} |\Vin|$. Consequently, $G$ contains at least $\frac{1}{4} \ell_k^d$ non-conductive vertices.

  Applying \cref{lem:graph_defect} with $m = \ell_k$, using $\Vin$ as the source set and $\Vout(v)$ as the target set, immediately yields $|B| \ge C_d \ell_k^d$.
\end{proof}

\begin{lemma}[Separation of defective sub-blocks]\label{lem:defect_separation}
  If the block $B_{k+1}$ is bad, and $L_0$ is chosen sufficiently large, there exist at least two bad scale-$k$ sub-blocks $B_k^1$ and $B_k^2$ in $B_{k+1}$ separated by either a vertical distance $\dist_V(B_k^1, B_k^2) \ge T_k$ or a horizontal distance $\dist_H(B_k^1, B_k^2) \ge \frac{1}{2} \ell_k^{1/2} L_k$.
\end{lemma}
\begin{proof}
  Suppose for contradiction that every pair of bad sub-blocks is separated by strictly less than $\ell_k^{1/2}$ horizontally and strictly less than $2$ vertically in grid distance. Then all $|B|$ bad vertices must fit inside a single grid bounding box with side lengths bounded by $\ell_k^{1/2}$ in each spatial dimension and $2$ in the temporal dimension. This box contains at most $2\lceil \ell_k^{1/2} \rceil^d \le 2(\ell_k^{1/2}+1)^d$ lattice points. Since $|B| \ge C_d \ell_k^d$ by \cref{lem:defect_multiplicity}, this requires $C_d \ell_k^d \le 2(\ell_k^{1/2}+1)^d$. For all $d \ge 1$, this inequality is violated provided $L_0$ is chosen sufficiently large.

  Thus, there must exist at least two bad sub-blocks separated by a grid distance of either $\Delta y \ge 2$ vertically or $\Delta x \ge \ell_k^{1/2}$ horizontally. If $\Delta y \ge 2$, their temporal intervals of length $T_k$ satisfy $\dist_V \ge (\Delta y - 1)T_k \ge T_k$. If $\Delta y < 2$, they must be separated horizontally by $\Delta x \ge \ell_k^{1/2}$, meaning their spatial grid centers are separated by $\Delta x L_k \ge \ell_k^{1/2} L_k$. Since the spatial width of each sub-block is $6L_k$, their horizontal separation is
  \begin{equation}
    \dist_H \ge \ell_k^{1/2} L_k - 6L_k \ge \frac{1}{2} \ell_k^{1/2} L_k
  \end{equation}
  for sufficiently large $L_0$.
\end{proof}

\newbigconstant{C:decoupling}
\newconstant{c:survival_eps_decay}
\newbigconstant{C:survival_eps_decay}
\begin{lemma}[Recursive inequality]\label{lem:decoupling}
  For each $k \ge 0$,
  \begin{equation}
    p_{k+1} \le \ubc{C:decoupling} \ell_k^{2d+2} (p_k^2 + \varepsilon_k),
  \end{equation}
  where the decoupling error $\varepsilon_k$ satisfies
  \begin{equation}\label{eq:survival_eps_decay}
    \varepsilon_k \le \ubc{C:survival_eps_decay} T_{k+1}^{d+1} \parens*{ \rho_0 \exp\braces*{ -\uc{c:survival_eps_decay} L_k^{1/3} } + \exp\braces*{ -\uc{c:survival_eps_decay} \rho_0 L_k^{2/15} } }
  \end{equation}
  for constants $\ubc{C:survival_eps_decay}, \uc{c:survival_eps_decay} > 0$ depending only on the dimension.
\end{lemma}
\begin{proof}
  As shown in \cref{lem:defect_separation}, whenever $B_{k+1}$ is bad, there exists a pair of scale-$k$ blocks $B_k^1, B_k^2$ that are separated vertically by $\dist_V \ge T_k$ or horizontally by $\dist_H \ge \frac{1}{2} \ell_k^{1/2} L_k$ such that both are bad. Using a union bound, we get
  \begin{equation}
    \label{eq:union_bound_bad}
    p_{k+1} = \P_{\rho_{k+1}}(B_{k+1} \text{ bad})
    \le \ubc{C:decoupling} \ell_k^{2d+2} \max_{(B_k^1, B_k^2)} \P_{\rho_{k+1}} (\{B_k^1 \text{ bad}\} \cap \{B_k^2 \text{ bad}\}),
  \end{equation}
  where the maximum ranges over the collection of at most $(5 \ell_k (4\ell_k+1)^d)^2 \leq \ubc{C:decoupling} \ell_k^{2d+2}$ pairs of separated blocks.
  
  We now bound the right-hand side of \eqref{eq:union_bound_bad} by dividing it into two cases depending on the separation between $B_k^1$ and $B_k^2$.

  For pairs that are vertically separated by $\dist_V \ge T_k$, we apply \cref{thm:vertical_decoupling} with box size $w = T_k$. Evaluating the joint event under the density $\rho_{k+1} = \rho_k(1 + \varepsilon)$ with multiplicative sprinkling increment $\varepsilon = L_k^{-s}$, we obtain
  \begin{equation}
    \P_{\rho_{k+1}} (\{B_k^1 \text{ bad}\} \cap \{B_k^2 \text{ bad}\}) \le p_k^2 + \varepsilon_k^{(V)},
  \end{equation}
  where, using $T_k = 5^k L_k L_0$ to bound $L_{k} \leq \dist_V \leq T_{k+1}$ and bounding $\rho_0 \le \rho_k \le 2\rho_0$, gives
  \begin{equation}
    \begin{aligned}
      \varepsilon_k^{(V)} &= \ubc{C:vertical_decoupling} (T_k + \dist_V)^{d+1} \Big( \rho_k \exp\braces*{-\uc{c:vertical_decoupling} \dist_V^{1/3}} + \exp\braces*{-\uc{c:vertical_decoupling} \rho_k \varepsilon^2 \dist_V^{1/3}} \Big) \\
      &\le C T_{k+1}^{d+1} \Big( 2\rho_0 \exp\braces*{-\uc{c:vertical_decoupling} L_k^{1/3}} + \exp\braces*{-\uc{c:vertical_decoupling} \rho_0 L_k^{-1/5} L_k^{1/3}} \Big) \\
      &\le C T_{k+1}^{d+1} \Big( \rho_0 \exp\braces*{-\uc{c:vertical_decoupling} L_k^{1/3}} + \exp\braces*{-\uc{c:vertical_decoupling} \rho_0 L_k^{2/15}} \Big).
    \end{aligned}
  \end{equation}

  If they are not vertically separated, meaning $\dist_V < T_k$, their temporal projections lie in an interval of length at most $2 T_k$. Their horizontal separation must then satisfy $\dist_H \ge \frac{1}{2} \ell_k^{1/2} L_k$. Applying \cref{thm:horizontal_decoupling} with spatial width parameter $w = 6L_k$ and time horizon $T = 2 T_k$ gives
  \begin{equation}
    \P_{\rho_{k+1}} (\{B_k^1 \text{ bad}\} \cap \{B_k^2 \text{ bad}\}) \le \P_{\rho_{k+1}} (B_k \text{ bad})^2 + \varepsilon_k^{(H)} \le p_k^2 + \varepsilon_k^{(H)},
  \end{equation}
  where, using $\rho_{k+1} \le 2\rho_0$ and $T_k = 5^k L_k L_0$, the error is bounded by
  \begin{equation}
    \begin{aligned}
      \varepsilon_k^{(H)} &= \ubc{C:horiz_decoupling} (6L_k)^{d-1} \rho_{k+1} (2T_k)^{d/2} \exp\braces*{-\uc{c:horiz_decoupling} \dist_H \log\parens*{1+\frac{\dist_H}{2T_k}}} \\
      &\le C L_k^{d-1} \rho_0 T_k^{d/2} \exp\braces*{- c \ell_k^{1/2} L_k \log\parens*{1 + \frac{c'\ell_k^{1/2}}{5^k L_0}}} \\
      &\le C \rho_0 L_k^{d-1} (5^k L_k L_0)^{d/2} \exp\braces*{- c L_k^{1/2}}.
      \end{aligned}
  \end{equation}

  Taking $\varepsilon_k = \max(\varepsilon_k^{(H)}, \varepsilon_k^{(V)})$, the joint probability of any specific separated pair is bounded above by $p_k^2 + \varepsilon_k$. The vertical error determines the overall decay rate, since the horizontal decay decreases much faster. The polynomial factor in $\varepsilon_k^{(H)}$ is readily absorbed by $T_{k+1}^{d+1}$ in the vertical error bound. Adjusting the constants $\ubc{C:survival_eps_decay}$ and $\uc{c:survival_eps_decay}$ accordingly yields \eqref{eq:survival_eps_decay}.
\end{proof}

\newbigconstant{C:recursive_contraction}
\begin{proposition}[Recursive contraction]\label{prop:recursive_contraction}
  Fix an initial density $\rho_0 > 0$. There exists a constant $\ubc{C:recursive_contraction} > 0$ such that if $L_0 > \ubc{C:recursive_contraction}$, $p_0 \le L_0^{-5d}$,  and for all $k \ge 0$ the decoupling error satisfies
  \begin{equation}\label{eq:error_bound}
    \ubc{C:decoupling} \ell_k^{2d+2} L_{k+1}^{5d} \varepsilon_k \le \frac{1}{2},
  \end{equation}
  then $p_{k} \le L_{k}^{-5d}$ for all $k \ge 0$.
\end{proposition}
\begin{proof}
  We proceed by induction. The base case holds by assumption. For the inductive step, assume $p_k \le L_k^{-5d}$ for some $k \ge 0$. By \cref{lem:decoupling}, $p_{k+1} \le \ubc{C:decoupling} \ell_k^{2d+2} (p_k^2 + \varepsilon_k)$. Multiplying by $L_{k+1}^{5d}$ and substituting $L_{k+1} = \ell_k L_k$ alongside the bound $\ell_k \le 2L_k^{1/2}$ yields
  \begin{equation}
    \begin{aligned}
      L_{k+1}^{5d} p_{k+1} &\le \ubc{C:decoupling} \ell_k^{2d+2} (\ell_k L_k)^{5d} (L_k^{-10d} + \varepsilon_k) \\
      &\le \ubc{C:decoupling} 2^{7d+2} L_k^{\frac{7d+2}{2} - 5d} + \ubc{C:decoupling} \ell_k^{2d+2} L_{k+1}^{5d} \varepsilon_k.
    \end{aligned}
  \end{equation}
  Since $\frac{7d+2}{2} - 5d \le -\frac{1}{2}$ for all $d \ge 1$, the exponent of $L_k$ in the first term is strictly negative. Because $L_k$ is an increasing sequence, choosing $\ubc{C:recursive_contraction} \ge (\ubc{C:decoupling} 2^{7d+3})^2$ ensures that for $L_0 > \ubc{C:recursive_contraction}$, this first term is bounded by $1/2$. The second term is bounded by $1/2$ by assumption \eqref{eq:error_bound}. Therefore, $L_{k+1}^{5d} p_{k+1} \le 1/2 + 1/2 = 1$, completing the induction.
\end{proof}

\begin{theorem}[Oriented percolation]\label{thm:oriented_percolation}
  Suppose the bounds of \cref{prop:recursive_contraction} are satisfied. Then for the zero-range process initialized with the stationary measure at density $\rho_\infty$, almost surely there exists an infinite chain of good scale-$0$ blocks.
\end{theorem}

\begin{proof}
  By \cref{prop:recursive_contraction}, the failure
  probability $p_k$ at density $\rho_k$ satisfies $p_k \le L_k^{-5d}$ for all $k \ge 0$. Since the
  process is initialized globally at density $\rho_\infty \ge \rho_k$ for all $k$, by monotonicity,
  the true failure probability of any specific scale-$k$ block in the process is bounded above by
  $p_k \le L_k^{-5d}$.

  We proceed by iteratively constructing a hierarchy of contiguous space-time blocks. Let $\B_0$
  denote an inverted pyramid of scale-$0$ blocks rooted at the origin, formed by following all $3^d$
  output interfaces for $(\ell_0 - 1)/2$ time steps. The top layer of $\B_0$ consists of $\ell_0^d$
  blocks, exactly covering the input interface of a single scale-$1$ block. Iteratively for each $k
  \ge 1$, we construct an inverted pyramid $\B_k$ of scale-$k$ blocks succeeding the top layer of $\B_{k-1}$, expanding from its root for $(\ell_k - 1)/2$ time steps.

  Let $\mathcal{N}_k = |\B_k|$ denote the total number of scale-$k$ blocks in this construction.
  Since $\B_k$ is a full pyramid of height $(\ell_k - 1)/2$, we have $\mathcal{N}_k \le \ell_k^{d+1}$. The
  probability that any specific block in $\B_k$ is bad is bounded by $p_k \le L_k^{-5d}$. Using $\ell_k
  \le 2L_k^{1/2}$, summing the failure probabilities over all blocks yields
  \begin{equation}
    \sum_{k=0}^\infty \mathcal{N}_k p_k \le \sum_{k=0}^\infty \ell_k^{d+1} L_k^{-5d} \le \sum_{k=0}^\infty 2^{d+1} L_k^{\frac{d+1}{2}} L_k^{-5d} \le \sum_{k=0}^\infty 2^{d+1} L_k^{-4d}.
  \end{equation}
  Since $L_{k+1} \ge \frac{1}{2} L_k^{3/2}$, the sequence $L_k$ grows super-exponentially, ensuring
  the sum converges. By the Borel--Cantelli lemma, almost surely only finitely many blocks in
  $\bigcup_{k=0}^\infty \B_k$ are bad.

  Consequently, there exists an almost-surely finite scale $K$ such that every block in $\B_k$ is
  good for all $k \ge K$.

  We now construct an infinite chain of good scale-0 blocks. Fix an arbitrary
  integer $n \ge 1$. Since all blocks in the pyramid $\B_{K+n}$ are good, we can trace a sequence of
  good scale-$(K+n)$ blocks from its base to its top layer. By \cref{lem:transitivity}, this
  macroscopic path contains a chain of good scale-$(K+n-1)$ blocks. Since the base
  of $\B_{K+n}$ coincides with the top layer of $\B_{K+n-1}$, and all blocks in $\B_{K+n-1}$ are
  also good, we can extend this scale-$(K+n-1)$ sequence backwards through $\B_{K+n-1}$ to its root.

  Iterating this argument down to scale $K$ yields a finite sequence of good scale-$K$ blocks
  originating at the root of $\B_K$ and terminating at the top layer of $\B_{K+n}$. As this holds
  for every $n \ge 1$, we obtain an infinite family of finite paths, all starting from the single
  root of $\B_K$ and reaching arbitrarily far forward in time. Because each scale-$K$ block has
  exactly $3^d$ output interfaces, the tree formed by all such paths is locally finite. By König's
  lemma, this tree must contain an infinite path, establishing the existence of an infinite chain
  of good scale-$K$ blocks originating from $\B_K$.

  Applying \cref{lem:transitivity} one final time to this infinite sequence provides the desired
  infinite chain of good scale-$0$ blocks originating from the base-scale.
\end{proof}

\begin{remark}\label{rmk:genealogical-survival}
If the scale-0 good event is defined so that an infinite chain of good scale-0 blocks produces an infinite genealogical path, then \cref{thm:oriented_percolation} implies that there exists an infinite genealogical path almost surely. By translation invariance and stationarity of the zero-range process and graphical construction, the probability that an infinite genealogical path starts from $(0,0)$ is therefore positive. Since all particles at the origin are initially infected, this implies survival of the infection with positive probability.
\end{remark}

\subsection{The seeding lemma}\label{subsec:seeding_lemma}

We define the seeding event $\S(B_0)$ for a scale-$0$ block $B_0$. This event guarantees that the
zero-range dynamics, ignoring any possible healing marks, effectively propagates the infection from the input to all required outputs.

\begin{definition}[Seeding event]\label{def:seeding_event} Consider a scale-$0$ block $B_0$ with
  input interface $\Uin(B_0)$ and output interfaces $\Uout(B_0, v)$. Let $\S(B_0)$ denote the \emph{seeding
  event}: for every $(x, 0) \in \Uin(B_0)$ and $v \in \{-1, 0, 1\}^d$, assuming we artificially place a particle at $x$ if it is
  initially empty, there exists at least one $J$-path  (recall \cref{def:weak_jump_path}) originating from $x$ and ending in $\Uout(B_0, v)$ at time $T_0$, such that this path is entirely contained within $B_0$.
\end{definition}

We now establish that the probability of the seeding event failing can be made arbitrarily small by
taking $L_0$ large enough. The $3^d$ output interfaces at time $T_0$ serve as our targets, which we
collect in the set $\mathcal{M}$. Each target $m_i \in \mathcal{M}$ is a spatial box of side length
$L_0$.

To facilitate the subsequent proofs, it is convenient to track these $J$-paths using a particle
coloring mechanism. For a given source vertex $x \in \Uin(B_0)$, suppose we color all particles initially
at $x$ green (adding a single green particle if $x$ is empty). We assume that any other particle
that shares a site with a green particle instantly becomes green, and that any particle exiting the
spatial boundary of $B_0$ instantly loses its green color. Under these rules, the existence of a
valid $J$-path from $x$ to a target $m_i \in \mathcal{M}$ strictly contained within $B_0$ is equivalent to the event that at least one green particle resides in $m_i$ at time $T_0$.

\begin{figure}[tpb]
  \centering
  \begin{subfigure}[b]{0.54\linewidth}
    \centering
    \includegraphics[width=\linewidth]{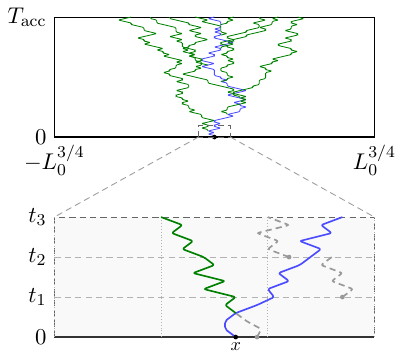}
    \caption{Stage 1: Accumulation}
    \label{fig:seeding_phase_1}
  \end{subfigure}
  \hfill
  \begin{subfigure}[b]{0.40\linewidth}
    \centering
    \includegraphics[width=\linewidth]{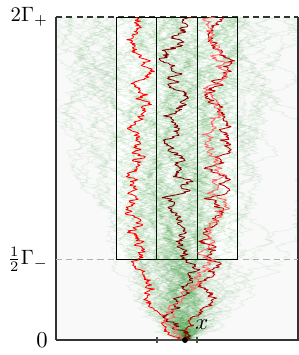}
    \caption{Stage 2: Diffusion}
    \label{fig:seeding_phase_2}
  \end{subfigure}
  \caption{The two stages of the seeding event. (a) A single seed particle locally spreads the infection, rapidly accumulating a large population of green particles. (b) These particles act as independent messengers, diffusing over a long time window to successfully reach all required macroscopic output targets.}
  \label{fig:seeding_event}
\end{figure}

To bound the failure probability of the seeding event, we divide the time interval $[0, T_0]$ into two stages: an accumulation stage and a diffusion stage, as illustrated in \cref{fig:seeding_event}.
\begin{itemize}
  \item Accumulation stage: During an initial short time window $[0, \Tacc]$ of length $\Tacc = L_0^{3/2}$, the single initial green particle infects a growing collection of particles. We show that, with high probability, this process creates a large number of green particles (proportional to $L_0^{1/4}$) within a small spatial neighborhood of the origin.
  \item Diffusion stage: In the remaining time $[\Tacc, T_0]$, the accumulated green particles act as messengers. To decouple their trajectories, we couple them to independent random walks constrained to remain within a space-time ``tunnel'' that ultimately confines them inside the target region for a large time interval. By invariance principles for Brownian motion, traversing this macroscopic tunnel has strictly positive probability, ensuring that at least one of the numerous messengers succeeds with high probability.
\end{itemize}

\newcommand{\barrho}{\bar{\rho}_0}
\newcommand{\Edisp}{\mathcal{E}_{\text{disp}}}

Before proving the lemmas that bound the failure probabilities of these events, we note a convention
regarding the base-scale $L_0$. As our goal is to select $L_0$ sufficiently large to trigger the
renormalization decay, any intermediate bounds or inequalities requiring $L_0$ to exceed an absolute
constant (independent of $\rho_0$) will be implicitly assumed to hold. This simplifies the
exposition by allowing us to absorb lower-order terms or apply asymptotic limits (such as the weak
convergence of random walks to Brownian motion) without repeatedly adding the qualification ``for
sufficiently large $L_0$''. For compactness, we denote the capped density by $\barrho = \min(1,
\rho_0)$.

\newbigconstant{C:spatial_containment}
\newconstant{c:spatial_containment}
\begin{lemma}[Spatial containment]\label{lem:spatial_containment} Let $\Tacc = L_0^{3/2}$. Denote by $\Elocal$ the event that no
  particle initially inside the box $\Lambda = [-2L_0, 2L_0]^d$ (including the possible additional seed particle) travels a distance
  greater than $L_0^{7/8}/2$ during $[0, \Tacc]$, and no particle initially outside $\Lambda$ enters
  the central box $[-L_0, L_0]^d$ by time $\Tacc$. There exist constants
  $\ubc{C:spatial_containment}, \uc{c:spatial_containment} > 0$ such that
  \begin{equation}
    \P_{\rho_0}(\Elocal^c) \le \ubc{C:spatial_containment} (\rho_0 + 1) \exp\braces*{-\uc{c:spatial_containment} L_0^{1/4}}.
  \end{equation}
\end{lemma}
\begin{proof}
  We first bound the probability of the event $\Edisp$ that any particle initially inside $\Lambda$ displaces by more than $s = L_0^{7/8}/2$ by time $\Tacc = L_0^{3/2}$. Conditioning on the initial configuration
  $\eta_0$, a union bound gives
  \begin{equation}
    \P^{\eta_0} (\Edisp) \le \sum_{x \in \Lambda} (\eta_0(x) + 1) \P\parens*{ \max_{0 \le t \le \Tacc} \|X_t(x) - x\| \ge s }.
  \end{equation}
  By the displacement bound \eqref{eq:displacement_bound} in \cref{lem:crossing_bound}, the probability on the right-hand side is at most $\ubc{C:displacement} \exp\braces*{-\uc{c:displacement} s \log(1+s/\Tacc)}$. Taking the expectation over $\eta_0$ and using $\E_{\rho_0}[\eta_0(x)] = \rho_0$, we obtain
  \begin{equation}\label{eq:interior_containment}
    \begin{aligned}
      \P_{\rho_0}(\Edisp )
      &\le \sum_{x \in \Lambda} (\rho_0 + 1) \ubc{C:displacement} \exp\braces*{-\uc{c:displacement} s \log(1+s/\Tacc)} \\
      &\le (4L_0+1)^d (\rho_0 + 1) \ubc{C:displacement} \exp\braces*{-\uc{c:displacement} (L_0^{7/8}/2)\log(1 + L_0^{-5/8}/2)} \\
      &\le C (\rho_0 + 1) L_0^d \exp\braces*{ - c L_0^{7/8} \cdot L_0^{-5/8} }\\
      &= C (\rho_0 + 1) L_0^d \exp\braces*{ - c L_0^{1/4} }.
    \end{aligned}
  \end{equation}
  
  For the exterior particles, any particle starting outside $\Lambda = [-2L_0, 2L_0]^d$
  that reaches the central box $[-L_0, L_0]^d$ must travel a distance of at least $s = L_0$.
  Applying the crossing bound \eqref{eq:crossing_bound} from \cref{lem:crossing_bound} and taking a
  union bound over all $(2L_0 + 1)^d$ sites in the central box, the probability of an exterior
  invasion is bounded by:
  \begin{equation}\label{eq:exterior_containment}
    (2L_0 + 1)^d \cdot \ubc{C:crossing} \rho_0 \Tacc^{d/2} \exp\braces*{ -\uc{c:crossing} L_0\log(1+L_0/\Tacc) } \le C \rho_0 L_0^{7d/4} \exp\braces*{ -c L_0^{1/2} }.
  \end{equation}

  If the event $\Elocal^c$ occurs, then at least one of the above events occur. Thus, the probability of $\Elocal^c$ is bounded by the sum of \eqref{eq:interior_containment} and \eqref{eq:exterior_containment}. The $L_0^{1/4}$ exponential decay in \eqref{eq:interior_containment} is the slowest decay of the two, absorbing the polynomial factor into the exponential yields the desired bound for an appropriate choice of constants.
\end{proof}

\newbigconstant{C:accumulation}
\newconstant{c:accumulation}
\begin{lemma}[Accumulation of infections]\label{lem:accumulation_infections} Suppose we initialize
  the process with a single green seed particle at a position $x$. Let $\Ephaseone$ be the event
  that by time $\Tacc = L_0^{3/2}$, there exist at least $N = \lceil \frac{\barrho}{4} L_0^{1/4} \rceil$ green
  particles located within a distance of $L_0^{7/8}$ from $x$. There exist constants
  $\ubc{C:accumulation}, \uc{c:accumulation} > 0$ such that
  \begin{equation}
    \P_{\rho_0}(\Ephaseone^c) \le \ubc{C:accumulation} (\rho_0 + 1) L_0^{2d} \exp\braces*{ -\uc{c:accumulation} \barrho L_0^{1/4} }.
  \end{equation}
\end{lemma}
\begin{proof}
  Fix the scaling exponents $\nu = \frac{1}{4d}$ and $\beta = \frac{1}{2}$, which satisfy
  $\frac{3}{2}\beta > 2\nu$ for all $d \ge 1$. Partition $[0, \Tacc]$ into $K = \lfloor
  \Tacc^{1-\beta} \rfloor = \lfloor L_0^{3/4} \rfloor$ intervals of length $\Tacc^\beta =
  L_0^{3/4}$, yielding times $t_j = j \Tacc^\beta$. Partition the spatial region $\Lambda = [-2L_0,
  2L_0]^d$ into disjoint spatial blocks $\mathcal{V}$ of side length $L_0^\nu$. Let $\Edense$ be
  the event that every block $V \in \mathcal{V}$ contains at least $\frac{\rho_0}{2} (L_0^\nu)^d$
  particles at all times $t_j \le \Tacc$. Applying the tail bounds from
  \cref{lem:tail_bounds} for each block $V$ at each discrete time $t_j$, the failure
  probability per block is bounded by $\exp\braces{-c \barrho (L_0^\nu)^d}$. A union bound over the $K \le
  L_0^{3/4}$ time steps and $|\mathcal{V}| \le C L_0^{d(1-\nu)}$ spatial blocks yields
  \begin{equation}
    \P_{\rho_0}(\Edense^c) \le C L_0^{d+1} \exp\braces*{-c \barrho L_0^{1/4}}.
  \end{equation}

  We track the recruitment process step-by-step over the intervals $t_j$. At each time $t_{j-1}$, we
  first check if the events $\Elocal$ and $\Edense$ hold up to that time. If either fails, we
  declare a \emph{failure} and abort the recruitment process. If both hold, let $V_j \in
  \mathcal{V}$ be the block containing the seed particle. If $V_j$ contains no uncolored particles,
  we declare an \emph{early success} and stop, having accumulated at least
  $\frac{\rho_0}{2}(L_0^\nu)^d \ge N$ green particles via $\Edense$. Otherwise, we
  deterministically select one uncolored particle $Y_j \in V_j$. The initial distance between the
  seed and $Y_j$ is bounded by the block diameter $L_0^\nu \sqrt{d}$. We apply \cref{lem:meeting} with spatial separation $R = L_0^\nu \sqrt{d}$ and time window $T = \Tacc^\beta = L_0^{3/4}$. For sufficiently large $L_0$, the required condition $R \le \sqrt{\ubc{C:meeting_distance} T}$ is satisfied since $T^{1/2} = L_0^{3/8}$ dominates $R$ for all $d \ge 1$. The meeting lemma ensures these two particles meet during $[t_{j-1}, t_j]$ with probability at least $\ubc{C:meeting} R^{-\gamma_d}$. By absorbing the dimension-dependent factor $d^{-\gamma_d/2}$ into the constant, we obtain $\pmeet \ge C L_0^{-\nu \gamma_d}$. By the strong Markov property and the definition of our
  stopping conditions, the total number of newly recruited green particles stochastically dominates
  \begin{equation}\label{eq:min_particles}
    \min\parens*{ \frac{\rho_0}{2}(L_0^\nu)^d, \sum_{j=1}^K B_j } \charf{\Elocal \cap \Edense},
  \end{equation}
  where the $B_j$ are independent $\ber(\pmeet)$ random variables.
  
  Since $K \ge L_0^{3/4}/2$, the expected value for the sum in \cref{eq:min_particles} is
  \begin{equation}
    \mu = K \pmeet \ge C L_0^{3/4 - \nu \gamma_d} = \begin{cases}
      C L_0^{3/4} & \text{if } d=1, \\
      C L_0^{5/8} & \text{if } d=2,\\
      C L_0^{1/2 + 1/(2d)} & \text{if } d \ge 3,
    \end{cases}
  \end{equation}
  so that $\mu \ge C L_0^{1/2}$ for all $d \ge 1$.

  The target threshold $N = \lceil \frac{\barrho}{4} L_0^{1/4} \rceil$ grows strictly slower than $\mu$, so a standard Chernoff bound asserts that $\P\parens*{ \sum_{j=1}^K B_j < N }$
  decays exponentially in $L_0^{1/4}$.

  By a simple union bound, the failure probability of the accumulation stage satisfies
  \begin{equation}
    \P_{\rho_0}(\Ephaseone^c) \le \P_{\rho_0}(\Elocal^c) + \P_{\rho_0}(\Edense^c) + \P\parens*{ \sum_{j=1}^K B_j < N }.
  \end{equation}
  Applying \cref{lem:spatial_containment} to bound the first term and combining it with the
  exponential bounds for the latter two terms yields the desired result.
\end{proof}

\newbigconstant{C:diffusion}
\newconstant{c:diffusion}
\begin{lemma}[Independent diffusion to targets]\label{lem:independent_diffusion} Let $\Ephasetwo$ be
  the event that each target $m_i \in \mathcal{M}$ contains at least one green particle at time
  $T_0$. There exist constants $\ubc{C:diffusion},
  \uc{c:diffusion} > 0$ such that
  \begin{equation}
    \P_{\rho_0}\parens*{ \Ephasetwo^c \mid \Ephaseone } \le \ubc{C:diffusion} \barrho L_0^{1/4} \exp\braces*{-\uc{c:diffusion} \barrho L_0^{1/4}}.
  \end{equation}
\end{lemma}
\begin{proof}
  On the event $\Ephaseone$, there exist at least $N$ green particles within distance $L_0^{7/8}$ of
  $x$ at time $\Tacc$. We deterministically select exactly $N$ of these specific particles and refer
  to them as \emph{messengers}.
  The selected messengers are located at initial positions $Z_1, \dots, Z_N$ at time $\Tacc$,
  satisfying $|Z_m - x| \le L_0^{7/8}$ for all $m \le N$. Under diffusive scaling by $L_0$, all
  $Z_m$ are effectively localized at the source $x$.

  We apply a variant of the slot representation, assigning to each messenger an independent simple
  random walk $(Y^{(m)}_n)_{n \ge 0}$ starting at $Y^{(m)}_0 = Z_m$, and two Poisson processes
  $\Pi_m^-$ and $\Pi_m^+$ with rates $\Gamma_-$ and $\Gamma_+$, such that $\Pi_m^-$ is a thinning of
  $\Pi_m^+$. The total number of jumps $J_m$ executed by messenger $m$ in $[\Tacc, T_0]$ satisfies
  \begin{equation}
    J_m^- \leq J_m \leq J_m^+,
  \end{equation}
  where $J_m^\pm = \Pi_m^\pm([\Tacc, T_0])$. Let $G_m$ be the typical event that $J_m^- \ge
  \frac{1}{2}\Gamma_- (T_0 - \Tacc)$ and $J_m^+ \le 2\Gamma_+ (T_0 - \Tacc)$. Then
  \begin{equation}
    \P_{\rho_0}(G_m^c) \le \exp\braces*{-c (T_0 - \Tacc)}.
  \end{equation}

  Conditioned on the jump processes, let $A_{m,i}$ be the event that messenger $m$ resides in target
  $m_i$ at time $T_0$ and never exits the spatial boundary of $B_0$ during $[\Tacc, T_0]$. We bound
  this below by the event
  \begin{equation}
    \tilde{A}_{m,i} = \braces*{
      \begin{aligned}
        &Y^{(m)}_j \in m_i \text{ for all } j \in [J_m^-, J_m^+] \text{ and} \\
        &Y^{(m)}_j \in [-3L_0, 3L_0]^d \text{ for all } j \le J_m^+
    \end{aligned} }.
  \end{equation}
  The mutual independence of the random walks $Y^{(m)}$ and jump bounds across messengers ensures
  the events $\tilde{A}_{1,i}, \dots, \tilde{A}_{N,i}$ are independent.

  By Donsker's Theorem, the sequence of scaled trajectories $t \mapsto (Y^{(m)}_{\lfloor t L_0^2
  \rfloor} - x) / L_0$ converges weakly to standard Brownian motion. The targets $(m_i - x)/L_0$
  correspond to spatial boxes of unit side length at relative displacements $v = (m_i - x)/L_0$. The
  spatial boundary of $B_0$ corresponds to the box $[-3, 3]^d$. Since $x \in [-L_0/2, L_0/2]^d$ and
  the targets $m_i$ are selected from a finite set of positions strictly inside $[-3L_0, 3L_0]^d$,
  the displacements $v$ are confined to a compact subset of the interior of the shifted boundary
  $[-3-x/L_0, 3-x/L_0]^d$. A standard Brownian motion starting at $0$ has full topological support,
  so the probability that it stays strictly within the interior of this shifted boundary, enters a
  unit box at displacement $v$ before time $\frac{1}{2}\Gamma_-$, and remains inside until time
  $2\Gamma_+$ is strictly positive. By continuity, this probability attains a strictly positive
  minimum $\plimit > 0$ over the compact set of all possible initial positions $x$ and targets
  $m_i$.

  Consequently, the weak convergence to Brownian motion implies that for all sufficiently large
  $L_0$, the discrete probability that the random walk successfully enters and remains in the target
  is uniformly bounded below by a positive constant $\phit = \plimit/2$ for all $x$ and $m_i$. Let
  $W$ be the number of messengers reaching $m_i$. Then $W$ stochastically dominates a binomial random
  variable $\bin(N, \phit)$. Since $N = \lceil \barrho L_0^{1/4} \rceil$, Hoeffding's inequality
  yields that, conditional on the jump bounds $G_1, \dots, G_N$, the probability that no messengers
  reach the target while $\Ephaseone$ holds is bounded by:
  \begin{equation}
    \P_{\rho_0}\parens*{ \{W < 1\} \cap \Ephaseone \;\big|\; G_1 \dots G_N } \le \exp\braces*{ -2\frac{(\phit N - 1)^2}{N} } \le \exp\braces{-c \barrho L_0^{1/4}}.
  \end{equation}

  Removing the conditioning introduces an error of at most
  \begin{equation}
    N \exp\braces*{-c (T_0 - \Tacc)} \le C \barrho L_0^{1/4} \exp\braces*{-c L_0^2}.
  \end{equation}
  Summing over the probabilities, the conditional failure bound is dominated by $C \barrho L_0^{1/4} \exp\braces*{-c \barrho L_0^{1/4}}$.
\end{proof}

\newbigconstant{C:seeding}
\newconstant{c:seeding}
\begin{lemma}[Seeding Lemma]\label{lem:seeding}
  There exist constants $\ubc{C:seeding}, \uc{c:seeding} > 0$ such that
  \begin{equation}
    \P_{\rho_0}(\mathcal{S}^c) \le \ubc{C:seeding} (\rho_0 + 1) L_0^{3d} \exp\braces*{-\uc{c:seeding} \barrho L_0^{1/4}}.
  \end{equation}
\end{lemma}
\begin{proof}
  The seeding event $\mathcal{S}$ requires success across all $x \in \Uin(B_0)$ and all $3^d$
  targets. For a fixed $x$, the success is guaranteed by the intersection $\Ephaseone \cap
  \Ephasetwo$. By \cref{lem:accumulation_infections} and \cref{lem:independent_diffusion}, the
  failure probability from a single $x$ is bounded by
  \begin{equation}
    \begin{split}
      \P_{\rho_0}((\Ephaseone \cap \Ephasetwo)^c) &\le \P_{\rho_0}(\Ephaseone^c) + \P_{\rho_0}\parens*{ \Ephasetwo^c \mid \Ephaseone } \\
      &\le C (\rho_0 + 1) L_0^{2d} \exp\braces*{ -c \barrho L_0^{1/4} }.
    \end{split}
  \end{equation}
  Applying a union bound over all $x \in \Uin(B_0)$, which has size $|\Uin(B_0)| \le (2L_0+1)^d$,
  the total failure probability of the seeding event is:
  \begin{align}
    \P_{\rho_0}(\mathcal{S}^c) &\le |\Uin(B_0)| \P_{\rho_0}((\Ephaseone \cap \Ephasetwo)^c) \nonumber \\
    &\le C (\rho_0 + 1) L_0^{3d} \exp\braces*{-c \barrho L_0^{1/4}}.\qedhere
  \end{align}
\end{proof}

\subsection{Contagion at high densities}\label{subsec:high_density}

We now apply the renormalization framework to establish survival at any healing rate $\delta \in [0,
\infty]$, provided the initial density is sufficiently high. The strategy consists of showing that
at very high densities, all sites throughout the base-scale block contain at least 2 particles.
This property effectively suppresses healing marks even if they arrive at an infinite rate, allowing
us to propagate the infection through the seeding event.

\newcommand{\Eocc}{\mathcal{E}_{\mathrm{occ}}}
\begin{definition}[High density good block]\label{def:good_block_high_rho}
  Let $\S(B_0)$ be the seeding event for the block $B_0$, and let $\Eocc(B_0)$ be the event that the particle occupation satisfies $\eta_t(x) \ge 2$ for all
  $(x,t) \in B_0$. We define the event $E(B_0)$ that a scale-$0$ block $B_0$ is \emph{good} as the intersection of these two events:
  \begin{equation}
    E(B_0) = \S(B_0) \cap \Eocc(B_0).
  \end{equation}
\end{definition}

This event is increasing with respect to density. Since $\Eocc(B_0)$ ensures that every site contains at least two particles at all times, healing marks are entirely suppressed within the region (\cref{rmk:healing_suppression}). Consequently, the $J$-paths guaranteed by the seeding event are
valid genealogical paths. Thus, if two successive scale-$0$ blocks are good, an initially occupied site in the first block generates a valid genealogical path to an initially occupied site in the
second block.

\begin{proof}[Proof of \cref{thm:contagion_high_rho}]
  By monotonicity of the survival probability with respect to the healing rate, it suffices to prove
  the result for the case $\delta = \infty$. We set the initial density as a function of the base-scale $L_0$ by taking
  $\rho_0 = \sqrt{L_0}$. The parameter $L_0$ remains free and will be chosen sufficiently large
  below.

  The probability that the block $B_0$ is bad satisfies the union bound
  \begin{equation}
    p_0(\infty) \le \P_{\rho_0}(\Eocc(B_0)^c) + \P_{\rho_0}(\S(B_0)^c).
  \end{equation}

  Since the bounding box $B_0$ comprises $(6L_0+1)^d$ spatial sites and a time horizon of $T_0 =
  L_0^2$, we apply the infimum bound in \cref{lem:occupation_bounds} with threshold $m=2$ and free
  parameter $K = \rho_0 / 2$. This choice is valid since for large $L_0$, $K = \rho_0 / 2 \ge
  \uc{c:occupation}^{-1} 2$. For a single site $x$, the probability that the occupation drops below
  $2$ in $[0, L_0^2]$ is bounded by:
  \begin{equation}
    \P_{\rho_0}\Big( \inf_{t \in [0, L_0^2]} \eta_t(x) < 2 \Big) \le (L_0^2 + 1) \Big[ \P_{\rho_0}\big( \eta_0(x) \le \rho_0 / 2 \big) + \exp\braces*{-\uc{c:occupation} \rho_0 / 2} \Big].
  \end{equation}
  By \cref{lem:tail_bounds} with deviation $\varepsilon = 1/2$, the initial condition term is
  bounded by $\exp\braces*{-\frac{\rho_0}{8\Gamma}} = \exp\braces*{-\frac{\sqrt{L_0}}{8\Gamma}}$. Taking a union
  bound over the $(6L_0+1)^d$ sites in $B_0$, the failure of the sustained occupation event satisfies
  \begin{equation}\label{eq:sustained_occupation_high}
    \P_{\rho_0}(\Eocc(B_0)^c) \le (6L_0+1)^d (L_0^2 + 1) \Big[ \exp\braces*{-\frac{\sqrt{L_0}}{8\Gamma}} + \exp\braces*{-\uc{c:occupation} \frac{\sqrt{L_0}}{2}} \Big].
  \end{equation}

  By \cref{lem:seeding}, the failure of the seeding event is bounded by
  \begin{equation}
    \P_{\rho_0}(\S(B_0)^c) \le \ubc{C:seeding} (\rho_0 + 1) L_0^{3d} \exp\braces*{-\uc{c:seeding} \barrho L_0^{1/4}}.
  \end{equation}
  Evaluating this at $\rho_0 = \sqrt{L_0}$ yields
  \begin{equation}
    \P_{\rho_0}(\S(B_0)^c) \le \ubc{C:seeding} L_0^{3d+1/2} \exp\braces*{-\uc{c:seeding} L_0^{1/4}}.
  \end{equation}
  Both \eqref{eq:sustained_occupation_high} and the bound above decay stretched-exponentially in $L_0$. It remains to evaluate the decoupling error bound \eqref{eq:error_bound} for the
  recursive contraction. Since $\rho_0 \le \rho_k \le 2\rho_0 = 2\sqrt{L_0}$, substituting the
  uniform error bound \eqref{eq:survival_eps_decay} yields:
  \begin{equation}
  \begin{aligned}
    \ubc{C:decoupling} \ell_k^{2d+2} L_{k+1}^{5d} \varepsilon_k \le C \ell_k^{2d+2} L_{k+1}^{5d} T_{k+1}^{d+1} \Big( \sqrt{L_0} \exp\braces*{-\uc{c:survival_eps_decay} L_k^{1/3}} \\ + \exp\braces*{-\uc{c:survival_eps_decay} \sqrt{L_0} L_k^{2/15}} \Big).
  \end{aligned}
  \end{equation}
  This expression exhibits exponential decay in $L_k$, ensuring that it is strictly decreasing in
  $k \ge 0$ whenever $L_0$ is large enough. The supremum is thus attained at $k = 0$:
  \begin{equation}
    \sup_{k \ge 0} \ubc{C:decoupling} \ell_k^{2d+2} L_{k+1}^{5d} \varepsilon_k \le C L_0^{n} \Big( \sqrt{L_0} \exp\braces*{-\uc{c:survival_eps_decay} L_0^{1/3}} + \exp\braces*{-\uc{c:survival_eps_decay} L_0^{19/30}} \Big),
  \end{equation}
  for some dimension-dependent exponent $n > 0$.

  Since all of these failure bounds decay exponentially in $L_0$, we can choose a single base-scale
  $L_0$ sufficiently large so that $\P_{\rho_0}(\Eocc(B_0)^c) \le \frac{1}{2} L_0^{-5d}$ and
  $\P_{\rho_0}(\S(B_0)^c) \le \frac{1}{2} L_0^{-5d}$ (yielding $p_0(\infty) \le L_0^{-5d}$), while
  also bounding the decoupling error supremum by $1/2$. This satisfies all conditions of \cref{thm:oriented_percolation}.
  Applying the theorem guarantees the existence of an infinite
  chain of good scale-$0$ blocks a.s., which in turn implies the existence of an infinite genealogical path. \cref{rmk:genealogical-survival} then shows the infection survives with positive probability at density $\rho_\infty$.
  Fixing such an $L_0$, we set $\rho_\star = \rho_\infty(L_0) \le 2\sqrt{L_0}$. Since survival is an increasing event in $\rho$, the infection survives with positive probability for all $\rho \geq \rho_\star$.
\end{proof}

\subsection{Contagion at low healing rates}\label{subsec:low_healing}

We now apply the general renormalization framework to establish survival of the zero-range process
at arbitrary positive densities for sufficiently low healing rates.

\begin{definition}[Low healing rate good block]\label{def:good_block_low_delta}
  Let $\S(B_0)$ be the seeding event for $B_0$, $\Eempty^c$ be the event that the input interface of $B_0$ is not completely empty, and $\mathcal{E}_{\text{heal}}^c$ be the event that there are no healing marks inside $B_0$. We define the event $E(B_0)$ that a scale-$0$ block $B_0$ is \emph{good} as the intersection of these events:
  \begin{equation}
    E(B_0) = \S(B_0) \cap \Eempty^c \cap \mathcal{E}_{\text{heal}}^c.
  \end{equation}
\end{definition}

This event is increasing with respect to density. Since the event requires
the complete absence of healing marks, the $J$-paths guaranteed by the seeding event are valid genealogical paths. Thus, if two successive scale-$0$ blocks are good, an initially occupied site in the first block generates a valid genealogical path to an initially occupied site in the second block.

\begin{proof}[Proof of \cref{thm:contagion_low_delta}]
  Given $\rho > 0$, we choose a base density $\rho_0 < \rho$. The probability that the block $B_0$
  is bad at healing rate $\delta$ is bounded by
  \begin{equation}
    p_0(\delta) = \P_{\rho_0, \delta}(E(B_0)^c) \le \P_{\rho_0}(\Eempty) + \P_{\rho_0}(\S^c) + \P_\delta(\mathcal{E}_{\text{heal}}).
  \end{equation}
  Under the measure $\mu_{\rho_0}$, $\P_{\rho_0}(\Eempty) \le \exp\braces{-c \rho_0 L_0^d}$. By \cref{lem:seeding}, the failure of the seeding event satisfies
  \begin{equation}
    \P_{\rho_0}(\S^c) \le \ubc{C:seeding} (\rho_0 + 1) L_0^{3d} \exp\braces*{-\uc{c:seeding} \barrho L_0^{1/4}}.
  \end{equation}
  For any fixed $\rho_0 > 0$, the sum of these two terms decays exponentially in $L_0^{1/4}$. Thus,
  we can take $L_0$ sufficiently large such that
  \begin{equation}
    \P_{\rho_0}(\Eempty) + \P_{\rho_0}(\S^c) \le \frac{1}{2} L_0^{-5d}.
  \end{equation}
  Taking $L_0$ larger if necessary ensures that the decoupling bounds of
  \cref{prop:recursive_contraction} hold, and that the limit density satisfies $\rho_\infty = \rho_0
  \prod_{k=0}^\infty (1 + L_k^{-s}) \le \rho$.

  With $L_0$ fixed, the base-scale-$0$ block $B_0$ has finite space-time volume $V_0$. Since healing
  marks arrive as an independent Poisson point process with rate $\delta$, the probability of
  observing at least one healing mark inside $B_0$ is bounded by $\P_\delta(\mathcal{E}_{\text{heal}}) \le
  \delta V_0$. Thus, we can choose a critical threshold $\delta_{\star}(\rho) > 0$ such that
  $\delta_{\star}(\rho) V_0 \le \frac{1}{2} L_0^{-5d}$. For any $\delta \le \delta_{\star}(\rho)$,
  we have $p_0(\delta) \le L_0^{-5d}$.

  For any $\delta \le \delta_{\star}(\rho)$, the conditions of \cref{thm:oriented_percolation} are
  satisfied. Applying the theorem, for the process initialized at density $\rho_\infty \le \rho$,
  almost surely there exists an infinite chain of good scale-$0$ blocks. Concatenating
  the genealogical paths traversing each good block along this sequence yields almost surely an infinite
  genealogical path under $\P_{\rho_\infty, \delta}$, and hence under $\P_{\rho, \delta}$ by monotonicity in density, the proof is concluded by \cref{rmk:genealogical-survival}.
\end{proof}

\section{Concluding remarks and open questions}
\label{sec:open_problems}

In this paper we have characterized the survival-extinction phase diagram for the spread of an infection in the zero-range process as the pair of parameters $(\rho,\delta)$ varies in $[0,\infty)\times [0,\infty]$.
In particular we have obtained the existence of the critical curve $\delta \mapsto \rho_c(\delta)$ and some of its properties.

This work leaves some interesting open questions.

For example, in \cref{sec:intro} we have implicitly stated:
\begin{problem}
    Is the critical curve continuous and strictly increasing as illustrated in \cref{fig:phase_diagram_left}?
\end{problem}

\begin{problem}
\label{prob:equality_rho_c}
    Is it true that $\rho_c(\infty-) = \rho_c(\infty)$?
\end{problem}
Recall that a positive answer for \cref{prob:equality_rho_c} would imply the existence of a non-trivial critical healing threshold $\delta_c(\rho)$ throughout the interval $(0,\rho_c(\infty))$.

While \cref{thm:phase_transition_rho,thm:phase_transition_delta} yield a phase transition between an extinction and a global survival phase, we believe that local survival also holds. 
That is, almost surely on the event that the infection survives, the origin gets reinfected infinitely often.
We state that as a problem:
\begin{problem}
    On the event that the infection survives, does local survival hold a.s.?
\end{problem}

Another interesting question concerns the rate of propagation for the front of the infection which we believe to be linear, on the event of survival:
\begin{problem}
    On the event that the infection survives, prove linear lower and upper bounds for the speed of propagation of the infection.
\end{problem}
A much more challenging problem would be to prove a full shape theorem as in~\cite{KestenSidoravicius08}:
\begin{problem}
    Does the set of sites that have once been infected satisfy a shape theorem?
\end{problem}
Recall that, in~\cite{KestenSidoravicius08} cures are not allowed, and the environment is composed of IRW.

As mentioned in \cref{sec:motivation} for the IRW environment, in~\cite{DauvergneSly23} the authors considered the situation where infected and healthy particles evolve with different diffusion rates.
In the case of the zero-range process this can be achieved, for instance, by considering two different rate functions $g_A$ and $g_B$.
A simple setup would be to consider $g_i = D_i g$, for $i=A,B$.
This leads to the following
\begin{problem}
  What can one say about the phase diagram and the speed of propagation when $D_A \neq D_B$?
\end{problem}

Regarding the type of environment that we consider, the reader might have noticed that the main feature we use is that it satisfies the horizontal and vertical decoupling inequalities given in \cref{thm:horizontal_decoupling,thm:vertical_decoupling}, respectively.
It is worth noticing that the simple symmetric exclusion process also satisfies these types of inequalities, hence we expect our methods to be useful to study the same model on this environment.
However, such types of decoupling inequalities are not available for the asymmetric simple exclusion process (ASEP), including the totally asymmetric case (TASEP).

\begin{problem}
    What can one say about the phase diagram when the ZRP is replaced by the ASEP (or the TASEP) process?
\end{problem}

\appendix
\crefalias{section}{appendix}

\section{Deferred proofs for the zero-range process}
\label{sec:deferred_proofs}

This appendix collects the proofs of several technical lemmas from \cref{sec:zr_bounds} regarding the behavior of the zero-range process that were deferred from the main text.

\begin{proof}[Proof of \cref{lem:poisson_domination}]
  Using $\E_\phi[g(X)] = \phi$ along with the bounds~\eqref{eq:g_cond} we obtain $\Gamma_- \rho \leq \phi \leq \Gamma_+ \rho$.
  Set $\lambda = (\Gamma_+/\Gamma_-) \rho = \Gamma \rho$ and write $\nu
  = \poisson(\lambda)$
  for the Poisson distribution with mean $\lambda$. Since $g(k) \ge \Gamma_- k$, the consecutive ratios satisfy
  \begin{equation}
    \frac{\mu_\rho(k + 1)}{\mu_\rho(k)} = \frac{\phi}{g(k + 1)}
    \le \frac{\Gamma_+ \rho}{(k + 1)\Gamma_-} = \frac{\lambda}{k + 1}
    = \frac{\nu(k + 1)}{\nu(k)},
    \qquad k \ge 0.
  \end{equation}
  It follows that the likelihood ratio $r(k) = \mu_\rho(k)/\nu(k)$ is non-increasing in $k$, which
  means $\mu_\rho$ is dominated by $\nu$ in the likelihood ratio order. In particular, each $X_i$
  under $\mu_\rho$ is stochastically dominated by a $\poisson(\lambda)$ random variable (see, e.g.,~\cite[Theorem 1.C.1]{ShakedShanthikumar07}).
  Since the variables $X_1, \dots, X_n$ are independent, their sum is dominated by a sum of $n$ independent $\poisson(\lambda)$ variables, which has distribution $\poisson(\lambda n) = \poisson(\Gamma \rho n)$.
\end{proof}

\begin{proof}[Proof of \cref{lem:crossing_bound}]
  The bound in~\eqref{eq:displacement_bound}, follows by conditioning on $Y = n$. If $n \le 2\Gamma_+ T$, one combines the reflection principle applied to the projections together with Hoeffding's inequality. If $n > 2\Gamma_+ T$, the bound follows from the Poisson tail, which decays as $\exp\braces{-n \log(n / (e\Gamma_+ T))}$. In both regimes, the tail is upper bounded by $\ubc{C:displacement} \exp\braces{-\uc{c:displacement} s \log(1+s/T)}$ for appropriate constants.

  For the crossing bound in~\eqref{eq:crossing_bound}, let $V_s = \{x \in \Z^d : \normI{x} \geq s\}$ and notice that, in order for a particle initially at $x$ to reach the origin by time $T$, its displacement must be by at least $\norm{x} \ge s$. Conditioning on $\eta_0$, a union bound over the $\eta_0(x)$ particles at site $x$ and summing over $x \in V_s$ gives
  \begin{equation}
    \P^{\eta_0} \big(\text{some particle from $V_s$ reaches $0$} \big)
    \le \sum_{x \in V_s} \eta_0(x) \ubc{C:displacement} e^{-\uc{c:displacement} \norm{x} \log(1+\norm{x}/T)}.
  \end{equation}
  Taking expectations under $\mu_\rho$, using $\E_\rho[ \eta_0(x) ] = \rho$ for every $x$, and writing $V_s$ as the union of the sets $\{x \in \Z^d : \abs{x_i} \geq s\}$ for $i = 1, \ldots, d$, we obtain
  \begin{equation}
    \begin{aligned}
    \P_\rho &\big( \text{some particle from $V_s$ reaches $0$} \big) \\ 
    &\qquad\le 2 d \rho \ubc{C:displacement} \sum_{k = s}^{\infty} e^{-c k \log(1+k/T)} \parens*{\sum_{x \in \Z} e^{-c \abs{x} \log(1+\abs{x}/T)}}^{d - 1}.
    \end{aligned}
  \end{equation}
  Using the super-additivity of $u \mapsto u \log(1+u/T)$, the first sum can be bounded by the following series:
  \begin{equation}
    \begin{split}
      \sum_{k = s}^{\infty} e^{-c k \log(1+k/T)}
      & \le e^{-c s \log(1+s/T)} \sum_{j = 0}^{\infty} e^{-c j \log(1+j/T)} \\
      & \le e^{-c s \log(1+s/T)} \parens*{1 + \int_0^\infty e^{-c u \log(1+u/T)} \dd u }.
    \end{split}
  \end{equation}
  To bound the integral, we split the domain at $u=T$ and use the lower bounds $\log(1+u/T) \ge u/(2T)$ for $u \le T$ and $\log(1+u/T) \ge \log 2$ for $u > T$:
  \begin{equation}
    \begin{split}
      \int_0^\infty e^{-c u \log(1+u/T)} \dd u 
      &\le \int_0^T e^{-\frac{c}{2} \frac{u^2}{T}} \dd u + \int_T^\infty e^{-c (\log 2) u} \dd u \\
      &\le \sqrt{T} \int_0^\infty e^{-c' v^2} \dd v + C e^{-c' T} \le C' \sqrt{T},
    \end{split}
  \end{equation}
  where we substituted $v = u/\sqrt{T}$ in the first integral, and used $T \ge 1$. The second sum can be bounded similarly by $C' \sqrt{T}$, yielding a final bound of
  \begin{equation}
    \P_\rho \big( \text{some particle from $V_s$ reaches $0$} \big) \le C'' \rho \sqrt{T} e^{-c s \log(1+s/T)} T^{\frac{d - 1}{2}}.
  \end{equation}
  Combining the factors and relabeling the constants yields \eqref{eq:crossing_bound}, concluding the proof.
\end{proof}
\begin{proof}[Proof of \cref{lem:meeting}]
  The probability that the process $W_t$ hits the origin by time $T$ is stochastically bounded below by the hitting probability of a continuous-time simple symmetric random walk jumping at a constant rate $2\Gamma_-$. The condition $R \le \sqrt{\ubc{C:meeting_distance} T}$ implies $T \ge \ubc{C:meeting_distance}^{-1} R^2$. 
  By standard continuous-time random walk estimates, for $\ubc{C:meeting_distance}$ sufficiently small (depending only on $\Gamma_-$ and $d$), the hitting probability is bounded below by a positive constant when $d=1$, by $C (\log R)^{-1}$ when $d=2$, and by $C R^{-(d-2)}$ when $d \ge 3$, where $C(d) > 0$ is a dimension-dependent constant (see, for example,~\cite[Proposition 2.1.2]{LawlerLimic10}).
  
  In dimensions $d=1$ and $d \ge 3$, these bounds directly yield $C R^{-\gamma_d}$. In the critical dimension $d=2$, since $(\log R)^{-1} \ge C' R^{-1}$ for a suitable constant $C'$ and all $R \ge 1$, we also obtain a lower bound of the form $C' R^{-1} = C' R^{-\gamma_2}$. Taking $\ubc{C:meeting}(d)$ to be the minimum of these constants yields the unified lower bound $\ubc{C:meeting} R^{-\gamma_d}$ across all dimensions.
\end{proof}

\begin{proof}[Proof of \cref{lem:occupation_bounds}]
  We decompose the time interval $[0, T]$ into $\lfloor T \rfloor + 1$ sub-intervals of the form
  $[n, n+1)$ and apply a union bound over each sub-interval. Let $p = e^{-\Gamma_+}$.

  For the supremum, let $\tau_n = \inf\{ t \in [n, n+1) : \eta_t(x) \ge M \}$. Since particles jump
  one at a time, if $\tau_n < \infty$, the site $x$ contains exactly $\ceil{M}$ particles at time
  $\tau_n$. In the slot representation, each particle jump rate is uniformly bounded by $\Gamma_+$.
  Therefore, any particle present at time $\tau_n$ avoids jumping and remains at $x$ until
  time $n + 1$ with probability at least $e^{-\Gamma_+ (n + 1 - \tau_n)} \ge p$, independently of
  the other particles. Ignoring any newly arriving particles, the number of particles present at
  time $\tau_n$ that stay at $x$ up to time $n+1$ stochastically dominates a binomial random
  variable $B \sim \bin(\ceil{M}, p)$. Since $\eta_{n+1}(x)$ is bounded from below by this number of
  surviving particles, $\eta_{n+1}(x)$ also dominates $B$ conditionally on $\tau_n < \infty$.

  For any threshold $K \le M p / 2$, the Chernoff bound for the binomial distribution gives
  \begin{equation}
    \P\big( B \le K \big) \le \exp\braces*{ - \frac{ (M p - K)^2}{2 M p} } \le \exp\braces*{ - \frac{ (M p / 2)^2}{2 M p} } = e^{-p M / 8}.
  \end{equation}
  It follows that
  \begin{equation}
    \begin{split}
      \P\big(\tau_n < \infty\big)
      & \le \P\big( \eta_{n+1}(x) \ge K \big) + \P\big( \eta_{n+1}(x) < K \mid \tau_n < \infty \big) \\
      & \le \P_\rho\big( \eta_0(x) \ge K \big) + e^{-p M / 8},
    \end{split}
  \end{equation}
  where we used the stationarity of $\eta_t$ under $\P_\rho$. Summing this over $n = 0, \ldots,
  \lfloor T \rfloor$ yields the bound for the supremum.

  For the infimum, suppose there is some time $\tau_n \in [n, n+1)$ where $\eta_{\tau_n}(x) \le m$.
  Using the slot configuration, we track the trajectories of the $\eta_n(x)$ particles that
  were present at $x$ at time $n$. The number of these particles that do not leave site $x$ up to time $t$ is non-increasing in time. Thus, the number of such particles at time $n+1$ cannot
  exceed the total number of particles $\eta_{\tau_n}(x) \le m$. On the other hand, each of the
  $\eta_n(x)$ particles stays at $x$ until time $n+1$ with probability at least $p$, independently
  of the others. Hence, the conditional distribution of the number of survivors at $n+1$
  stochastically dominates a binomial random variable $\tilde B \sim \bin(\eta_n(x), p)$. For any threshold $K \ge 2m / p$, if $\eta_n(x) \ge K$, the probability of observing at most $m$ survivors
  is bounded by
  \begin{equation}
    \begin{aligned}
    \P\big( \tilde B \le m \mid \eta_n(x) \ge K \big) &\le \P\big( \bin(\ceil{K}, p) \le m \big) \\
    &\le \exp\braces*{ - \frac{(K p - m)^2}{2 K p} }\\
    &\le \exp\braces*{ - \frac{(K p / 2)^2}{2 K p} } = e^{-K p / 8}.
    \end{aligned}
  \end{equation}
  Consequently, we can bound the probability of the infimum event by
  \begin{equation}
    \begin{split}
      \P\Big( \inf_{t \in [n, n+1)} \eta_t(x) \le m \Big)
      & \le \P\big( \eta_n(x) < K \big) + \P\big( \tilde B \le m \mid \eta_n(x) \ge K \big) \\
      & \le \P_\rho\big( \eta_0(x) < K \big) + e^{-K p / 8}.
    \end{split}
  \end{equation}
  Taking a union bound over the $\lfloor T \rfloor + 1$ intervals concludes the proof for the infimum. Setting $\uc{c:occupation} = p/8$ ensures both statements hold.
\end{proof}

\section{Exponential families and tail bounds}
\label{sec:exponential_families}

Recall from \cref{sec:zr_invariant_measures} the measure $\nu_\phi$ with independent marginals given by~\eqref{eq:nu_phi} and the step-size condition~\eqref{eq:g_cond} on $g$. 
It is convenient to reparameterize this marginal as a natural exponential family with parameter $\eta = \log \phi$. Under this parametrization, the log-partition function $A(\eta) = \log Z(e^\eta)$ satisfies $A'(\eta) = \E_\eta[X] = \rho$ and $A''(\eta) = \var_\eta(X)$, and its cumulant generating function (CGF) is precisely $\psi(\lambda) = \log \E_\eta[e^{\lambda X}] = A(\eta+\lambda) - A(\eta)$.

For a sequence of $n$ independent copies $X_1, \dots, X_n$ of the marginal distribution, the sample mean $\bar{X} = \frac{1}{n}\sum_{i=1}^n X_i$ satisfies the standard Chernoff bounds:
\begin{equation}\label{eq:chernoff}
  \P(\bar{X} \ge x) \le e^{-n I(x)} \quad (x > \rho), \qquad \P(\bar{X} \le x) \le e^{-n I(x)} \quad (x < \rho),
\end{equation}
where $I(x) = \sup_{\lambda} [\lambda x - (A(\eta+\lambda) - A(\eta))]$ and the supremum is taken over the positive or negative reals, respectively. Because the exact functional form of $g(k)$ is unknown, we cannot compute $A(\eta)$ directly. Instead, we establish a differential inequality that bounds the variance $A''(\eta)$ in terms of the mean $A'(\eta)$.

\begin{lemma}[Variance bound]\label{lem:variance_bound}
  For any valid parameter $\eta$, the variance is bounded by $A''(\eta) \le \Gamma A'(\eta)$, where $\Gamma = \frac{\Gamma_+}{\Gamma_-}$.
\end{lemma}
\begin{proof}
Natural exponential families satisfy $\cov_\eta(X, f(X)) = \frac{d}{d\eta} \mathbb{E}_\eta[f(X)]$ for any $f$. Since $\E_\eta[g(X)] = e^\eta$ we have $\cov(X, g(X)) = e^\eta$. Expressing the covariance in terms of independent copies $X_1$ and  $X_2$ of $X$, we obtain
\begin{equation}
  e^\eta = \cov(X, g(X)) = \frac{1}{2}\E_\eta[(X_1 - X_2)(g(X_1) - g(X_2))].
\end{equation}
Condition~\eqref{eq:g_cond} implies $(x-y)(g(x)-g(y)) \ge \Gamma_- (x-y)^2$, yielding $e^\eta \ge \Gamma_- \var_\eta(X) = \Gamma_- A''(\eta)$. Conversely, $g(k) \le \Gamma_+ k$ implies $e^\eta \le \Gamma_+ \E_\eta[X] = \Gamma_+ A'(\eta)$. Combining these yields the result.\end{proof}

\begin{lemma}[Global CGF bound]\label{lem:cgf_bound}
  For any $\lambda \in \R$, $A(\eta+\lambda) - A(\eta) \le \frac{\rho}{\Gamma} \left( e^{\Gamma\lambda} - 1 \right)$.
\end{lemma}
\begin{proof}
  Integrating the differential inequality $\frac{d}{ds} \log A'(\eta+s) = \frac{A''(\eta+s)}{A'(\eta+s)} \le \Gamma$ yields $A'(\eta+s) \le \rho e^{\Gamma s}$ for $s > 0$ and $A'(\eta+s) \ge \rho e^{\Gamma s}$ for $s < 0$. Integrating $A'(\eta+s)$ again over the interval between $0$ and $\lambda$ establishes the upper bound for all $\lambda \in \R$.
\end{proof}

Substituting this CGF bound into $I(x)$ and optimizing over $\lambda$ yields $\lambda^* = \frac{1}{\Gamma} \log(x/\rho)$, which provides the explicit rate function lower bound:
\begin{equation}\label{eq:rate-function}
  I(x) \ge \frac{1}{\Gamma} \left( x \log\left(\frac{x}{\rho}\right) - x + \rho \right) = \frac{1}{\Gamma} d(x, \rho),
\end{equation}
where $d(x, \rho)$ is the rate function of a Poisson random variable with parameter $\rho$. 

\subsection{Large deviation tails}

Substituting relative deviations $x = \rho(1 \pm \delta)$ and absolute deviations $x = \rho \pm \epsilon$ (where $\delta = \epsilon/\rho$) into \cref{eq:rate-function}, and applying the standard rational lower bounds for the Poisson rate function $d(\rho(1 + \delta), \rho) \geq \rho\frac{\delta^2}{2+2\delta/3}$ and $d(\rho(1 - \delta), \rho) \geq \rho\frac{\delta^2}{2}$, we obtain the following tail bounds.

\begin{corollary}[Upper tail bounds]\label{cor:upper_bounds}
  For any relative deviation $\delta > 0$ and corresponding absolute deviation $\varepsilon = \rho\delta > 0$:
  \begin{equation}
  \begin{gathered}
    \P(\bar{X} \ge \rho(1+\delta)) \le \exp\braces*{ -n \frac{\rho}{\Gamma} \frac{\delta^2}{2 + 2\delta/3} }, \\
    \P(\bar{X} \ge \rho + \varepsilon) \le \exp\braces*{ -\frac{n}{\Gamma} \frac{\varepsilon^2}{2 \rho + 2\varepsilon/3} }.
  \end{gathered}
  \end{equation}
  If restricted to $\delta \in (0, 1]$ (or $\varepsilon \le \rho$), the bounds simplify to the sub-Gaussian forms $\exp\braces{ -n \frac{\rho \delta^2}{3\Gamma} }$ and $\exp\braces{ -n \frac{\varepsilon^2}{3 \Gamma \rho} }$.
\end{corollary}

\begin{corollary}[Lower tail bounds]\label{cor:lower_bounds}
  For any relative deviation $\delta \in (0, 1)$ and corresponding absolute deviation $\varepsilon = \rho\delta \in (0, \rho)$:
  \begin{equation}
    \P(\bar{X} \le \rho(1-\delta)) \le \exp\braces*{ -n \frac{\rho \delta^2}{2\Gamma}},
    \qquad 
    \P(\bar{X} \le \rho - \varepsilon) \le \exp\braces*{ -n \frac{\varepsilon^2}{2 \Gamma \rho} }.
  \end{equation}
\end{corollary}

\bibliographystyle{plain}
\bibliography{mybib}

\end{document}